\documentclass[10pt, a4paper, reqno, english]{amsart}

\usepackage{hyperref}
\usepackage{amssymb,mathrsfs}
\usepackage[all]{xy}
\usepackage{enumitem}
\usepackage{tikz}

\newcommand\ZZ{\mathbb{Z}}
\newcommand{\Z}{\mathbb{Z}}
\newcommand\NN{\mathbb{N}}
\newcommand\CC{\mathbb{C}}
\newcommand\RR{\mathbb{R}}
\newcommand\QQ{\mathbb{Q}}
\newcommand\Qbar{{\overline{\QQ}}}
\newcommand\Fbar{{\overline{\FF}}}
\newcommand\FF{\mathbb{F}}
\newcommand\AAA{\mathbf{A}}

\newcommand\PP{\mathbb{P}} 
\newcommand{\A}{\mathbb{A}}
\newcommand\GG{\mathbb{G}}

\newcommand{\tS}{{\widetilde S}}
\newcommand{\SZ}{{\mathcal S}}
\newcommand{\tSZ}{{\widetilde{\mathcal S}}}
\newcommand{\tT}{Y}
\newcommand{\TZ}{{\mathcal Y}}
\newcommand{\cT}{{}_{\classtuple}\TZ}
\newcommand{\cpi}{{}_{\classtuple}\pi}
\newcommand{\lines}{S\smallsetminus U}

\DeclareMathOperator{\Pic}{Pic}
\newcommand{\Ss}{{S_{\text{reg}}}}

\newcommand{\Aone}{{\mathbf A}_1}
\newcommand{\Atwo}{{\mathbf A}_2}
\newcommand{\Athree}{{\mathbf A}_3}
\newcommand{\Dfive}{{\mathbf D}_5}

\newcommand{\St}{Z}
\newcommand{\Gt}{G}
\newcommand{\Xt}{X}
\newcommand{\Yt}{Y}
\newcommand{\GtX}{\Gt_{\Xt}}
\newcommand{\G}{\mathbb{G}}
\newcommand{\Zt}{W}
\newcommand{\XtOk}{X}
\newcommand{\YtOk}{Y}
\newcommand{\Ytk}{\YtOk_K}
\newcommand{\Xtk}{\XtOk_K}
\newcommand{\conorm}{\mathcal{C}}

\newcommand{\aAN}{{}_{\ua}\A^N}
\newcommand{\aYtOk}{{}_{\ua}\YtOk}
\newcommand{\cYtOk}{{}_{\classtuple}\YtOk}

\newcommand{\api}{{}_{\ua}\pi}
\newcommand{\classtuplepi}{{}_{\classtuple}\pi}

\newcommand{\zt}{z}
\newcommand{\ftOk}{f}
\newcommand{\gtOk}{g}
\newcommand{\mA}{A}

\newcommand{\mt}{m}
\newcommand{\st}{s}

\newcommand{\integersk}{\mathcal{O}_K}
\newcommand{\Ok}{\mathcal{O}_{K}}
\newcommand{\units}{{\integersk^\times}}
\newcommand{\idealsk}{\mathcal{I}_K}
\newcommand{\principalidealsk}{\mathcal{P}_K}
\newcommand{\id}[1]{\mathfrak{#1}}
\newcommand{\ppp}{\id p}
\newcommand{\p}{\mathfrak{p}}
\newcommand{\aaa}{\id a}
\newcommand{\aid}{\mathfrak{a}}
\newcommand{\ua}{\underline{\mathfrak{a}}}
\newcommand{\bbb}{\id b}
\newcommand{\places}{{\Omega_K}}
\newcommand{\archplaces}{{\Omega_\infty}}
\newcommand{\nonarchplaces}{{\Omega_0}}
\newcommand{\abs}[1]{\left|#1\right|}
\newcommand{\absv}[1]{\left|#1\right|_v}
\newcommand{\normv}[1]{\left\lVert #1 \right\rVert_v}
\DeclareMathOperator{\N}{\mathfrak{N}}
\newcommand{\OO}{\mathcal{O}}

\newcommand{\pigenerator}{\rho}
\newcommand{\cl}{Cl_K}

\newcommand{\classrepsyst}{\mathcal{C}}
\newcommand{\classrep}{\id{c}}
\newcommand{\classtuple}{{\underline{\id{c}}}}

\newcommand{\us}{\underline s}
\newcommand{\sg}{s}
\newcommand{\usg}{\underline{\sg}}

\newcommand{\coord}{a}
\newcommand{\coordtuple}{\underline{a}}
\newcommand{\coordfracideal}{\mathcal{O}}
\newcommand{\coordfracidealproduct}{{\mathcal{O}_*'}}
\newcommand{\coordideal}{\id{a}}
\newcommand{\coordidealtuple}{\underline{\id{a}}}

\newcommand{\FDtot}{\mathcal{F}}
\newcommand{\FDxi}{\mathcal{F}_0}
\newcommand{\FDei}{\mathcal{F}_1}

\newcommand{\divisorideal}{\id{d}}
\newcommand{\divisoridealtuple}{\underline{\id{d}}}
\newcommand{\latticefracideal}{\id{b}}
\newcommand{\countinggroup}{\mathcal{G}}
\newcommand{\countingbox}{\mathcal{B}}
\newcommand{\cutoff}{\mathcal{R}}

\newcommand\dd{\,\mathrm{d}}
\newcommand{\ex}[1]{*+<5pt>[o][F]{E_{#1}}}
\newcommand{\li}[1]{*+<3pt>[F]{E_{#1}}}
\newcommand{\tN}{\tilde{N}}
\newcommand{\tM}{\tilde{M}}
\newcommand\xx{\mathbf{x}}
\newcommand\WHERE{\,\Bigg|\,}
\newcommand{\bomega}{\boldsymbol{\omega}}
\newcommand{\dual}{\vee}
\DeclareMathOperator{\vol}{vol}
\DeclareMathOperator{\Gal}{Gal}
\DeclareMathOperator{\cox}{Cox}
\DeclareMathOperator{\Hom}{Hom}
\DeclareMathOperator{\pic}{Pic}
\DeclareMathOperator{\spec}{Spec}
\DeclareMathOperator{\proj}{Proj}
\DeclareMathOperator{\sgn}{sgn}

\newtheorem{theorem}{Theorem}
\newtheorem{lemma}[theorem]{Lemma}
\newtheorem{prop}[theorem]{Proposition}
\theoremstyle{definition}
\newtheorem{defin}[theorem]{Definition}
\newtheorem{rem}[theorem]{Remark}

\newtheorem*{ack}{Acknowledgements}
\newtheorem*{remark}{Remark}
\newtheorem{Remark}[theorem]{Remark}
\newtheorem{construction}[theorem]{Construction}

\numberwithin{theorem}{section}
\numberwithin{equation}{section}

\begin{document}

\setcounter{tocdepth}{1}

\title[O-minimality on twisted universal torsors and Manin's conjecture]{o-minimality on twisted universal torsors and Manin's  conjecture over number fields}

\author{Christopher Frei} \email{frei@math.tugraz.at}

\address{Technische Universit\"at Graz, Institut f\"ur Analysis und Computational Number Theory (Math A), Steyrergasse 30/II, A-8010 Graz, Austria}

\author{Marta Pieropan} \email{pieropan@math.uni-hannover.de}

\address{Institut f\"ur Algebra, Zahlentheorie und Diskrete
  Mathematik, Leibniz Universit\"at Hannover,
  Welfengarten 1, 30167 Hannover, Germany}

\date{July 17, 2015}

\begin{abstract}
  Manin's conjecture predicts the distribution of rational points on
  Fano varieties. Using explicit parameterizations of rational points
  by integral points on universal torsors and lattice-point-counting
  techniques, it was proved for several specific varieties over $\QQ$,
  in particular del Pezzo surfaces.  We show how this method can be
  implemented over arbitrary number fields, by proving Manin's
  conjecture for a singular quartic del Pezzo surface of type
  $\Athree+\Aone$. The parameterization step is treated in high
  generality with the help of twisted integral
  models of universal torsors. To make the counting step feasible over
  arbitrary number fields, we deviate from the usual approach over
  $\QQ$ by placing higher emphasis on the geometry of numbers in the
  framework of o-minimal structures.
%
%
\end{abstract}


\maketitle
\vspace{-1cm}
\tableofcontents

\section{Introduction}
Let $K$ be a number field and $S$ the anticanonically embedded del
Pezzo surface of degree $4$ and type $\Athree+\Aone$ given in
$\PP^4_K$ by the equations
\begin{equation}\label{eq:def_S}
  x_0x_3-x_2x_4=x_0x_1+x_1x_3+x_2^2 = 0.
\end{equation}
Let $U$ be the complement of the lines in $S$, and let $H$ be the
anticanonical height on $S(K)$ induced by the Weil height on
$\PP^4(K)$,
\begin{equation*}
  H(x_0 : \cdots : x_4) := \prod_{v \in \places}\max\{|x_0|_v,
  \ldots, |x_4|_v\},
\end{equation*}
where $\places$ is the set of places of $K$ and the normalized
absolute values $\absv{\ \cdot\ }$ are given as follows: let $w$ be
the place of $\QQ$ below $v$ and $K_v$ (resp.~$\QQ_w$) the completion
of $K$ at $v$ (resp.~of $\QQ$ at $w$). Then $\absv{\ \cdot\ } :=
\abs{N_{K_v|\QQ_w}(\ \cdot\ )}_w$, where $\abs{\ \cdot\ }_w$ is the usual
real or $p$-adic absolute value on $\QQ_w$. We investigate the
counting function
\begin{equation*}
  N_{U, H}(B) := |\{\xx \in U(K)\mid H(\xx)\leq B\}|.
\end{equation*}
Generalized versions \cite{MR1679843, MR2019019} of Manin's conjecture
\cite{MR89m:11060} predict an asymptotic formula
\begin{equation*}
  N_{U, H}(B) = c_{S, H} B(\log B)^{5}(1 + o(1)),\quad\text{ as }B\to\infty,
\end{equation*}
with a positive constant $c_{S, H}$, which has been conjecturally
interpreted in \cite{MR1340296, MR1679843, MR2019019}. Our first main
result is a proof of Manin's conjecture for $S$:
\begin{theorem}\label{maintheorem}
  Let $K$ be a number field of degree $d$, let $S$ be given in
  $\PP^4_K$ by \eqref{eq:def_S}, let $U$ be the complement of the
  lines in $S$, and let $\epsilon > 0$. As $B \to \infty$,
  \begin{equation*}
    N_{U,H}(B) = c_{S,H} B  (\log B)^5 + O(B(\log B)^{5-1/d+\epsilon}),
  \end{equation*}
  with an explicit $c_{S,H} > 0$. This formula agrees with Peyre's
  refined version of Manin's conjecture \cite[Formule empirique
  5.1]{MR2019019}. The implicit constant in the error term depends on
  $K$ and $\epsilon$.
\end{theorem}
We describe the constant $c_{S,H}$ explicitly later in this
section. The special cases of Theorem \ref{maintheorem} where $K=\QQ$
or $K$ is imaginary quadratic were proved in \cite{MR2520770,
  arXiv:1304.3352}. A version of Manin's conjecture over arbitrary global function fields was proved for our surface $S$ in \cite{MR3124933}.

Manin's conjecture is known in some general
cases. For complete intersections of large dimension compared to their
degree, it follows from an application of the Hardy-Littlewood circle
method (cf. \cite{MR1340296, arXiv:1210.1792}). Moreover, it has been
proved for certain classes of Fano varieties with additional structure
coming from actions of algebraic groups, using Langlands' work on
Eisenstein series \cite{MR89m:11060} or harmonic analysis on adelic
points (for example, for toric varieties \cite{MR1620682} and
equivariant compactifications of additive groups \cite{MR1906155}).

Other known cases of Manin's conjecture concern specific varieties of
low dimension. Del Pezzo surfaces over $\QQ$ have received the most
attention: some milestones here are the first special cases of Manin's
conjecture for (singular or nonsingular) del Pezzo surfaces of degrees
$5$ \cite{MR1909606}, $4$ \cite{MR2373960}, $3$ \cite{MR2332351}, and
$2$ \cite{MR3100953} that are not covered by \cite{MR1620682} or
\cite{MR1906155}. The method behind these results and many further
proofs of Manin's conjecture for specific varieties over $\QQ$ is by
now classical. It is usually referred to as the \emph{universal torsor
  method}.

A major drawback of this method is that almost all of its successful
applications are restricted to varieties over $\QQ$. Recently,
Derenthal and the first-named author started a project with the aim to
generalize the universal torsor method to number fields beyond
$\QQ$. So far, they were able to adapt the basic framework to
imaginary quadratic fields \cite{arXiv:1302.6151}, and to apply it to
some singular del Pezzo surfaces of degrees $4$ \cite{arXiv:1304.3352}
and $3$ \cite{E6} over imaginary quadratic fields. To our best
knowledge, the only published proofs of Manin's conjecture for
varieties over arbitrary number fields that can be interpreted as
applications of the universal torsor method concern projective spaces
$\PP^n_K$ \cite{MR557080} and a specific toric variety
\cite{MR3107569}, which are also covered by \cite{MR1620682}.

Theorem \ref{maintheorem} is a first step to overcome this
restriction. It is based on the universal torsor approach and is the
first proof of Manin's conjecture over arbitrary number fields for a
variety that is not included in the general results mentioned above
(see \cite{MR2753646}). One should note that the surface $S$ is an equivariant compactification of a semidirect product $\GG_a \rtimes \GG_m$, so recent techniques of Tanimoto and Tschinkel \cite{MR2858922} using harmonic analysis could also apply. So far, this was worked out only over $\QQ$.

\subsection{The universal torsor method}
Universal torsors were introduced and studied by Colliot-Th\'el\`ene
and Sansuc \cite{MR0414556,MR89f:11082} to investigate arithmetic
properties such as the Hasse principle and weak
approximation. Salberger \cite{MR1679841} was the first to apply them
to Manin's conjecture (see also \cite{MR1679842}). After Salberger's
pioneering work, the universal torsor method became a prevalent tool
to prove special cases of Manin's conjecture over $\QQ$.

A typical application of the universal torsor method to a specific del
Pezzo surface $S$ consists essentially of two parts:
\begin{enumerate}
\item[(a)] Parameterizing the rational points on an open subset $U$ by
  integral points on a universal torsor over a minimal
  desingularization $\tS \to S$, subject to certain coprimality
  conditions, and lifting the height function to these points.
\item[(b)] Counting these integral points of bounded height,
  essentially replacing sums by integrals and estimating the
  difference.
\end{enumerate}
A framework covering these parts in some generality was developed over
$\QQ$ in \cite{MR2520770} and generalized in \cite{arXiv:1302.6151} to
imaginary quadratic fields.

\subsection{Parameterization}
The minimal desingularization $\tS$ of $S$ is a smooth projective
variety over $K$. For such a variety $\tS$ and a torsor $\tT$ over
$\tS$ under an algebraic $K$-group $\Gt$, a classical result of
Colliot-Th\'el\`ene and Sansuc \cite{MR89f:11082} shows that there is
a partition
\begin{equation*}
  \tS(K) = \bigsqcup_{[\sigma] \in H^1(K, \Gt)}{}_\sigma\pi({}_\sigma \tT(K)),
\end{equation*}
for twists ${}_\sigma\pi : {}_\sigma \tT \to \tS$ of $\tT$. The finest
partitions of this kind are achieved if $\tT$ is a universal
torsor. For quantitative problems such as Manin's conjecture, it is
desirable to obtain a parameterization of $\tS(K)$ by points with
integral coordinates, which allows us to apply lattice-point-counting
techniques. Such a parameterization was obtained by Salberger
\cite{MR1679841} for proper, smooth, split toric varieties $X$ over
$\QQ$ with globally generated anticanonical sheaf. In this case, the
partition induced by a model $\pi : \TZ \to \mathcal{X}$ of a
universal torsor $\tT\to X$ is trivial:
\begin{equation}\label{eq:parameterization_integers}
  X(\QQ) = \pi(\TZ(\ZZ)).
\end{equation}
Here, the fibers of $\pi$ are just the orbits under the action of
$\G_m^r(\ZZ) \cong (\ZZ^\times)^r$, where $r$ is the rank of the
Picard group of $X$. Hence, we obtain a $(2^r:1)$-parameterization
of $X(\QQ)$ by integral points, which reduces Manin's conjecture to
a lattice-point-counting problem. In almost all applications of the
universal torsor method to special cases of Manin's conjecture over
$\QQ$, a parameterization of the form
\eqref{eq:parameterization_integers} is constructed by elementary
methods that essentially consist of removing greatest common divisors
between existing coordinates by introducing new ones. With some
exceptions (e.g.~\cite{MR1909606,arXiv:1205.0190}), the connection to
universal torsors is usually not made precise.

The first case where it was necessary to consider a partition by
integral points on more than one torsor was encountered by de la
Bret\`eche, Browning and Peyre in \cite{MR2874644}. In
\cite{arXiv:1302.6151}, Derenthal and the first-named author observed
that similar disjoint unions (over all $r$-tuples of ideal classes)
always appear when considering split del Pezzo surfaces over number
fields $K$ of class number $h_K > 1$, even if a trivial partition with
just one universal torsor exists over $\QQ$. They interpreted this
phenomenon as points on one universal torsor, but with coordinates in
different ideal classes.

We provide a more conceptual explanation in terms of $\Ok$-points on twisted torsors over the ring of integers $\Ok$. This explanation also gives entirely explicit descriptions, which can be used to apply lattice-point-counting techniques. For split toric varieties, a similar description was found by Robbiani \cite{MR1650339}, based on ideas of Salberger. In the function field case, the parameterization was treated in high generality by Bourqui \cite{MR2573192}.

The basic idea is as follows. Let $\TZ$ be an $\Ok$-model of a
universal torsor $\tT$ over $\tS$, such that $\TZ$ is a torsor over a
proper model $\tSZ$ of $\tS$ under a split torus $\G_{m,\tSZ}^r$. We
apply the general theory of torsors and a properness argument to
obtain a partition
\begin{equation*}
  \tS(K) =  \bigsqcup_{[\sigma] \in H_{\text{\it{\'et}}}^1(\spec(\Ok), \G_{m,\Ok}^r)}{}_\sigma\pi({}_\sigma \TZ(\Ok)).
\end{equation*}
The identification $H_{\text{\it{\'et}}}^1(\spec(\Ok), \G_{m,\Ok}^r)
\cong \cl^r$ explains the disjoint union over $r$-tuples of ideal
classes appearing in the parameterization in \cite{arXiv:1302.6151}.
Under some additional technical conditions, we give an explicit
construction of the twists ${}_\sigma\pi:{}_\sigma \TZ \to \tSZ$ in
terms of fractional ideals of $\Ok$ representing the classes
corresponding to $\sigma$. This is worked out in a general context in
Section \ref{sec:param_twisted_torsors} and summarized in Theorem
\ref{thm:param}.

Explicit descriptions of universal torsors over minimal
desingularizations $\tS$ of singular del Pezzo surfaces $S$ over
$\Qbar$ can be obtained from the descriptions of their Cox rings in
\cite{math.AG/0604194}. In Section \ref{sec:universal_torsors}, we
show how to construct from this data a model $\tSZ$ of $\tS$ and a
model $\TZ$ of the universal torsor over $\Qbar$. In Theorem
\ref{universaltorsor}, we give conditions under which $\TZ$ is a
(universal) torsor over $\tSZ$ as above.

In Section \ref{sec:passage}, we present in detail an application of
the results from Sections \ref{sec:param_twisted_torsors} and \ref{sec:universal_torsors} to the quartic del Pezzo surface given by
(\ref{eq:def_S}) and obtain a parameterization of $U(K)$, where $U$ is
the open subset in Theorem \ref{maintheorem}. As summarized in Remark \ref{rem:other_cases}, analogous arguments apply to all other singular del Pezzo surfaces whose universal torsors are hypersurfaces, classified in \cite{math.AG/0604194}, allowing us to obtain in each case a good model of the universal torsor and a parameterization.

\subsection{Counting integral points}
Using the partition described above, we reduce the problem of counting
rational points on the open subset $U$ to counting $\Ok$-points in the
preimages of $U(K)$ under ${}_\sigma\pi$, modulo the action of
$\G_{m,\tSZ}^r(\Ok)$. This action is equivalent to an action of
$(\units)^r$, which is harmless when the unit group $\units$ is finite
(i.e., if $K=\QQ$ or $K$ is imaginary quadratic). If $\units$ is
infinite, however, one needs to count integral points in a fundamental
domain for this action. The difficulties arising in the treatment of
such counting problems are the main reason why the universal torsor
method was so far restricted to $\QQ$ and imaginary quadratic fields.

To deal with these problems in the case of our specific $S$ given by
\eqref{eq:def_S}, we introduce a major deviation from the usual
strategy in part (b). Instead of summing over the coordinates on the
twisted torsors one-by-one, we start by considering three coordinates
at the same time. The motivation for this departure comes from the
specific structure of the action of $\G_{m,\tSZ}^r(\Ok)$. This
structure is reflected in the shape of our fundamental domain, which
we construct in Section \ref{sec:fundamental_domain}.

The usual embedding $K\to K\otimes_\QQ \RR \cong \RR^d$ transforms
this first summation to the problem of counting lattice points in
certain  subsets of $\RR^{3d}$, depending on the remaining
coordinates, with an error term that can be estimated uniformly with
respect to the remaining coordinates, see Section
\ref{sec:lattice_point_problem}. These subsets are a priori unbounded,
and we need to remove cusps coming from small conjugates of certain
coordinates in Section \ref{sec:small_conjugates}.

Even after the removal of the cusps and the exploitation of certain
symmetries in Section \ref{sec:linear_transformations}, our sets are
of the ``long and thin'' kind, which makes them resistant to counting
arguments that depend on Lipschitz-parameterizations of the boundary,
such as \cite[Theorem VI.2.2]{MR1282723} or \cite[Lemma
2]{MR2247898}. In principle, Davenport's classical counting theorem
\cite{MR0043821} would apply, but its error term depends on certain
regularity properties which are hard to control uniformly in
general. In typical applications of Davenport's theorem, the sets
under consideration are semialgebraic, a condition that is not
satisfied in our case due to the restriction to a fundamental
domain.

A natural generalization of semialgebraic sets is given by the
model-theoretic framework of o-minimal structures. The celebrated
upper bound by Pila and Wilkie \cite{MR2228464} for the number of
rational points of bounded height in the transcendental part of sets
definable in an o-minimal structure has led to many applications in
Diophantine geometry. We apply o-minimality in a new way.

In Section \ref{sec:definability}, we show that the sets whose lattice
points are to be counted form a definable family in Wilkie's
\cite{MR1398816} o-minimal structure $\langle \RR;
<,+,\cdot,-,\exp\rangle$. This allows us to apply a recent adaptation
of Davenport's counting principle to definable sets by Barroero and
Widmer \cite{arXiv:1210.5943}.

The error term in Barroero and Widmer's theorem is given, as in
Davenport's result, in terms of the volumes of the orthogonal
projections of our set to all proper coordinate subspaces of
$\RR^{3d}$. In Section \ref{sec:volumes_of_projections}, we establish
summable upper bounds for these volumes, which allow us to perform
the first summation over three coordinates in Section
\ref{sec:first_summation}. The proof of Theorem \ref{maintheorem}
is completed in Section \ref{sec:completion}.

\subsection{Description of the leading constant}\label{subsec:constant}
Let us briefly describe the leading constant $c_{S,H}$ in Theorem
\ref{maintheorem}. Let $r_1$ (resp. $r_2$) be the number of real
(resp. complex) places of $K$, and let $\Delta_K$, $R_K$, $h_K$,
$\mu_K$ denote the discriminant, regulator, class number, and group of
roots of unity of $K$.  For any non-archimedean place $v$ of $K$,
corresponding to a prime ideal $\ppp$ of $\Ok$ of absolute norm
$\N\ppp$, we define
\begin{equation*}
  \omega_v(\tS) := \left(1-\frac{1}{\N\ppp}\right)^6\left(1+\frac{6}{\N\ppp}+\frac{1}{\N\ppp^2}\right). 
\end{equation*}
For any archimedean place $v$ of $K$ and $(x_0, x_1, x_2) \in K_v^3$,
we write
\begin{equation}\label{eq:def_N_v}
  N_v(x_0,x_1,x_2) := \max\left\{
    \begin{aligned}
      &\absv{x_0x_1x_2}, \absv{x_1^3}, \absv{x_1^2x_2},\\
      &\absv{x_1x_2(x_0+x_2)}, \absv{x_0x_2(x_0+x_2)}
    \end{aligned}\right\},
\end{equation}
and
\begin{equation*}
  \omega_v(\tS) :=
  \begin{cases}
    \frac{3}{2}\cdot \int_{N_v(x_0,x_1,x_2)\leq 1}\dd x_0\dd x_1 \dd x_2 &\text{ if $v$ is real,}\\
    \frac{12}{\pi}\cdot \int_{N_v(x_0,x_1,x_2) \leq 1}\dd x_0 \dd x_1
    \dd x_2&\text{ if $v$ is complex,}
  \end{cases}
\end{equation*}
where the integrals are taken with respect to the usual Lebesgue
measure on $K_v \in \{\RR, \CC\cong\RR^2\}$. Then the leading constant in
Theorem \ref{maintheorem} has the form
\begin{equation*}
  c_{S,H} =\frac{1}{8640}\cdot\left(\frac{2^{r_1}(2\pi)^{r_2}R_K h_K}{|\mu_K|}\right)^6\cdot\frac{1}{|\Delta_K|^4}\cdot\prod_{v \in \places}\omega_v(\tS).
\end{equation*}
In Section \ref{sec:constant}, we show that this constant is the one
from Peyre's empirical formula \cite[Formule empirique
5.1]{MR2019019}.

\subsection{More notation}
Unless explicitly stated otherwise, the symbol $K$ denotes a number
field.  Let $U_K$ be the subgroup of $\units$ generated by a fixed
system of fundamental units. Then $U_K$ is free abelian of rank
$q:=r_1+r_2-1$, and $\units \cong \mu_K \times U_K$. Let $\idealsk$
be the monoid of nonzero ideals of $\integersk$, let
$\principalidealsk$ be the group of nonzero principal fractional
ideals, and $\cl$ the class group of $\Ok$. The ideal class of a
nonzero fractional ideal $\aaa$ is denoted by $[\aaa]$.

For any $v \in \places$, we denote by $\sigma_v$ the embedding $K \to
K_v$, as well as its component-wise extensions $\sigma_v : K^n \to
K_v^n$ for $n \in \NN$. Moreover, $\sigma : K \to \prod_{v \in
  \archplaces}K_v$ denotes the embedding $\coord \mapsto
(\sigma_v(\coord))_{v \in \archplaces}$ or its coordinate-wise
continuation $K^n \to \prod_{v\in\archplaces}K_{v}^n$, $n \in
\NN$. When it is convenient, we will also write $\coord^{(v)}$ instead
of $\sigma_v(\coord)$ and $\absv{\coord}$ instead of
$\absv{\sigma_v(\coord)}$, for $\coord\in K$.

For each place $v \in \places$, lying over a place $w$ of $\QQ$, we
write $d_v:=[K_v : \QQ_w]$ for the local degree at $v$. The set of
archimedean (resp. non-archimedean) places of $K$ will be denoted by
$\archplaces$ (resp. $\nonarchplaces$). The completion $K_v$ at
$v\in\archplaces$ is identified with $\RR$ (resp. $\CC$) if $v$ is
real (resp. complex).

For a prime ideal $\p$ of a ring $A$, we write $k(\p)$ for the residue
field. Given an inclusion of rings $A\subseteq A'$ and an $A$-scheme $X$, we denote by $X_{A'}$ the base change
$X\times_{\spec(A)}\spec(A')$. Moreover, $\G_{m,X}$ (resp. $\G_{m,A}$)
denotes the multiplicative group scheme over $X$ (resp. over
$\spec(A)$).  We denote by $\pic(X)$ the Picard group of a scheme $X$
and by $N^1(X)$ its group of numerically equivalent divisor
classes. Given an ideal $I$ of a ring $R$, we denote by $V(I)$ both
the closed subset of $\spec(R)$ defined by $I$ and the closed
subscheme $\spec(R/I)$.

All implied constants in Landau's $O$-notation and Vinogradov's
$\ll$-notation may depend on $K$. Additional dependencies are
indicated by a subscript.

\section{Parameterization by integral points on twisted torsors}\label{sec:param_twisted_torsors}

Torsors over varieties under algebraic groups are known to give
partitions of the set of rational points of the variety in terms of
images of rational points on twisted torsors (see
\cite[(2.7.2)]{MR89f:11082} and \cite[p.~ 22]{MR1845760}). This
phenomenon holds also for torsors over more general schemes as the
following Proposition shows. 
For universal torsors of smooth projective split toric schemes over the ring of integers of a number field it was observed in \cite[p.~12]{MR1650339}. 
For the definition and basic properties
of torsors we refer to \cite[\S III.4]{MR559531} and
\cite[\S2.2]{MR1845760}. For the notion of twisted torsors see, for
example, \cite[p.~20]{MR1845760}.

\begin{prop}\label{generaldisjunion}
  Let $\St$ be a scheme, $\Gt$ an abelian group scheme over $\St$,
  $\Xt$ a $\St$-scheme, and $\pi:\Yt\to\Xt$ a torsor under
  $\GtX:=\Gt\times_{\St}\Xt$. Assume that the twisted torsors
  ${}_{\Zt}\Yt$ exist for all $\St$-torsors $\Zt$ under $\Gt$ (this is
  the case, for example, if $\Gt$ is affine over $\St$). Then
\begin{equation*}
\Xt(\St)=\bigsqcup_{[\Zt]\in H_{fppf}^1(\St,\Gt)}{}_\Zt\pi(({}_{\Zt}\Yt)(\St)),
\end{equation*}
where $\bigsqcup_{[\Zt]\in H_{fppf}^1(\St,\Gt)}$ is a disjoint union running through a system of representatives for the classes in $H_{fppf}^1(\St,\Gt)$ and ${}_\Zt\pi:{}_{\Zt}\Yt\to\Xt$ is a twist of $\Yt$ by $-[\Zt]\in H_{fppf}^1(\St,\Gt)$.
\end{prop}
\begin{proof}
The proof given in  \cite[p.~ 22]{MR1845760} works also with $\spec k$ replaced by our base scheme $\St$.
\end{proof}

Our twisted torsors will be given as open subschemes of closed
subschemes of twisted affine spaces. Hence, we start with a definition
of those.

\begin{defin}\label{def:twist_affine}
  Let $A$ be a Dedekind domain with fraction field $K$, and assume
  that we are given a $\ZZ^r$-grading on $K[x_1, \ldots, x_N]$ defined
  by $\deg x_i = m^{(i)} \in \ZZ^r$. For any $r$-tuple $\ua = (\aid_1,
  \ldots, \aid_r)$ of nonzero fractional ideals of $A$, we define the
  \emph{$\ua$-twisted affine space over $A$} as the spectrum $\aAN :=
  \spec({}_{\ua}R)$ of the $\ZZ^r$-graded ring
  \begin{equation*}
    {}_{\ua}R
 := \bigoplus_{m \in \ZZ^r}\ua^{-m}R_m \subseteq K[x_1, \ldots, x_N], 
  \end{equation*}
  where $\ua^{-m} := \aid_1^{-m_1}\cdots\aid_r^{-m_r}$ if $m=(m_1,
  \ldots, m_r)$, and $R_m$ is the degree-$m$-part of $A[x_1,
  \ldots, x_N]$.
\end{defin}

The twisted affine spaces defined above depend, of course, not only on
$N$ and $\ua$, but also on the chosen $\ZZ^r$-grading. Here are some
simple properties.

\begin{prop}\label{prop:twist_aff_properties}
  The $\ua$-twisted affine space over $A$ defined above has the
  following properties.
  \begin{enumerate}[label=(\roman{*}), ref=\emph{(\roman{*})}]
  \item There is a canonical isomorphism
    $\aAN\times_{\spec(A)}\spec(K)\cong\A^N_K$.\label{twist_aff_extension}
  \item \label{twist_aff_covering} Let $U = \spec(A_U)$ be an affine open subset of $\spec(A)$
    such that the fractional ideals $\aid_1A_U, \ldots, \aid_rA_U$ of $A_U$ are
    principal. Then
    \begin{equation*}
      \aAN\times_{\spec(A)}U\cong\A^N_{A}\times_{\spec(A)}U.
  \end{equation*}
  \item \label{twist_aff_ratpoint} Via base extension and the canonical isomorphism from \emph{\ref{twist_aff_extension}},
    we have
    \begin{equation*}
      \aAN(A)=\{(a_1, \ldots, a_N) \in K^N \mid a_i \in
      \ua^{m^{(i)}} \text{ for all } 1\leq i\leq N\}.
    \end{equation*}
  \item $\aAN$ depends, up to isomorphism, only on the ideal classes
    of $\aid_1,\dots,\aid_r$.\label{twist_aff_unique}
  \end{enumerate}
\end{prop}

\begin{proof}
  The canonical homomorphism ${}_{\ua}R \otimes_{A} K \to K[x_1,  \ldots,  x_N]$ provided by the universal property of the tensor product is an isomorphism, which implies \ref{twist_aff_extension}. For $j\in\{1,
  \ldots, r\}$, let $\pigenerator_j$
   be a generator of $\aid_jA_U$ and, with $m \in \ZZ^r$,
  write $\underline{\pigenerator}^m := \pigenerator_1^{m_1}\cdots \pigenerator_r^{m_r}$. Then
  ${}_{\ua}R\otimes_{A} A_U \cong A_U[\underline\pigenerator^{-m^{(1)}}x_1,
  \ldots, \underline\pigenerator^{-m^{(N)}}x_N] \cong A_U[x_1, \ldots, x_N]$,
  which implies \ref{twist_aff_covering}. For
  \ref{twist_aff_ratpoint}, we observe that every $A$-homomorphism
  $\varphi : {}_{\ua}R \to A$ extends uniquely to a $K$-homomorphism
  $\varphi : K[x_1, \ldots, x_N] \to K$. The $K$-homomorphisms coming
  from such $A$-homomorphisms are exactly those with
  $\varphi({}_{\ua}R)\subseteq A$, that is, $\varphi(x_i) \in
  \ua^{m^{(i)}}$ for all $i\in\{1,\dots,N\}$. To prove \ref{twist_aff_unique}, let $\underline b=(b_1, \ldots,
  b_r) \in (K^\times)^r$ and $\aid_j' := b_j\aid_j$ for $j\in\{1,\dots,r\}$. Then the
  $K$-automorphism of $K[x_1, \ldots, x_N]$ mapping $x_i \mapsto
  \underline b^{-m^{(i)}}x_i$ induces an $A$-isomorphism between ${}_{\ua}R$ and
  ${}_{\ua'}R$.
\end{proof}

Next, we define twists of open subschemes of closed subschemes of
$\A^N_A$ as certain subschemes of twisted affine spaces.

\begin{defin}\label{def:twist_subscheme}
  With the hypotheses of Definition \ref{def:twist_affine}, let $I_1,
  I_2$ be $\ZZ^r$-homogeneous ideals of $A[x_1, \ldots, x_N]$, and let
  $\YtOk$ be the subscheme  of $\A^N_A$
  defined by $\YtOk := V(I_1)\smallsetminus V(I_2)$. With $I_{j,m}$
  denoting the degree-$m$-part of $I_j$, we define the ideals
  \begin{equation*}
    {}_{\ua}I_j := \bigoplus_{m \in \ZZ^r}\ua^{-m}I_{j,m} \subseteq {}_{\ua}R.
  \end{equation*}
  The \emph{twist of $\YtOk$ by $\ua$} is the subscheme of $\aAN$ defined
  by
  \begin{equation*}
    \aYtOk := V({}_{\ua}I_1)\smallsetminus V({}_{\ua}I_2).
  \end{equation*}
\end{defin}

\begin{prop}\label{prop:twist_subscheme_properties}
  The twist of $\YtOk$ by $\ua$ defined above has the following
  properties.
  \begin{enumerate}[label=(\roman{*}), ref=\emph{(\roman{*})}]
  \item The canonical isomorphism from Proposition
    \ref{prop:twist_aff_properties}, \emph{\ref{twist_aff_extension}},
    induces an isomorphism
    $\aYtOk\times_{\spec(A)}\spec(K)\cong\Ytk$.\label{twist_subscheme_extension}
  \item \label{twist_subscheme_covering} Let $U = \spec(A_U)$ be an affine open subset of $\spec(A)$
    such that the fractional ideals $\aid_1A_U, \ldots, \aid_rA_U$ of
    $A_U$ are principal. Then
    \begin{equation*}
      \aYtOk\times_{\spec(A)}U\cong\YtOk\times_{\spec(A)}U.
    \end{equation*}
  \item \label{twist_subscheme_ratpoint} Via base extension and the canonical isomorphism from
    \emph{\ref{twist_subscheme_extension}}, the set of $A$-points
    $\aYtOk(A)$ is the subset of all $\coordtuple = (\coord_1, \ldots,
    \coord_N) \in K^N$ with $a_i \in \ua^{m^{(i)}}$ for all $i \in
    \{1, \ldots, N\}$, such that
    \begin{equation}\label{eq:twist_subscheme_cop}
      \sum_{m\in\ZZ^r}\sum_{f\in I_{2,m}}f(\coordtuple)\ua^{-m} = A 
    \end{equation}
    and
    \begin{equation}\label{eq:twist_subscheme_torsor}
      g(\coordtuple)=0 \text{ for all }g \in I_1.
    \end{equation}
  \item $\aYtOk$ depends, up to isomorphism, only on the ideal classes
    of $\aid_1,\dots,\aid_r$.\label{twist_subscheme_unique}
  \end{enumerate}
\end{prop}

\begin{proof}
  Since the inclusion $A\to K$ is flat and ${}_{\ua}I_1\otimes_AK\cong
  I_1\otimes_AK$ under the canonical isomorphism ${}_{\ua}
  R\otimes_AK\cong K[x_1,\dots,x_N]$, we see that
  $V({}_{\ua}I_1)\times_{\spec(A)}\spec(K)\cong
  V(I_1)\times_{\spec(A)}\spec(K)$.  Let $I_2$ be generated by
  homogeneous polynomials $f_1, \ldots, f_\mt \in A[x_1, \ldots, x_N]$, and for
  every $i \in \{1, \ldots, \mt\}$, let $b_{i,1}, \ldots, b_{i,t_i}$
  be generators of the fractional ideal $\ua^{-\deg f_i}$. Then
  ${}_{\ua}I_2$ is generated in ${}_{\ua}R$ by the elements
  $b_{i,j}f_i$, and $\aYtOk$ is covered by affine open subsets
  $\spec(({}_{\ua}R/{}_{\ua}I_1)_{b_{i,j}f_i})$. Moreover,
  \begin{equation*}
  ({}_{\ua}R/{}_{\ua}I_1)_{b_{i,j}f_i}\otimes_AK\cong (A[x_1,\dots,x_N]/I_1)_{f_i}\otimes_AK
  \end{equation*}
  for every $i\in\{1,\dots,m\}$ and $j\in\{1,\dots,t_i\}$, which shows \ref{twist_subscheme_extension}.
 
  For $j\in\{1,
  \ldots, r\}$, let $\pigenerator_j$
   be a generator of $\aid_jA_U$ and, with $m \in \ZZ^r$,
  write $\underline{\pigenerator}^m := \pigenerator_1^{m_1}\cdots \pigenerator_r^{m_r}$.
  Let $\varphi_{\underline{\pigenerator}}
  : A_U[x_1, \ldots, x_N] \to A_U[\underline\pigenerator^{-m^{(1)}}x_1,
  \ldots,\underline\pigenerator^{-m^{(N)}}x_N]$ be the isomorphism that sends $x_i \mapsto
  \underline\pigenerator^{-m^{(i)}}x_i$. For every homogeneous $f\in I_2$ we obtain
  \begin{equation*}\begin{aligned}
  (A[x_1,\dots,x_N]/I_1)_{f}\otimes_AA_U &\cong \left(A_U[x_1, \ldots, x_N]/(I_1\otimes_AA_U)\right)_{f}\\
  &\cong\left(\varphi_{\underline \pigenerator}(A_U[x_1, \ldots, x_N])/\varphi_{\underline\pigenerator}(I_1\otimes_AA_U)\right)_{\varphi_{\underline\pigenerator}(f)} \\&\cong\left(({}_{\ua}R\otimes_AA_U)/({}_{\ua}I_1\otimes_AA_U)\right)_{\underline\pigenerator^{-\deg f}f}.
  \end{aligned}\end{equation*}
This proves \ref{twist_subscheme_covering}, since $f \in
I_2\otimes_A A_U$ is equivalent to
$\underline\pigenerator^{-\deg f}f \in {}_{\ua}I_2\otimes_A A_U$.  For
\ref{twist_subscheme_ratpoint}, we first consider
$V({}_{\ua}I_1)(A)$. Via the identification in Proposition
\ref{prop:twist_aff_properties}, \ref{twist_aff_ratpoint}, these
points correspond to $K$-homomorphisms $\varphi : K[x_1, \ldots,
x_N]\to K$ with $\varphi(x_i) \in \ua^{m^{(i)}}$ whose kernel contains
the homogeneous ideal ${}_{\ua}I_1 \otimes_{A} K = I_1 \otimes_{A} K $, that is, to
points $\coordtuple \in K^N$ with $\coord_i \in \ua^{m^{(i)}}$ and
satisfying \eqref{eq:twist_subscheme_torsor}.

  Next, let us consider $(\aAN\smallsetminus
  V({}_{\ua}I_2))(A)$. These points correspond to $A$-homo\-morphisms
  $\varphi : {}_{\ua}R \to A$ such that ${}_{\ua}I_2\not\subseteq\varphi^{-1}(\p)
  $ for all prime ideals $\p$ of $A$. That is,
  $\varphi({}_{\ua}I_2)A = A$. Keeping in mind that ${}_{\ua}I_2$ is
  generated by its homogeneous elements and using the description
  of $\aAN(A)$ from Proposition
  \ref{prop:twist_aff_properties}, \ref{twist_aff_ratpoint}, we see
  that $(\aAN\smallsetminus
  V({}_{\ua}I_2))(A)$ corresponds to the set of all $\coordtuple \in
  K^N$ with $\coord_i \in \ua^{m^{(i)}}$ and satisfying
  \eqref{eq:twist_subscheme_cop}.
  
  To prove \ref{twist_subscheme_unique}, let $\underline b=(b_1, \ldots,
  b_r) \in (K^\times)^r$ and $\aid_j' := b_j\aid_j$ for  $j\in\{1,\dots,r\}$. Then the
  $K$-automorphism of $K[x_1, \ldots, x_N]$ mapping $x_i \mapsto
  \underline b^{-m^{(i)}}x_i$ induces an $A$-isomorphism between ${}_{\ua}R$ and
  ${}_{\ua'}R$ which maps ${}_{\ua}I_j$ isomorphically onto
  ${}_{\ua'}I_j$, for $j\in\{1, 2\}$.
\end{proof}

Now we can focus on the case where $\YtOk$ is a torsor over an
$A$-scheme $\XtOk$ under a split torus $\G_{m,\XtOk}^r$. Throughout the rest of this section, we assume
the following setup.

Let $A$ be a Dedekind domain with fraction field $K$, and let there be
a $\ZZ^r$-grading on $K[x_1, \ldots, x_N]$ defined by $\deg x_i =
m^{(i)} \in\ZZ^r$.

Let $\XtOk$ be a separated scheme of finite type over $A$ that admits
an $\XtOk$-torsor $\pi : \YtOk \to \XtOk$ under a split torus
$\G_{m,\XtOk}^r$. We assume that there are
$\ZZ^r$-homogeneous polynomials $\ftOk_1,\dots,\ftOk_\mt$,
$\gtOk_1,\dots,\gtOk_\st\in A[x_1,\dots,x_N]$ such that
$\YtOk=V(\gtOk_1,\dots,\gtOk_\st)\smallsetminus V(\ftOk_1,\dots,\ftOk_\mt)$
as subscheme of $\A^N_{A}$. Moreover, we assume that the action of
$\G_{m,\XtOk}^r$ on $\YtOk$ is induced by the following action on
points:
\begin{align*}
  (\sg_1,\dots,\sg_r)*(a_1,\dots,a_N)=(\usg^{m^{(1)}} a_1,\dots,
  \usg^{m^{(N)}}a_N)
\end{align*}
for all $\usg=(\sg_1,\dots,\sg_r)\in\G_{m,\XtOk}^r(A)$ and
$(\coord_1,\dots,\coord_N)\in\YtOk(A)$, where we write $\usg^m :=
\sg_1^{m_1}\cdots \sg_r^{m_r}$ for $m=(m_1, \ldots,
m_r)\in\ZZ^r$. Under these assumptions, we now define the twists of
$\pi : \YtOk \to \XtOk$.

\begin{defin}\label{defin:twisted_torsors}
  Under the above hypotheses, let $\ua=(\aid_1,\dots,\aid_r)$ be an $r$-tuple of nonzero fractional ideals of $A$, and let $\aYtOk$
  be the twist of $\YtOk$ from Definition
  \ref{def:twist_subscheme}. Then the \emph{$\ua$-twist of $\pi :
    \YtOk \to \XtOk$} is the morphism $\api : \aYtOk \to \XtOk $
  obtained by glueing the following morphisms:
  \begin{equation*}
  \api_U:\aYtOk\times_{\spec(A)}U \to \XtOk\times_{\spec(A)}U,
  \end{equation*}
  where $U$ runs through an open covering of $\spec(A)$ by affine
  subschemes $U=\spec(A_U)$ such that $\aid_1A_U, \ldots, \aid_rA_U$
  are principal ideals of $A_U$, and $\api_U$ is defined as
  composition of $\pi$ after the isomorphism
  $\phi_{\underline\pigenerator}:\aYtOk\times_{\spec(A)}U\to\YtOk\times_{\spec(A)}U$
  from Proposition \ref{prop:twist_subscheme_properties},
  \ref{twist_subscheme_covering}, induced by the isomorphism
 \begin{equation*}
   \varphi_{\underline{\pigenerator}}
   : A_U[x_1, \ldots, x_N] \to A_U[\underline\pigenerator^{-m^{(1)}}x_1,
   \ldots,\underline\pigenerator^{-m^{(N)}}x_N], \quad x_i \mapsto
   \underline\pigenerator^{-m^{(i)}}x_i, \end{equation*}
 where $\pigenerator_j$ is a generator of $\aid_jA_U$ for $j\in\{1,
 \ldots, r\}$, and $\underline{\pigenerator}^m := \pigenerator_1^{m_1}\cdots \pigenerator_r^{m_r}$ for all $m \in \ZZ^r$. As will be shown in the proof below, the definition of $\api$ does not depend on the choice of the open subsets $U$ nor on the choice of the generators $\pigenerator$.
   \end{defin}

Now we are ready to state the second main theorem of this article.

\begin{theorem}\label{thm:param}
  The $\ua$ twists $\api : \aYtOk \to \XtOk$ defined above have the
  following properties.
  \begin{enumerate}[label=(\roman{*}), ref=\emph{(\roman{*})}]
  \item The morphism $\api : \aYtOk \to \XtOk$ is a torsor over $\XtOk$
    under $\G_{m,\XtOk}^r$ of class $[\YtOk]-[\ua]\in
    H^1_{\text{\it{\'et}}}(X,\G_{m,\XtOk}^r)$.
  \item \label{thm:param_union} Let $\classrepsyst$ be a system of representatives for the
    class group $\Pic(A)$ of $A$. If $\XtOk$ is proper over $A$
    then, under base extension, the set of rational points on $\Xtk$
    decomposes as a disjoint union
    \begin{equation*}
      \Xtk(K)=\bigsqcup_{\classtuple\in\classrepsyst^r}\classtuplepi(\cYtOk(A)).
    \end{equation*}
  \item \label{thm:param_coord} As a subset of $K^N$, the set $\cYtOk(A)$ is equal to the
    set of all $\coordtuple \in K^N$ whose coordinates $\coord_i$ lie
    in the fractional ideals $\classtuple^{m^{(i)}}$, satisfying the
    coprimality conditions expressed by
    \begin{equation*}
      \sum_{i=1}^\mt\ftOk_i(\coordtuple)\classtuple^{-\deg\ftOk_i}=A
    \end{equation*}
and the torsor equations
\begin{equation*}
   \gtOk_j(\coordtuple)=0  \text{ for all } j \in \{1,\ldots,\st\}.
\end{equation*}
    \end{enumerate}
\end{theorem}

\begin{proof}
  For every choice of affine open subsets $U$, $U'$ of $\spec(A)$ as
  in Definition \ref{defin:twisted_torsors}, and corresponding
  $r$-tuples $\underline\pigenerator$, $\underline\pigenerator'$ of
  generators for the principal fractional ideals over $U$, resp. $U'$,
  let $
  \varphi_{\underline\pigenerator,\underline\pigenerator'}:A_{U\cap
    U'}[x_1, \ldots,x_N]\to A_{U\cap U'}[x_1, \ldots,x_N]$ be the
  isomorphism induced by the automorphism of $K[x_1,\dots,x_N]$
  mapping $x_i\mapsto
  \underline\pigenerator^{-m^{(i)}}\underline\pigenerator'^{m^{(i)}}x_i$,
  and $\phi_{\underline\pigenerator,\underline\pigenerator'}$ the
  automorphism of $\YtOk\times_{\spec(A)}(U\cap U')$ induced by $
  \varphi_{\underline\pigenerator,\underline\pigenerator'}$. Then
  $\phi_{\underline\pigenerator}=\phi_{\underline\pigenerator,\underline\pigenerator'}\circ\phi_{\underline\pigenerator'}$
  on $\aYtOk\times_{\spec(A)}(U\cap U')$.  We observe that
  $\phi_{\underline\pigenerator,\underline\pigenerator'}$ are the
  automorphisms induced by the $\G_{m,\XtOk}^r$-action of the cocycle
  $(\pigenerator^{-1}_1\pigenerator'_1,\dots,\pigenerator^{-1}_r\pigenerator_r')_{U',U}$
  that represents the class $-[\ua]\in
   \Pic(A)^r$. Thus
  $\pi\circ\phi_{\underline\pigenerator,\underline\pigenerator'}=\pi$
  on $\YtOk\times_{\spec(A)}(U\cap U')$, as $\pi:\YtOk\to\XtOk$ is a
  torsor under $\G_{m,\XtOk}^r$,
and the morphism $\api$ is well defined.
Since the automorphisms
$\phi_{\underline\pigenerator,\underline\pigenerator'}$ are
$\G_{m,\XtOk}^r$-equivariant, the $\XtOk$-scheme $\aYtOk$ is endowed
with an action of $\G_{m,\XtOk}^r$, and the morphism $\api$ is an
$\XtOk$-torsor under $\G_{m,\XtOk}^r$ of class $[\YtOk]-[\ua]\in
H^1_{\text{\it{\'et}}}(X,\G_{m,\XtOk}^r)$, 
via the homomorphism of cohomology groups
\begin{equation*}
 \Pic(A)^r\cong H^1_{\text{\emph{\'et}}}
  (\spec(A),\G_{m,A}^r)\to H^1_{\text{\it{\'et}}}(X,\G_{m,\XtOk}^r),
\end{equation*}
where the first isomorphism comes from the fact that \'etale cohomology
commutes with direct sums (see \cite[Remark III.3.6 (d)]{MR559531}).
Here we used {\it{\'etale}} cohomology groups in place of {\it{fppf}} because for $\G_{m}^r$ they coincide.

We recall that two torsors with the same class in
$H^1_{\text{\it{\'et}}}(X,\G_{m,\XtOk}^r)$ are $\XtOk$-isomorphic, so the images of
their structure morphisms coincide as subsets of
$\XtOk$. 
By the valuative criterion of properness, $\XtOk_K(K)=\XtOk(A)$ under base extension.
Thus, property \ref{thm:param_union} follows from  Proposition \ref{generaldisjunion}. 

Finally, \emph{(iii)} was already proved in Proposition \ref{prop:twist_subscheme_properties}, \ref{twist_subscheme_ratpoint}.
\end{proof}

In the case of universal torsors of smooth projective split toric schemes over the ring of integers of a number field, our Theorem \ref{thm:param} recovers the same parameterization via integral points on twisted universal torsors embedded into twisted affine spaces of \cite[p.~15]{MR1650339}.

\section{Models of universal torsors}\label{sec:universal_torsors}

This section is devoted to descent properties of universal torsors of
certain projective varieties. Let $A$ be a noetherian integral domain
with fraction field $K$ of characteristic $0$, and let $\overline{K}$
be an algebraic closure of $K$.

Given an integral, smooth, projective variety $\overline \Xt$ over
$\overline K$, whose Cox ring $\cox(\overline \Xt)$ is finitely
generated and defined over $A$, we construct an $A$-model $\Xt$ of
$\overline \Xt$ and an $A$-model $\Yt$ of a universal torsor
$\overline \Yt$ of $\overline \Xt$ contained in the spectrum of
$\cox(\overline \Xt)$ that turns out to be a universal torsor over
$\Xt$ under some additional conditions.

\begin{construction}\label{hypot3}
  We assume that $\pic(\overline \Xt)\cong\Z^r$,
  and that the Cox ring of $\overline X$ is $\cox(\overline \Xt)=\overline K[\eta_1,\dots,\eta_N]/I$,
  where $\eta_1,\dots,\eta_N$ are $\pic(\overline\Xt)$-homogeneous and
  $I$ is generated by polynomials $g_1,\dots,g_s\in
  A[\eta_1,\dots,\eta_N]$.  We denote by $\overline
  \Yt\subseteq\spec(\cox(\overline \Xt))$ the characteristic space
  defined in \cite[Constructions 1.6.1.3, 1.6.3.1]{cox}. Then
  $\overline \Yt$ is a universal torsor of $\overline \Xt$ by \cite[Proposition 6.1.3.9]{cox}. By \cite[Corollary
  1.6.3.6]{cox}, we know that $\overline \Yt$ is an open subset of
  $\spec(\cox(\overline \Xt))$, whose complement is defined by monic
  monomial equations
  \begin{equation*}
  f_1,\dots,f_m\in \overline K[\eta_1,\dots,\eta_N]\smallsetminus I.\end{equation*}

Fix an isomorphism between $\pic(\overline \Xt)\cong\mathbb{Z}^r$
given by a basis $\ell_1,\dots,\ell_r$ of $\pic(\overline \Xt)$. For
$i\in\{1,\dots,N\}$, let $m^{(i)}\in\ZZ^r$ be the degree of
$\eta_i$. By \cite[Construction 1.6.1.3]{cox}, the action of
$\G_{m,\overline \Xt}^r$ on $\overline\Yt$ is induced by the action of
$\G_{m,\overline \Xt}^r$ on $\spec(\cox(\overline\Xt))(\overline K)$
defined by the homomorphism
  \begin{equation*}
  \cox(\overline\Xt)\ \to\ \overline K[\zt_1,\zt_1^{-1},\dots,\zt_r,\zt_r^{-1}]\otimes_{\overline K}\cox(\overline\Xt),\quad \eta_j\mapsto\underline\zt^{m^{(j)}}\otimes\eta_j.
  \end{equation*}
   
   Without
  loss of generality, we can assume 
  that $I\cap A[\eta_1,\dots,\eta_N]=(g_1,\dots,g_s)$.
  Let
  \begin{equation*}
  R:=A[\eta_1,\dots,\eta_N]/(g_1,\dots,g_s),\end{equation*}
   and let $\Yt$ be the
  complement of the closed subset of $\spec(R)$ defined by
  $f_1,\dots,f_m$. 
  For $i\in\{1,\dots,m\}$, let $U_i:=\spec(R[f_i^{-1}])$
  and \begin{equation*} \overline{U}_i:=U_i\times_{\spec(A)}\spec(\overline
    K)\cong\spec(\cox(\overline \Xt)[f_i^{-1}]).
\end{equation*}
Then $\{U_i\}_{1\leq i\leq m}$ is an affine open covering of $\Yt$, the
family $\{\overline U_i\}_{1\leq i\leq m}$ is an affine open covering of
$\overline \Yt$, and $\Yt_{\overline K}\cong \overline \Yt$.

The $\pic(\overline \Xt)$-grading of $\cox(\overline \Xt)$ induces a $\pic(\overline \Xt)$-grading  on $R$ by assigning the degrees of $\eta_1,\dots,\eta_N$. 
We assume that $(R;f_1,\dots,f_m)$ satisfies the following condition:
\begin{equation}
    \parbox{0.8\linewidth}{for every $i,j\in\{1,\dots,m\}$, there is a homogeneous
      invertible element of $R[f_i^{-1}]$ of degree a 
      multiple of $\deg f_j$.}\label{conditionstar}
\end{equation}
For $i\in\{1,\dots,m\}$, let $R_i$ be the degree-$0$-part of the ring
$R[f_i^{-1}]$ and $V_i:=\spec(R_i)$. Then $R_i\otimes_A\overline K$ is
the degree-$0$-part of $\cox(\overline \Xt)[f_i^{-1}]$ for all
$i\in\{1,\dots,m\}$, and by construction of the universal torsor
structure $\overline \Yt\to\overline \Xt$ (see \cite[Remark
1.25]{HAUSEN}), gluing the family of schemes
$\{\spec(R_i\otimes_A\overline K)\}_{1\leq i\leq m}$ yields a variety
isomorphic to $\overline \Xt$. Let $\Xt$ be the $A$-scheme obtained by
gluing $\{V_i\}_{1\leq i\leq m}$. Then $\Xt$ is a model of
$\overline\Xt$ over $A$ and comes endowed with a natural morphism
$\pi:\Yt\to \Xt$ induced by the inclusions $R_i\to R[f_i^{-1}]$ for
$i\in\{1,\dots,m\}$. Since the inclusions $R_i\to R[f_i^{-1}]$ induce
surjective morphisms $U_i\to V_i$ for all $i\in\{1,\dots,m\}$, the
morphism $\pi$ is surjective. Moreover, $\pi$ is of finite
presentation because $X$ is noetherian and $R[f_i^{-1}]$ is a finitely
generated $R_i$-algebra for every $i\in\{1,\dots,m\}$.  Since
$f_1,\dots,f_m$ are $\ZZ^r$-homogeneous, the homomorphism
\begin{equation*}
  R\ \to\ A[\zt_1,\zt_1^{-1},\dots,\zt_r,\zt_r^{-1}]\otimes_{A}R,\quad \eta_j\mapsto\underline\zt^{m^{(j)}}\otimes\eta_j
\end{equation*}
induces an action of $\G_{m,\Xt}^r$ on $\Yt$ which is given by
\begin{equation}\label{action_on_points}
  \us*(\coord_1,\dots,\coord_N)=(\us^{m^{(1)}}\coord_1,\dots,\us^{m^{(N)}}\coord_N)
  \end{equation}
  on $A$-points $\us=(s_1,\dots,s_r)\in\G_{m,A}^r(A)$ and
  $(\coord_1,\dots,\coord_N)\in\Yt(A)$, where $\us^m:=s_1^{m_1}\cdots
  s_r^{m_r}$ for all $m=(m_1,\dots,m_r)\in\ZZ^r$. 

  Moreover, $\pi$ is an $\Xt$-torsor under $\G_{m,\Xt}^r$ (compatible
  with the universal torsor structure of $\overline \Yt$ over
  $\overline \Xt$) if and only if $\pi$ is flat and the morphism of
  schemes $\phi:\G^r_{m,A}\times_{\spec(A)}\Yt\to \Yt\times_{\Xt}\Yt$
  that sends $(\us,\coordtuple)\mapsto(\us*\coordtuple,\coordtuple)$,
  obtained by gluing the morphisms
  \begin{equation*}
    \varphi_i:R[f_i^{-1}]\otimes_{R_i}R[f_i^{-1}]\to A[\zt_1,\zt_1^{-1},\dots,\zt_r,\zt_r^{-1}]\otimes_AR[f_i^{-1}],\quad \eta_j\otimes \eta_l\mapsto\underline\zt^{m^{(j)}}\otimes \eta_j\eta_l
  \end{equation*}
  for $1\leq i\leq m$, is an isomorphism. 
\end{construction}

\begin{rem}\label{udesck}
  If $A=K$, then $\pi:\Yt\to\Xt$ is a universal torsor by \emph{fpqc} descent, see \cite[\S2]{MR89f:11082}.
\end{rem}

In his proof of Manin's conjecture for toric varieties, Salberger
introduced universal torsors for certain schemes defined over
noetherian base schemes \cite[Definition 5.14]{MR1679841}. Under
reasonable hypotheses, the following theorem shows that $\pi : \Yt \to
\Xt$ is indeed a universal torsor according to Salberger's definition.

\begin{theorem}\label{universaltorsor}Let $\pi$ be as in Construction \ref{hypot3}. If $(R; f_1,\dots,f_m)$ satisfies the condition that
    \begin{equation}
      \parbox{0.8\linewidth}{every element of $\pic(\overline \Xt)$ is the degree of a homogeneous invertible element of
        $R[f_i^{-1}]$, for all $i\in\{1, \ldots,m\}$,}\label{picardgeneration}
    \end{equation}
    then $\pi$ is an $\Xt$-torsor under $\G_{m,\Xt}^r$. If we
    additionally assume that $X(A)\neq\emptyset$, that $X$ is smooth,
    projective, of constant relative dimension, and with geometrically
    integral fibers over $A$, that $\pic(\Xt_K)=\pic(\overline{\Xt})$,
    and that for every prime ideal $\p$ of $A$ the cohomology groups
    $H^i(\Xt_{k(\p)},\mathcal{O}_{\Xt_{k(\p)}})$ vanish for
    $i\in\{1,2\}$, then $\pi$ is a universal torsor of $\Xt$.
\end{theorem}

\begin{proof}
  Flatness of $\pi$ is equivalent to flatness of all the inclusions
  $R_i\to R[f_i^{-1}]$, i.e., to injectiveness of the induced
  morphisms $J\otimes_{R_i}R[f_i^{-1}]\to R[f_i^{-1}]$ for all ideals
  $J$ of $R_i$. Fix $i\in\{1,\dots,m\}$. Let $J$ be an ideal of $R_i$
  and $J\otimes_{R_i}R[f_i^{-1}]\to R[f_i^{-1}]$ the induced
  morphism. A general element in the kernel of this morphism is
  $h=\sum_{j=1}^nh_j\otimes h_j'$, where $h_j\in J$ has degree $0$ and
  $h_j'\in R[f_i^{-1}]$, and such that $\sum_{j=1}^nh_jh_j'=0$ in
  $R[f_i^{-1}]$. Since $R[f_i^{-1}]$ is a graded ring, it is enough to
  consider homogeneous elements $h$, i.e., with all $h_j'$ homogeneous
  of fixed degree $\deg h\in \ZZ^r$.  Since the degrees of the
  homogeneous invertible elements of $R[f_i^{-1}]$ generate
  $\pic(\overline \Xt)\cong\ZZ^r$, there exists $f\in
  R[f_i^{-1}]^\times$ of degree $\deg h$.  Then
  $h=(\sum_{j=1}^nh_jh_j'f^{-1})\otimes f=0$ in
  $J\otimes_{R_i}R[f_i^{-1}]$.

  In order to prove that $\phi$ is an isomorphism, it suffices to
  prove that all $\varphi_i$ are isomorphisms. For every
  $i\in\{1,\dots,m\}$ and $k\in\{1,\dots,r\}$, let $h_{i,k}\in
  R[f_i^{-1}]^\times$ be a homogeneous element of degree $\ell_k$.
  Then the morphism
  \begin{equation*}
    \psi_i: A[\zt_1,\zt_1^{-1},\dots,\zt_r,\zt_r^{-1}]\otimes_AR[f_i^{-1}]\to R[f_i^{-1}]\otimes_{R_i}R[f_i^{-1}]
  \end{equation*}
  that sends
  \begin{equation*}
    1\otimes\eta_j\mapsto1\otimes\eta_j\quad \text{ and } \quad \zt_k\otimes 1\mapsto h_{i,k}\otimes h_{i,k}^{-1}
  \end{equation*} 
  for all $j\in\{1,\dots,N\}$ and $k\in\{1,\dots,r\}$, is well defined
  and inverse to $\varphi_i$, for all $i\in\{1,\dots,m\}$. 

  By \cite[Proposition 2.1]{MR2950702} the relative \'etale Picard
  functor of $\Xt$ over $A$ is representable by a twisted constant
  $A$-group scheme $\pic_{\Xt/A}$. Since
  $\pic(\Xt_K)=\pic(\Xt_{\overline K})$, the group scheme
  $\pic_{\Xt/A}$ is constant and represented by $\ZZ^r$ by \'etale
  descent. By \cite[Corollary III.12.9]{MR0463157},
  $R^2f_*\mathcal{O}_{\Xt}=0$, where $f:X\to\spec(A)$ is the structure
  morphism.  Since $\Yt_K$ is a universal torsor of $\Xt_K$ by Remark
  \ref{udesck} and the morphism \begin{equation*}
    \Hom_{A}(\widehat{\G_{m,A}^r},\pic_{\Xt/A})\to
    \Hom_{K}(\widehat{\G_{m,K}^r},\pic_{\Xt_K/K})\end{equation*}
  is injective, the torsor $\Yt\to\Xt$ is universal, as the exact sequences
  \cite[5.13]{MR1679841} are functorial.
\end{proof}

\begin{rem}\label{geometric_interpretation}
A geometric interpretation of \eqref{picardgeneration} is the following equivalent formulation:  the open subset complement to the support of the effective divisor defined by $f_i$ has trivial Picard group for all $i\in\{1,\dots,m\}$. 
\end{rem}

The rest of this section provides criteria to check the various
hypotheses of Theorem \ref{universaltorsor}. We start by showing that
the model $X$ of Construction \ref{hypot3} is independent of the
choice of $f_1,\dots,f_m$ under some conditions.
\begin{lemma}\label{gluinglemma}
Let $f_1',\dots,f'_{m'}\in\overline K[\eta_1,\dots,\eta_N]\smallsetminus I$ be monic monomials such that $(R;f_1,\dots,f_m,f'_1,\dots,f'_{m'})$ satisfies the condition $\eqref{conditionstar}$. Let $C_A$ and $C'_A$ be the ideals of $A[\eta_1,\dots,\eta_N]$ generated by $f_1,\dots,f_m,g_1,\dots,g_s$, and $f'_1,\dots,f'_{m'},g_1,\dots,g_s$ respectively, and assume that $\sqrt{C'_A}=\sqrt{C_A}$. Then $\Xt$ is isomorphic to the $A$-model $X'$ of $\overline \Xt$ constructed using $f_1',\dots,f_{m'}'$ in Construction \ref{hypot3}.
\end{lemma}
\begin{proof}
  For every $i\in\{m+1,\dots,m+m'\}$, let $f_{i}:=f'_{i-m}$, and
  $V_{i}:=\spec(R_{i})$, where $R_i$ is the degree-0-part of
  $R[{f_i}^{-1}]$. For every $i,j\in\{1,\dots,m+m'\}$, let $h_{i,j}\in
  R[f_i^{-1}]^\times$ be a homogeneous element of degree $-n_{i,j}\deg
  f_j$ for some positive integer $n_{i,j}$, and let
  $V_{i,j}:=\spec(R_i[(f_j^{n_{i,j}}h_{i,j})^{-1}])\subseteq V_i$.
  Since $\sqrt{C'_A}=\sqrt{C_A}$, the ideal of $R_i$ generated by
  ${f_{m+1}}^{n_{i,m+1}}h_{i,m+1},\dots,{f_{m+m'}}^{n_{i,m+m'}}h_{i,m+m'}$
  contains $f_i^nf_i^{-n}=1$ for some positive integer $n$. Hence,
  $V_i=\bigcup_{j=m+1}^{m+m'}V_{i,j}$ for every
  $i\in\{1,\dots,m\}$. Likewise, $V_i=\bigcup_{j=1}^{m}V_{i,j}$ for
  every $i\in\{m+1,\dots,m+m'\}$.

The identifications $R_i[(f_j^{n_{i,j}}h_{i,j})^{-1}]=R_j[(f_i^{n_{j,i}}h_{j,i})^{-1}]$ inside $R[(f_if_j)^{-1}]$ induce isomorphisms $V_{i,j}\cong V_{j,i}$, for all $i,j\in\{1,\dots,m+m'\}$, that are compatible on the intersections. The schemes $X$ and $X'$ are the gluing of $\{V_i\}_{1\leq i\leq m}$, and $\{V_i\}_{m+1\leq i\leq m+m'}$ respectively, along the isomorphisms mentioned above. 
Since
  $\{V_{i,j}\}_{\substack{1\leq i\leq m,m+1\leq j\leq m+m'}}$ is an open
  covering of $\Xt$, $\{V_{j,i}\}_{\substack{1\leq i\leq m,m+1\leq j\leq
      m+m'}}$ is an open covering of
  $\Xt'$, and all the
  isomorphisms $V_{i,j}\cong V_{j,i}$ are compatible on the
  intersections, they glue to a global isomorphism
  $\Xt\cong\Xt'$.
\end{proof}

The next three propositions provide sufficient conditions for $X$
having geometrically integral fibers, and being smooth and projective over
$A$.

\begin{prop}\label{geomintegralfibers}
If $\spec(R)\to\spec(A)$ has geometrically integral fibers, then $\Xt\to\spec(A)$ has geometrically integral fibers.
\end{prop}
\begin{proof}
  Let $\p$ be a prime ideal of $A$, and let $k$ be an algebraic
  extension of the residue field $k(\p)$. Since $R\otimes_A k$ is an
  integral domain by hypothesis, the ring $R_i\otimes_A k$ is an
  integral domain for all $i\in\{1,\dots,m\}$. Thus, $\Xt_k$ is covered by a
  family of integral open subsets $\{W_i:=\spec(R_i\otimes_A
  k)\}_{1\leq i\leq m}$ such that $W_i\cap W_j$ is nonempty for all
  nonempty $W_i$ and $W_j$.  Indeed, for $i,j\in\{1,\dots,m\}$, the
  intersection $W_i\cap W_j$ is the spectrum of the degree-0-part of
  the ring $(R\otimes_A k)[(f_if_j)^{-1}]$, which is nonzero whenever
  $f_i$ and $f_j$ are nonzero elements of $R\otimes_A k$.

Given any nonempty open subset $U$ of $\Xt_k$ and nonzero sections
$s_1,s_2\in\mathcal{O}_{\Xt_k}(U)$, there exist
$i_1,i_2\in\{1,\dots,m\}$ such that $s_j|_{U\cap W_{i_j}}\neq0$ for
$j\in\{1,2\}$. Therefore, $s_j|_{U\cap W_{i_1}\cap W_{i_2}}\neq0$ for
$j\in\{1,2\}$ as $W_{i_1}$, $W_{i_2}$ are integral, and $U\cap
W_{i_1}$, $W_{i_1}\cap W_{i_2}$ are dense in $W_{i_1}$. Thus,
$(s_1s_2)|_{U\cap W_{i_1}\cap W_{i_2}}\neq0$ and $s_1s_2\neq0$ in
$\mathcal{O}_{\Xt_k}(U)$.
\end{proof}

\begin{prop}\label{smoothnesscriterion}
Assume that $A$ is a Dedekind domain, $\spec(R)\to\spec(A)$ has geometrically integral fibers, and  $\pi$ is flat (the last holds, for example, if \eqref{picardgeneration} is satisfied). If
the Jacobian matrix \begin{equation*}
\left(\frac{\partial g_i}{\partial\eta_j}(\coordtuple)\right)_{\substack{1\leq i\leq s\\ 1\leq j\leq N}}
\end{equation*} 
has rank $N-\dim \overline\Xt-r$ for all $\coordtuple\in Y(\overline{k(\p)})$ and $\p\in\spec(A)$, where $\overline{k(\p)}$ is an algebraic closure of the residue field $k(\p)$, then $X$ is smooth over $A$.
\end{prop}
\begin{proof}
  We prove first that $Y$ is smooth over $A$. By \cite[Proposition
  17.8.2]{MR0238860}, the scheme $Y$ is smooth over $A$ if and only if
  $Y\to\spec(A)$ is flat and $Y_{k(\p)}$ is smooth over $k(\p)$ for
  all $\p\in\spec(A)$. Since $\cox(\overline\Xt)$ is an integral
  domain (see \cite[Theorem 1.5.1.1]{cox}) and $I\cap
  A[\eta_1,\dots,\eta_N]=(g_1,\dots,g_s)$, the ring $R$ is an integral
  domain. Moreover, $A\to R$ is injective. Thus, $R$ is a flat
  $A$-algebra by \cite[Corollary 1.2.14]{MR1917232} as $A$ is a
  Dedekind domain, and in particular $Y$ is flat over $A$.

  Since $\overline\pi:\overline \Yt\to\overline \Xt$ is a torsor under
  $\G_{m,\overline \Xt}^r$, the fiber $\overline\Yt_x$ of $\overline
  \pi$ at a point $x\in\overline\Xt$ is a trivial $k(x)$-torsor under
  $\G_{m,k(x)}^r$, where $k(x)$ is the residue field of $\overline
  \Xt$ at $x$, (see \cite[Corollary III.4.7 and Lemma
  III.4.10]{MR559531}). Hence, $\overline\Yt_x\cong\G_{m,k(x)}^r$ has
  dimension $r$ for all $x\in \overline \Xt$, and $\overline\Yt$ has
  dimension $\dim \overline \Xt+r$ by \cite[Exercise
  II.3.22]{MR0463157}. Then $\dim\Yt_{k(\p)}\geq\dim \overline \Xt+r$
  for all $\p\in\spec(A)$ by \cite[Lemme 13.1.1]{MR0217086}. Let $\p\in\spec(A)$. By the
  assumptions on the Jacobian matrix and \cite[Theorem I.3.2 (c),
  Theorem I.5.1 and Proposition I.5.2A]{MR0463157}, we conclude that
  $\dim\Yt_{k(\p)}=\dim \overline \Xt+r$ and $\Yt_{\overline{k(\p)}}$
  is regular at all its closed points. Then $Y_{k(\p)}$ is
  smooth over $k(\p)$.

  Therefore, $X$ is smooth over $A$ by
  \cite[Proposition 17.7.7]{MR0238860}, as $Y$ is smooth over $A$ and
  $\pi$ is flat and surjective.
\end{proof}

\begin{prop}\label{tors_projectivity}
Assume that $f_1,\dots,f_m$ have all the same degree $[D]$, which is an ample class in $\pic(\overline\Xt)$.
Let $C_{\overline K}$ and $C_A$ be the ideals of $\overline K[\eta_1,\dots,\eta_N]$ and $A[\eta_1,\dots,\eta_N]$ generated by $f_1,\dots,f_m,g_1,\dots,g_s$. If $\sqrt{C_{\overline K}}\cap A[\eta_1,\dots,\eta_N]=\sqrt{C_A}$ then
$X$ is projective over $A$.
\end{prop}
\begin{proof}
  Since $R$ is a finitely generated $A$-algebra, the subring
  $\bigoplus_{n\in\mathbb{N}}R_{n[D]}$, where $R_{n[D]}$ denotes the
  degree $n[D]$-part of $R$, is a finitely generated $A$-algebra by
  \cite[Proposition 1.1.2.4]{cox}. By \cite[Lemme 2.1.6]{MR0163909},
  there exists a positive integer $d$ such that
  $R':=\bigoplus_{n\in\mathbb{N}}R_{nd[D]}$ is generated by $R_{d[D]}$
  as $A$-algebra.  Let $f'_1,\dots,f'_{m'}$ be generators of the
  $A$-module $R_{d[D]}$.

  For all $i\in\{1,\dots,m'\}$, denote by $R'_i$ the
  degree-$0$-part of $R'[{f'_i}^{-1}]$, which is generated by
  $f'_1{f'_i}^{-1},\dots,f'_{m'}{f'_i}^{-1}$ and coincides with the
  degree-$0$-part of $R[{f'_i}^{-1}]$. We
  recall that $\proj(R')$ is defined
  as gluing of $\{V'_i:=\spec(R'_i)\}_{1\leq i\leq m'}$ along the
  isomorphisms on principal open subsets induced by the
  identifications $R'_i[f'_i{f'_j}^{-1}]=R'_j[f'_j{f'_i}^{-1}]$ inside
  $R[(f'_if'_j)^{-1}]$ for $1\leq i,j\leq m'$.
  
  Let $C'_A$ be the ideal of $A[\eta_1,\dots,\eta_N]$ generated by
  $f'_1,\dots,f'_{m'}, g_1,\dots,g_s$. Since $\sqrt{C_A}=\sqrt{C_A^d}$
  and $C_A^d\subseteq C'_A$ by construction, there is an inclusion of
  radical ideals $\sqrt{C_A}\subseteq\sqrt{C'_A}$.  By \cite[Corollary
  1.6.3.6]{cox}, the polynomials $f'_1,\dots,f'_{m'}$ and
  $g_1,\dots,g_s$ generate an ideal of $\overline
  K[\eta_1,\dots,\eta_N]$ whose radical is $\sqrt{C_{\overline
      K}}$. Hence, $\sqrt{C'_A}\subseteq\sqrt{C_A}$.  Since
  $(R;f_1,\dots,f_m,f'_1,\dots,f'_{m'})$ satisfies the condition
  \eqref{conditionstar}, there is an isomorphism $\Xt\cong\proj(R')$
  by Lemma \ref{gluinglemma}.
\end{proof}

In the applications that we have in mind, $\overline{X}$ is obtained
from $\PP^2_{\overline{K}}$ by a chain of blowing-ups at closed
points. The next proposition provides some conditions that make
Construction \ref{hypot3} compatible with such blowing-ups. This can
be used to verify the cohomology conditions of Theorem
\ref{universaltorsor}.

In the situation of Construction \ref{hypot3}, we assume that
$\overline\Xt$ is a surface. We assume that the effective divisor on
$\overline\Xt$ corresponding to the section $\eta_i$ is an integral
curve $E_i$ for all $i\in\{1,\dots,N\}$, and that $E_N$ is a
$(-1)$-curve on $\overline\Xt$. Let $b:\overline\Xt\to\overline\Xt'$
be a birational morphism that contracts exactly $E_N$ according to
Castelnuovo's criterion. For every $i\in\{1,\dots,N-1\}$, let
$E_i'=b(E_i)$. Assume that $x=b(E_N)$ belongs to $E_i'$ exactly for
$i\in\{1,2\}$, and $E_1\cap E_2=\emptyset$. Then
$\cox(\overline\Xt')\cong\cox(\overline\Xt)/(\eta_N-1)$ by
\cite[Proposition 2.2]{computing_Cox_rings}, and the canonical
pull-back of sections is defined by
\begin{equation*}
b^*:\cox(\overline\Xt')\to\cox(\overline\Xt),\quad \eta_i\mapsto\begin{cases}\eta_i\eta_N &\text{ if $i\in\{1,2\}$;}\\ \eta_i&\text{ otherwise.}\end{cases}
\end{equation*}
Let $\overline\Yt'\subseteq\spec(\cox(\overline\Xt'))$ be the
characteristic space of $\overline\Xt'$, and let
$f_1',\dots,f_{m'}'\in \overline K[\eta_1,\dots,\eta_{N-1}]$ be monic
monomials that define the closed subset of
$\spec(\cox(\overline\Xt'))$ complement to $\overline\Yt'$. Let $I'$
be the ideal of $\overline K[\eta_1,\dots,\eta_N]$ generated by
$g_1,\dots,g_s$ and $\eta_N-1$. Assume that $I'\cap
A[\eta_1,\dots,\eta_N]=(g_1,\dots,g_s,\eta_N-1)$. Let
$R'=R/(\eta_N-1)$, and let $\Yt'\to \Xt'$ be the $A$-model of the
universal torsor $\overline\Yt'\to\overline\Xt'$ defined in
Construction \ref{hypot3}.  Let $C_A$ and $C'_A$ be the ideals of
$A[\eta_1,\dots,\eta_N]$ generated by $f_1,\dots,f_m,g_1,\dots,g_s$,
and
$b^*(f'_1)\eta_1,\dots,b^*(f_{m'}')\eta_1,b^*(f'_1)\eta_2,\dots,b^*(f_{m'}')\eta_2,g_1,\dots,g_s$
respectively.  We assume that $\sqrt{C_A}=\sqrt{C'_A}$, that
\eqref{picardgeneration} holds for $(R;f_1,\dots,f_m)$ and
$(R';f'_1,\dots,f'_{m'})$, and that $(\eta_1,\eta_2)$ is a prime ideal
in $\cox(\overline\Xt')$.

\begin{prop}\label{under_blowing-up}
  Under the hypotheses listed above, $\Xt$ is a blowing-up of $\Xt'$
  with center the closed subscheme defined by $\eta_1,\eta_2$.
\end{prop}
\begin{proof}
  Let $f'\in\{f_1',\dots,f'_{m'}\}$ and $f:=b^*(f')$. Since
  $\pic\overline\Xt\cong\pic(\overline\Xt')\oplus\mathbb{Z}[E_N]$ and
  $\deg\eta_j=\deg b^*\eta_j -[E_N]$ for $j\in\{1,2\}$, the degrees of
  the homogeneous invertible elements of $R[(f\eta_j)^{-1}]$ generate
  $\pic(\overline\Xt)$ for $j\in\{1,2\}$. Hence,
  \eqref{picardgeneration} holds for
  $(R;b^*(f'_1)\eta_1,\dots,b^*(f_{m'}')\eta_1,b^*(f'_1)\eta_2,\dots,b^*(f_{m'}')\eta_2)$.
  Let $\Xt'_{f'}$ be the spectrum of the degree-0-part $R'_0$ of the
  ring $R'[{f'}^{-1}]$, and let $\Xt_{f\eta_j}$ be the spectrum of the
  degree-0-part $R[(f\eta_j)^{-1}]_0$ of the ring $R[(f\eta_j)^{-1}]$
  for $j\in\{1,2\}$.  Let $\overline\Xt'_{f'}$ be the complement in
  $\overline\Xt'$ of the support of the effective divisor
  corresponding to the section $f'$, analogously we define
  $\overline\Xt_{f}$ and $\overline\Xt_{f\eta_j}$ for $j\in\{1,2\}$.
  By \cite[Corollary 1.6.3.5]{cox},
  $\overline\Xt'_{f'}=\spec(R'_0\otimes_A\overline K)$.  Since
  $E_1\cap E_2=\emptyset$ in $\overline\Xt$,
  $\overline\Xt_{f}=\overline\Xt_{f\eta_1}\cup\overline\Xt_{f\eta_2}$.

  Let $h_1,h_2\in R'[f'^{-1}]^\times$ be homogeneous elements of
  degrees $-\deg\eta_1,-\deg\eta_2$ respectively. Then
  $(\eta_1h_1,\eta_2h_2)$ is the ideal of $R'_0\otimes_A\overline K$
  defining
  $\{x\}\cap\overline\Xt'_{f'}$.

  If $f'\in (\eta_1,\eta_2)$ in $R'$, then
  $x\notin\overline\Xt'_{f'}$, and $b$ induces an isomorphism between
  $\overline\Xt_{f}=b^{-1}(\overline\Xt'_{f'})$ and
  $\overline\Xt'_{f'}$. That is, $b^*$ induces an isomorphism between
  the degree-0-part of $\cox(\overline\Xt')[{f'}^{-1}]$ and the
  degree-0-part of $\cox(\overline\Xt)[f^{-1}]$ that descends to an
  isomorphism between $R'_0$ and the degree-0-part $R[f^{-1}]_0$ of
  $R[f^{-1}]$ with the quotient morphism as inverse.  Moreover,
  $\Xt_{f\eta_j}$ is the spectrum of the degree-0-part of
  $R[(f\eta_j\eta_N)^{-1}]$ for $j\in\{1,2\}$, as $f$ is a multiple of
  $\eta_N$ in $R$. Then
  $\Xt_{f\eta_1}\cup\Xt_{f\eta_2}=\spec(R[f^{-1}]_0)$, as
  $1\in(\eta_1\eta_Nb^*(h_1),\eta_2\eta_Nb^*(h_2))$ in $R[f^{-1}]_0$.

If $f'\notin (\eta_1,\eta_2)$ in $R'$, then
$x\in\overline\Xt'_{f'}$, and $\overline\Xt_f=b^{-1}(\overline\Xt'_{f'})$ is the
blowing-up of $\overline\Xt'_{f'}$ with center $x$.
The blowing-up of $\Xt'_{f'}$ with center $V(\eta_1h_1,\eta_2h_2)$ is
covered by two open subsets that are the spectra of the degree-0-parts
of the localizations of $\bigoplus_{d\geq0}(\eta_1h_1,\eta_2h_2)^d$ at
its degree-1-elements $\eta_jh_j$ for $j\in\{1,2\}$,
respectively. Such an open covering is isomorphic to the gluing of the
spectra of $R'_0[\eta_ih_i(\eta_jh_j)^{-1}]$, for $\{i,j\}=\{1,2\}$.
Since, for $\{i,j\}=\{1,2\}$, $b^*$ induces an isomorphism
$R'_0[\eta_ih_i(\eta_jh_j)^{-1}]\to R[(f\eta_j)^{-1}]_0$ with the
quotient morphism as inverse, the gluing of $\Xt_{f\eta_1}$ and
$\Xt_{f\eta_2}$ is the blowing-up of $\Xt'_{f'}$ with center
$V(\eta_1h_1,\eta_2h_2)$.

By Lemma \ref{gluinglemma}, the scheme $\Xt$ is isomorphic to the gluing of $\Xt_{b^*(f')\eta_j}$ for $f'\in\{f'_1,\dots,f'_{m'}\}$ and $j\in\{1,2\}$. Hence, it is a blowing-up of $\Xt'$ with center the closed subscheme  defined by $\eta_1,\eta_2$.
\end{proof}

\section{Parameterization of rational points on $S$}\label{sec:passage}

We recall that $S$ is the anticanonically embedded del Pezzo surface
of degree $4$ and type $\Athree+\Aone$ given by (\ref{eq:def_S}). Let
$\overline K$ be an algebraic closure of $K$, and $\tS_{\overline K}$
the minimal desingularization of $S_{\overline K}$ as in
\cite{math.AG/0604194}.

The aim of this section is to apply Theorem \ref{thm:param} to an
$\Ok$-model of a universal torsor of $\tS_{\overline K}$ obtained by
Construction \ref{hypot3} in order to get a parameterization of 
the set of $K$-rational points on the open subset $U$ complement of the lines in $S$
via integral points on twisted torsors. An elementary application of
the results in \cite{arXiv:1302.6151} would lead to the same
parameterization.

We start by describing the universal torsor of $\tS_{\overline K}$
inside the spectrum of its Cox ring.  By the data provided in \cite[\S
3.4]{math.AG/0604194}, $\tS_{\overline K}$ is a blowing-up of
$\PP^2_{\overline K}$ in five points in almost general position with
Picard group $\pic(\tS_{\overline K})\cong\Z^6\cong N^1(\tS_{\overline
  K})$, and the Cox ring of $\tS_{\overline K}$ is a
$\pic(\tS_{\overline K})$-graded $\overline K$-algebra with nine
generators and one homogeneous relation:
\begin{equation*}
\cox(\tS_{\overline K})=\overline K[\eta_1,\dots,\eta_9]/(\eta_1\eta_9+\eta_2\eta_8+\eta_3\eta_4^2\eta_5^3\eta_7).
\end{equation*}
 For $i\in\{1,\dots,9\}$, the degree of $\eta_i$ is
$[E_i]\in\pic(\tS_{\overline K})$, 
where  $[E_i]$ are the divisor classes listed below. 
Let $\ell_0,\dots,\ell_5$ be the basis of $\pic(\tS_{\overline K})$ given in \cite{math.AG/0604194}. Then  the intersection form is defined by $\ell_0^2=1$, $\ell_i^2=-1$ for $1\leq i\leq 5$, and $\ell_i.\ell_j=0$ for all $0\leq i<j\leq5$. The classes

\begin{equation*}
 [E_1]=\ell_5,\quad
 [E_2]=\ell_4,\quad
 [E_5]=\ell_3
\end{equation*}
are the $(-1)$-curves on $\tS_{\overline K}$,

\begin{equation*}
[E_3]=\ell_1-\ell_2, \quad 
[E_4]=\ell_2-\ell_3,\quad
[E_6]=\ell_0-\ell_1-\ell_4-\ell_5,\quad
[E_7]=\ell_0-\ell_1-\ell_2-\ell_3
\end{equation*}
are the $(-2)$-curves on $\tS_{\overline K}$, and
 \begin{equation*}
[E_8]=\ell_0-\ell_4,\quad
[E_9]=\ell_0-\ell_5.
\end{equation*}
 
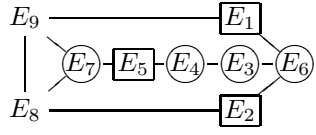
\begin{figure}[ht]
  \centering
  \[\xymatrix@R=0.05in @C=0.05in{E_9 \ar@{-}[dr] \ar@{-}[dd] \ar@{-}[rrrr]& & & & \li{1} \ar@{-}[dr]\\
    & \ex{7} \ar@{-}[r] & \li{5} \ar@{-}[r] & \ex{4} \ar@{-}[r] & \ex{3} \ar@{-}[r] & \ex{6}\\
    E_8 \ar@{-}[ur] \ar@{-}[rrrr] & & & & \li{2} \ar@{-}[ur]}\]
  \caption{Configuration of curves on $\tS_{\overline K}$.}
  \label{fig:A3+A1_dynkin}
\end{figure}

The Dynkin diagram in Figure \ref{fig:A3+A1_dynkin} encodes the configuration of curves on $\tS_{\overline K}$. For any $i\neq j$ the number of edges between $E_i$ and $E_j$ is the intersection number $[E_i].[E_j]$. 

To construct an $\Ok$-model of the universal torsor $\overline\tT\to\tS_{\overline K}$ which is a universal torsor over a projective $\Ok$-model of $\tS_{\overline K}$, we consider the following monomials. For all $1\leq i<j\leq9$, let $\mA_{i,j}:=\prod_{l\in\{1,\dots,9\}\smallsetminus\{i,j\}}\eta_l$, and  $\mA_{7,8,9}:=\eta_1\eta_2\eta_3\eta_4\eta_5\eta_6$.
Let $J$ be the ideal of $\cox(\tS_{\overline K})$ generated by the following monomials:
\begin{equation}\label{firstgenerators}
\mA_{7,8,9}, \quad \mA_{1,6},\quad  \mA_{1,9}, \quad \mA_{2,6},\quad  \mA_{2,8},\quad  \mA_{3,4},\quad \mA_{3,6},\quad  \mA_{4,5}, \quad \mA_{5,7},
\end{equation}
which are obtained from the Dynkin diagram in Figure \ref{fig:A3+A1_dynkin} by considering the subsets of vertices that are pairwise connected by at least one edge.

Since $E_7\cap E_8\cap E_9\neq\emptyset$ by \cite{math.AG/0604194}, the open subscheme $\overline\tT$ complement to $V(J)$ in $\spec(\cox(\tS_{\overline K}))$ is a universal torsor of $\tS_{\overline K}$ by \cite[Remark 6]{MR2783385}.

We denote by $f_1,\dots,f_9$  the monomials in \eqref{firstgenerators}.
Let
\begin{equation*}
R:=\Ok[\eta_1,\dots,\eta_9]/(\eta_1\eta_9+\eta_2\eta_8+\eta_3\eta_4^2\eta_5^3\eta_7)
\end{equation*}
and let $\TZ\to\tSZ$ be the $\Ok$-model of the universal torsor
$\overline\tT\to\tS_{\overline K}$ defined by $f_1,\dots,f_9$ in
Construction \ref{hypot3}.  Some properties of this model are
described in the following proposition, which is an application of the
results of Section \ref{sec:universal_torsors}.

\begin{prop}\label{model_properties}
 \begin{enumerate}[label=(\roman{*}), ref=\emph{(\roman{*})}]
 \item The scheme $\tSZ$ is smooth, projective, and with geometrically
   integral fibers over $\Ok$.\label{smooth_geom_integral}
 \item For every prime ideal $\p$ of $\Ok$, the fibre $\tSZ_{k(\p)}$
   is obtained from $\PP^2_{k(\p)}$ by a chain of $5$ blowing-ups at
   $k(\p)$-points.\label{splitness}
 \item The morphism $\TZ\to\tSZ$ is a universal torsor under
   $\G_{m,\tSZ}^6$.
 \end{enumerate}
\end{prop}
\begin{proof}

We start by proving that the model $\tSZ$ is obtained from $\PP^2_{\Ok}$ by a chain of five blowing-ups.
By the data provided in \cite{math.AG/0604194}, $\tS_{\overline K}$ is
a blowing-up of a split toric ${\overline K}$-variety $S_{\overline
  K}'$ at a closed point and with exceptional divisor corresponding to
the section $\eta_1\in\cox(\tS_{\overline K})$. The center of the
blowing-up $b:\tS_{\overline K}\to S_{\overline K}'$ is the
intersection of the prime divisors of $S'_{\overline K}$ corresponding
to the sections $\eta_6,\eta_9\in\cox(S'_{\overline K})$ under the
identification
  \begin{equation*}
  \cox(S'_{\overline K})\cong \cox(\tS_{\overline K})/(\eta_1-1)\cong \overline K[\eta_2,\dots,\eta_9]/(\eta_9+\eta_2\eta_8+\eta_3\eta_4^2\eta_5^3\eta_7)
   \end{equation*}
   provided by \cite[Proposition 2.2]{computing_Cox_rings}. The rays
   of the fan $\Delta$ defining $S_{\overline K}'$ correspond to the
   generators $\eta_2,\dots,\eta_8$ of $\cox(S'_{\overline K})$. We
   denote them by $\rho_2,\dots,\rho_8$ (see Figure \ref{fig:fan}).
\begin{figure}[ht]
 \begin{tikzpicture}
    \draw[step=1cm,gray,thin,dotted] (-1,-1) grid (2,1);
    \begin{scope}[thick]
      \draw(0,0) -- (1,0) node[anchor=west] {$\rho_7$};
      \draw(0,0) -- (0,1) node[anchor=west] {$\rho_8$};
      \draw(0,0) -- (-1,0) node[anchor=east] {$\rho_2$};
      \draw(0,0) -- (-1,-1) node[anchor=east] {$\rho_6$};
      \draw(0,0) -- (0,-1) node[anchor=west] {$\rho_3$};
      \draw(0,0) -- (1,-1) node[anchor=west] {$\rho_4$};
      \draw(0,0) -- (2,-1) node[anchor=west] {$\rho_5$};
    \end{scope}
  \end{tikzpicture}
\caption{The fan $\Delta$.}
  \label{fig:fan}
\end{figure}
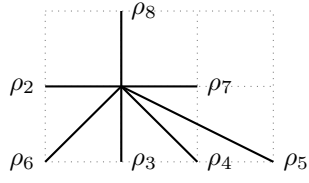
Let $\SZ'$ be the $\Ok$-toric scheme defined by $\Delta$ as in \cite[Remarks 8.6 (b)]{MR1679841}, and
\begin{equation*}
R':=\Ok[\eta_2,\dots,\eta_9]/(\eta_9+\eta_2\eta_8+\eta_3\eta_4^2\eta_5^3\eta_7)\cong\Ok[\eta_2,\dots,\eta_8].
\end{equation*}
The fan $\Delta$ has 7 maximal cones. For $1\leq i\leq 7$, let $f_i'$
be the product $\prod\eta_j$ running over the indices
$j\in\{2,\dots,8\}$ such that the ray $\rho_j$ does not belong to the
$i$-th maximal cone. 
By \cite[\S8]{MR1679841}, the monomials $f_1',\dots,f_7'$ define the
complement of the universal torsor of $S'_{\overline K}$ contained in
$\spec(\cox(S'_{\overline K}))$. 

For every $i\in\{1,\dots,7\}$, the open affine toric subvariety of
$S'_{\overline K}$ corresponding to the $i$-th maximal cone has
trivial Picard group, and its complement consists of the effective
divisor defined by the section $f_i'$. Hence, $(R';f'_1,\dots,f'_7)$
satisfies \eqref{picardgeneration} by Remark
\ref{geometric_interpretation}, and $\SZ'$ is the $\Ok$-model of $S'$
defined by Construction \ref{hypot3}.

Recall the notation before Proposition \ref{under_blowing-up}. Since
the radical of the ideal of $\Ok[\eta_1,\dots,\eta_9]$ generated by
$f_1,\dots,f_9$ is the radical of the ideal generated by
${b}^*(f'_1)\eta_6,\dots,{b}^*(f'_7)\eta_6,{b}^*(f'_1)\eta_9,\dots,{b}^*(f'_7)\eta_9$,
the model $\tSZ$ is a blowing-up of $\SZ'$ with center the closed
subscheme defined by $\eta_6,\eta_9$ by Proposition
\ref{under_blowing-up}.
We observe that  $\SZ'$ is obtained from $\PP^2_{\Ok}$ by a chain of four toric blowing-ups, which correspond to adding the rays $\rho_2,\rho_3,\rho_4,\rho_5$, respectively, to the fan of $\PP^2_{\Ok}$ with rays $\rho_6,\rho_7,\rho_8$. Hence, the model $\tSZ$ is 
obtained from $\PP^2_{\Ok}$ by a chain of five blowing-ups, and it is projective (cf.~\cite[Proposition 13.96]{MR2675155}.

To prove \ref{splitness}, let $\p$ be a prime ideal of $\Ok$, and
$\overline{k(\p)}$ an algebraic closure of the residue field $k(\p)$.
Since the closed subscheme of $\SZ'$ defined by $\eta_6,\eta_9$ 
is flat over $\OO_K$, the variety
$\tSZ_{k(\p)}$ is the blowing-up of $\SZ'_{k(\p)}$ in the
$k(\p)$-point defined by $\eta_6, \eta_9$. Moreover, $\SZ'_{k(\p)}$ is
the split toric $k(\p)$-variety defined by $\Delta$, which is obtained
from $\PP^2_{k(\p)}$ by four toric blowing-ups at
$k(\p)$-points. Therefore,
$\SZ$ has geometrically integral fibers over $\Ok$, and 
$H^i(\tSZ_{k(\p)},\mathcal{O}_{\tSZ_{k(\p)}})=0$ for $i\in\{1,2\}$ by
\cite[Proposition V.3.4]{MR0463157} and
\cite[p.~74]{MR1234037}.

  Simple computations show that the degrees of the variables $\eta_j$
  appearing in $f_i$ generate $\pic(\tS_{\overline K})$ for all
  $i\in\{1,\dots,9\}$. Since these $\eta_j$ are invertible in
  $R[{f_i}^{-1}]$, the condition \eqref{picardgeneration} holds for
  $(R;f_1,\dots,f_9)$.

  The Jacobian matrix $(\partial g/\partial\eta_i)_{1\leq i\leq 9}$ is
\begin{equation*}
(\eta_9, \ \eta_8, \ \eta_4^2\eta_5^3\eta_7, \ 2\eta_3\eta_4\eta_5^3\eta_7, \ 3\eta_3\eta_4^2\eta_5^2\eta_7, \ 0, \ \eta_3\eta_4^2\eta_5^3, \ \eta_2, \ \eta_1),
\end{equation*}
and has rank 1 on $\TZ(\overline{k(\p)})$ because the monomials
$f_1,\dots,f_9$ belong to the ideal generated by $\eta_1,\eta_2$. Then
$\tSZ$ is smooth by Proposition \ref{smoothnesscriterion}. Hence, all the hypotheses of Theorem \ref{universaltorsor} are satisfied.
\end{proof}

\begin{Remark}
The projectiveness and integrality of the fibers of the $\Ok$-model $\tSZ$ can be verified also using Propositions \ref{geomintegralfibers} and \ref{tors_projectivity}, which are not restricted to the case of surfaces. To show  a concrete application, we verify their hypotheses for our surface $\tS$.
Since $g:=\eta_1\eta_9+\eta_2\eta_8+\eta_3\eta_4^2\eta_5^3\eta_7$ is
  irreducible in $\overline{k(\p)}[\eta_1,\dots,\eta_9]$ for all prime
  ideals $\p$ of $\Ok$, the $\Ok$-scheme $\tSZ$ has geometrically
  integral fibers by Proposition \ref{geomintegralfibers}. 
To verify the hypotheses of Proposition \ref{tors_projectivity}, we
define $C'_{\overline K}$ and $C'_{\Ok}$ as the ideals of $\overline
K[\eta_1,\dots,\eta_9]$ and $\Ok[\eta_1,\dots,\eta_9]$, respectively, generated by
 $f_1,\dots,f_9$ and $g$. One can check
that $C'_{\overline K}$ has a Gr\"obner basis
$\{h_1,\dots,h_l\}\subseteq C'_{\Ok}$ consisting of polynomials whose
coefficients are all equal to $1$. This implies that $C'_{\overline
  K}\cap\Ok[\eta_1,\dots,\eta_9]=(h_1,\dots,h_l)=C'_{\Ok}$. 
According to  \cite[\S3.4]{math.AG/0604194}, the surface $\tS_{\overline K}$ is 
a blowing-up of $\PP^2_{\overline K}$ in five points. Such a description of $\tS_{\overline K}$ 
allows us to determine the irreducible curves and the intersection pairing on 
$\tS_{\overline K}$ (see \cite[\S V.3]{MR0463157}),  and to show that the divisor class 
\begin{equation*}
[D]:=9\ell_0-3\ell_1-2\ell_2-\ell_3-\ell_4-\ell_5 
\end{equation*}
is ample by the Nakai-Moishezon criterion.
Let $C_{\overline K}$ and $C_{\Ok}$ be the ideals of $\overline
K[\eta_1,\dots,\eta_9]$ and $\Ok[\eta_1,\dots,\eta_9]$, respectively, generated by $g$ and by the following monomials of degree $[D]$: 
\begin{gather}
\eta_1^8\eta_2^8\eta_3^6\eta_4^4\eta_5^3\eta_6^9, 
\eta_2^3\eta_3\eta_4^3\eta_5^6\eta_7^4\eta_8^4\eta_9, 
\eta_2^5\eta_3\eta_4^2\eta_5^4\eta_6\eta_7^3\eta_8^5, 
\eta_1^3\eta_3\eta_4^3\eta_5^6\eta_7^4\eta_8\eta_9^4,
\eta_1^5\eta_3\eta_4^2\eta_5^4\eta_6\eta_7^3\eta_9^5,\nonumber\\
\eta_1^5\eta_2\eta_5\eta_6\eta_7^2\eta_8\eta_9^5,
\eta_1^2\eta_2^2\eta_4\eta_5^3\eta_7^3\eta_8^3\eta_9^3,
\eta_1^6\eta_2^3\eta_3\eta_6^3\eta_7\eta_8\eta_9^4,
\eta_1^7\eta_2^6\eta_3^3\eta_4\eta_6^6\eta_8\eta_9^2.\label{monomials_second_generators}
\end{gather} 
We observe that these monomials are obtained from $f_1,\dots,f_9$ by increasing the exponents of some variables. Hence, the $\Ok$-model of the universal torsor
$\overline\tT\to\tS_{\overline K}$ defined by the monomials \eqref{monomials_second_generators}  in
Construction \ref{hypot3} is the same as the one defined by $f_1,\dots,f_9$, namely $\tSZ$, and 
the radical of the ideal $C_{\overline K}$ (resp. $C_{\Ok}$) 
coincides with the radical of $C'_{\overline K}$ (resp. of
$C'_{\Ok}$). Thus, $\tSZ$ is projective over $\Ok$ by Proposition \ref{tors_projectivity}.
\end{Remark}

The action of $\G_{m,\Ok}^6(\Ok)\cong(\Ok^\times)^6$ on $\TZ(\Ok)$ is
given by (\ref{action_on_points}), where $m^{(1)}$, $\dots$,
$m^{(9)}\in \mathbb{Z}^6$ denote the degrees of $\eta_1,\dots,\eta_9$,
respectively, under the identification $\pic(\tS_{\overline
  K})\cong\mathbb{Z}^6$ provided by the basis
$\ell_0,\dots,\ell_5$. Namely,
\begin{align*}
  m^{(1)}&=(0,0,0,0,0,1),& m^{(2)}&=(0,0,0,0,1,0),& m^{(3)}&=(0,1,-1,0,0,0),\\
  m^{(4)}&=(0,0,1,-1,0,0),& m^{(5)}&=(0,0,0,1,0,0),& m^{(6)}&=(1,-1,0,0,-1,-1),\\
  m^{(7)}&=(1,-1,-1,-1,0,0),&
  m^{(8)}&=(1,0,0,0,-1,0),&m^{(9)}&=(1,0,0,0,0,-1).
\end{align*}

Before we apply Theorem \ref{thm:param} to obtain a parameterization
of $U(K)$ by integral points on twists of $\TZ$, we describe the
preimage of $U$ inside the universal torsor, and we fix some more
notation.

Let $\tS:=\tSZ_K$, $\tT:=\TZ_K$, and $\pi:\tT\to\tS$ the base change
of the torsor morphism under the inclusion $\Ok\subseteq K$. We
observe that $\pi$ is a universal torsor of $\tS$ by Remark
\ref{udesck}.

Let $\overline\Psi:\overline\tT\to S_{\overline K}$ be the composition
of the universal torsor morphism $\overline\tT\to\tS_{\overline K}$
with the minimal desingularization morphism $\tS_{\overline K}\to
S_{\overline K}$.  According to \cite[\S 3.4]{math.AG/0604194}, the
map $\overline\Psi:\overline\tT(\overline K)\to S_{\overline
  K}(\overline K)$ sends a point $(\coord_1, \ldots,
\coord_9)\in\overline\tT(\overline K)$ to the point
  \begin{equation}\label{Psi}
    (\coord_2\coord_3\coord_4\coord_5\coord_6\coord_7\coord_8: \coord_1^2\coord_2^2\coord_3^2\coord_4\coord_6^3 : \coord_1\coord_2\coord_3^2\coord_4^2\coord_5^2\coord_6^2\coord_7 :\coord_1\coord_3\coord_4\coord_5\coord_6\coord_7\coord_9: \coord_7\coord_8\coord_9)
  \end{equation}
  in $S_{\overline K}(\overline K)\subseteq \PP^4(\overline K)$.
  Since $\overline\Psi$ is defined over $K$, it induces a morphism
  $\Psi:\tT\to S\subseteq\mathbb{P}^4_K$ which is given by (\ref{Psi})
  on $K$-rational points.

  Since $\pi:\tT\to\tS$ is a geometric quotient, the invariant
  morphism $\Psi$ factors through a minimal desingularization
  $\gamma:\tS\to S$, which is a model of the minimal desingularization
  $\tS_{\overline K}\to S_{\overline K}$.
  
  We recall that $U$ is defined as the complement of the lines in
  $S$. By \cite[Table 6]{math.AG/0604194}, the surface $S_{\overline
    K}$ contains exactly three lines of $\PP^4_{\overline K}$. These
  are defined over $K$ and an easy computation shows that
\begin{equation*}
\lines=\{x_0x_1=x_0x_3=x_1x_3=x_2=0\}.
\end{equation*} 
Then $\Psi^{-1}(S\smallsetminus U)=\{\eta_1\eta_2\eta_3\eta_4\eta_5\eta_6\eta_7=0\}$, and
  \begin{equation}\label{eq:inverse_image_U}
  \Psi^{-1}(U(K))=\tT(K)\cap (K^\times)^7\times K^2.\end{equation}

  From now on, $\classrepsyst$ refers to a fixed system of integral
  representatives for $\cl$, that is, it contains exactly one integral
  ideal from each class. For any given $\classtuple = (\classrep_0,
  \dots, \classrep_5) \in\classrepsyst^6$, we denote by
  $\cpi:\cT\to\tSZ$ the twist of $\TZ$ constructed as in Definition
  \ref{defin:twisted_torsors}. We write 
  \begin{gather*}
  u_\classtuple := \N(\classrep_0^3
  \classrep_1^{-1}\cdots \classrep_5^{-1}),\qquad
  \coordfracideal_j:=\classtuple^{m^{(j)}}\text{ for }1\leq j\leq 9,\\
  \text{and}\quad
  \coordfracideal_{j*} :=
  \begin{cases}
    \coordfracideal_j^{\neq 0} &\text{ if } j \in \{1,\ldots, 7\},\\
    \coordfracideal_j &\text{ if } j \in \{8,9\}.
  \end{cases}
\end{gather*}
For $\coord_j \in \coordfracideal_j$, let
\begin{equation*}
  \coordideal_j := \coord_j \coordfracideal_j^{-1}\text.
\end{equation*}

For $v \in \places$ and $(x_1, \ldots, x_8) \in K_v^8$ with $x_1 \neq
0$, we write
\begin{equation*}
  \tN_v(x_1, \ldots, x_8) := \max\left\{
    \begin{aligned}
      &\absv{x_2x_3x_4x_5x_6x_7x_8},\absv{x_1^2x_2^2x_3^2x_4x_6^3},\\
      &\absv{x_1x_2x_3^2x_4^2x_5^2x_6^2x_7},\\
      &\absv{x_3x_4x_5x_6x_7(x_3x_4^2x_5^3x_7 + x_2x_8)},\\
      &\absv{\frac{x_2x_7x_8^2 + x_3x_4^2x_5^3x_7^2x_8}{x_1}}
    \end{aligned}\right\}.
\end{equation*}
Let $\FDtot$ be a fundamental domain for the action
\begin{equation*}
  \text{of}\quad U_K\times(\units)^5 \quad\text{ on }\quad (K^\times)^7\times K^2,
\end{equation*}
induced by (\ref{action_on_points}), where $\underline u=(u_0, \ldots, u_5)$ maps $(\coord_1, \ldots, \coord_9)$ to
\begin{equation}\label{eq:unit_action}
  \begin{aligned}
   (\underline u^{m^{(1)}}\cdot\coord_1,\dots, \underline u^{m^{(9)}}\cdot\coord_9).
  \end{aligned}
\end{equation}
After all these preparations, we define $M_\classtuple(B)$ as
the set of all
\begin{equation*}
  (\coord_1, \ldots, \coord_9) \in (\coordfracideal_{1*} \times \cdots \times \coordfracideal_{9*}) \cap \FDtot
\end{equation*}
that satisfy the \emph{height conditions}
\begin{equation}\label{eq:A3+A1_height}
  \prod_{v\in\archplaces}\tN_v(\sigma_v(\coord_1, \ldots, \coord_8)) \leq u_\classtuple B,
\end{equation}
the \emph{torsor equation}
\begin{equation}\label{eq:A3+A1_torsor}
   \coord_1\coord_9 + \coord_2\coord_8 + \coord_3\coord_4^2\coord_5^3\coord_7 = 0,
\end{equation}
and the \emph{coprimality conditions}
\begin{equation}\label{eq:A3+A1_coprimality}
  \coordideal_j + \coordideal_k = \integersk \text{ for all distinct nonadjacent vertices $E_j$, $E_k$ in Figure~\ref{fig:A3+A1_dynkin}.}
\end{equation}
We can now reduce our counting problem to counting the
$M_\classtuple(B)$.
\begin{lemma}\label{lem:A3+A1_passage_to_torsor}
  With the sets $M_\classtuple(B)$ defined as above and $N_{U,H}(B)$ as in Theorem \ref{maintheorem}, we have
  \begin{equation*}
    N_{U,H}(B) = \frac{1}{|\mu_K|}\sum_{\classtuple \in \classrepsyst^6}|M_\classtuple(B)|.
  \end{equation*}
\end{lemma}

\begin{proof}
  Since $U$ is contained in the smooth locus of $S$, the minimal
  desingularization $\gamma :\tS\to S$ induces an isomorphism
  $\gamma^{-1}(U) \to U$, so
  \begin{equation*} N_{U, H}(B) = |\{x \in \gamma^{-1}(U)(K) \mid
    H(\gamma(x)) \leq B\}|.
\end{equation*}

By Theorem \ref{thm:param}, \ref{thm:param_union}, there is a disjoint union
\begin{equation*}
\gamma^{-1}(U)(K)=\bigsqcup_{\classtuple\in\classrepsyst^6}\cpi(\cT(\Ok)\cap\Psi^{-1}(U(K))).
\end{equation*}

Let $\classtuple\in\classrepsyst^6$. By \eqref{eq:inverse_image_U} and Theorem \ref{thm:param}, \ref{thm:param_coord}, we see that
 $\cT(\Ok)\cap\Psi^{-1}(U(K))$ is the set of all 
\begin{equation*}
(\coord_1, \ldots, \coord_9) \in (\coordfracideal_{1*} \times \cdots \times \coordfracideal_{9*})
\end{equation*}
that satisfy (\ref{eq:A3+A1_torsor}) and
\begin{equation}\label{eq:coprim_with_f_i}
 \sum_{i=1}^9f_i(\coordidealtuple) = \Ok.  
\end{equation}
By $f_i(\coordidealtuple)$, we mean the ideal $\coordideal_1^{e_1}\cdots\coordideal_9^{e_9}$, if $f_i$ is the monomial $\eta_1^{e_1}\cdots\eta_9^{e_9}$.
Let us show that \eqref{eq:coprim_with_f_i} is equivalent to the coprimality conditions \eqref{eq:A3+A1_coprimality}. These are certainly equivalent to
\begin{equation}\label{eq:coprim_product}
\prod_{\substack{1\leq i<j\leq9\\ [E_i].[E_j]=0}}(\coordideal_i+\coordideal_j)=\Ok.
\end{equation}
The ideal
\begin{equation*}
I=\prod_{\substack{1\leq i<j\leq9\\ [E_i].[E_j]=0}}(\eta_i,\eta_j)
\end{equation*}
in $K[\eta_1, \ldots, \eta_9]$ is generated by $t:=2^{25}$ monic monomials $h_1, \ldots, h_t$, obtained by choosing for each of the $25$ factors $(\eta_i,\eta_j)$ one of the generators $\eta_i,\eta_j$ and multiplying them all.  Due to the distributive property of ideals, condition \eqref{eq:coprim_product} is equivalent to 
\begin{equation}\label{eq:coprim_with_h_i}
  \sum_{i=1}^th_i(\coordidealtuple) = \Ok.
\end{equation}

The radical of $I$ is generated by $f_1, \ldots, f_9$. Since the $f_i$ and $h_i$ are monic monomials, this implies that for each $h_i$ there is an $f_j$ such that a power $h_i^{m}$ is divisible by $f_j$, and vice versa. Therefore, \eqref{eq:coprim_with_h_i} is equivalent to \eqref{eq:coprim_with_f_i}.

Next, we consider the height condition. Let $(\coord_1, \ldots, \coord_9)\in\cT(\Ok)$. Using
  the torsor equation \eqref{eq:A3+A1_torsor} to eliminate $\coord_9$,
  we see that
  \begin{equation*}
    H(\Psi(\coord_1, \ldots, \coord_9)) = \prod_{v \in \places}\tN_v(\sigma_v(\coord_1, \ldots, \coord_8)).
  \end{equation*}
  Moreover, the coprimality conditions \eqref{eq:A3+A1_coprimality}
  imply that
  \begin{equation*}
    \prod_{v \in \nonarchplaces}\tN_v(\sigma_v(\coord_1, \ldots, \coord_8)) = \N(\coord_2\coord_3\coord_4\coord_5\coord_6\coord_7\coord_8\integersk + \cdots + \coord_7\coord_8\coord_9\integersk)^{-1} = u_\classtuple^{-1}. 
  \end{equation*}
Thus the condition $H(\Psi(\coord_1, \ldots, \coord_9))\leq B$ is
  equivalent to our height conditions \eqref{eq:A3+A1_height}.
  
Let $\FDtot'\subseteq\FDtot$ be a fundamental domain for the action of $\G^6_{m,\Ok}(\Ok)=(\Ok^\times)^6$ on $(K^\times)^7\times K^2$ and consider the set $M'_\classtuple(B)$ of all
\begin{equation*}
  (\coord_1, \ldots, \coord_9) \in (\coordfracideal_{1*} \times \cdots \times \coordfracideal_{9*}) \cap \FDtot'
\end{equation*}
that satisfy (\ref{eq:A3+A1_height}), (\ref{eq:A3+A1_torsor}) and (\ref{eq:A3+A1_coprimality}).

Since the action of $(\units)^6$ on $\cT(\Ok)$ is free, each
  orbit is the union of $|\mu_K|$ orbits of the induced action of $U_K
  \times (\units)^5$. Each of these orbits has exactly one
  representative in $\FDtot$, so $|M_\classtuple(B)|=|\mu_K|\cdot|M'_\classtuple(B)|$. 
  
Finally, we observe that the fibers of $\cpi$ are the orbits of the action of $\G_{m,\tSZ}^6$ on $\cT$. Hence, there is a bijection between the sets $U(K)$ and $\bigsqcup_{\classtuple\in\classrepsyst^6}(\cT(\Ok)\cap\Psi^{-1}(U(K))\cap\FDtot')$, so
\begin{equation*}
N_{U,H}(B) = \sum_{\classtuple \in \classrepsyst^6}|M'_\classtuple(B)|=\frac{1}{|\mu_K|}\sum_{\classtuple \in \classrepsyst^6}|M_\classtuple(B)|.\qedhere
\end{equation*}
\end{proof}

\begin{Remark}\label{rem:other_cases}
  Our surface $S$ is one of $30$ types of del Pezzo surfaces (over
  $\bar{\QQ}$) whose universal torsors are hypersurfaces, classified
  in \cite{math.AG/0604194}. It is not hard to adapt our arguments to
  the remaining 29 cases. Indeed, in each case we checked that
  applying Construction \ref{hypot3} to the universal torsor given by
  the monomials obtained from the Dynkin diagram (like the monomials
  \eqref{firstgenerators} for $S$) provides us with $\Ok$-models of
  the surface and the universal torsor, for which the analogues of
  Proposition \ref{model_properties} and Lemma
  \ref{lem:A3+A1_passage_to_torsor} hold. That is, these monomials
  satisfy condition \eqref{picardgeneration}, the equation defining
  the Cox ring as a hypersurface is irreducible modulo every prime,
  and the hypotheses of Proposition \ref{smoothnesscriterion} and of
  Proposition \ref{under_blowing-up} are satisfied. Regarding the last
  one, we recall that blowing-down a (-1)-curve on a weak del Pezzo
  surface produces another weak del Pezzo surface of smaller
  degree. Hence, we checked the hypotheses of Proposition
  \ref{under_blowing-up} going down step by step for the chain of
  blowing-ups of $\PP^2$ that define the del
  Pezzo surfaces in \cite{math.AG/0604194}.
 
The verification of these facts in each case requires straightforward but lengthy computations entirely analogous to the arguments from this section. Since they are not needed for our main result, we omit the details here.

During the verification of our claims, we found the following two misprints in \cite{math.AG/0604194}: in the description of the del Pezzo surface of degree 5 of type $\Atwo$ at page 656 the class of the curve $E_1$ is $\bar E_1=\ell_0-\ell_1-\ell_3-\ell_4$, and in the description of the del Pezzo surface of degree 2 of type $\Dfive+\Aone$ at page 671 the curves 
$E_5$ and $E_6$ must be exchanged in the second and third line.
\end{Remark}

\section{Construction of a fundamental domain}\label{sec:fundamental_domain}
In this section, we choose our fundamental domain $\FDtot$ for the
action \eqref{eq:unit_action}. Our main objective is to find a
fundamental domain that lends itself well to lattice point
counting. In a much simpler case, such a fundamental domain was
constructed by Schanuel \cite{MR557080}.  Our notation is inspired by
\cite{MR2247898}.

Let $\Sigma$ be the hyperplane in $\RR^\archplaces$ where the sum of
the coordinates vanishes, and $\delta := (d_v)_{v\in\archplaces}\in
\RR^\archplaces$. By Dirichlet's unit theorem, the usual logarithmic
embedding $l : \units \to \Sigma$,
\begin{equation*}
  l(u) := (\log\absv{\sigma_v(u)})_{v\in\archplaces},
\end{equation*}
has $\mu_K$ as its kernel and a lattice $l(\units) = l(U_K)$ in
$\Sigma$ as its image.

Let us fix, once and for all, a fundamental parallelotope $F$ for this
lattice and denote the vector sum $F + \RR\delta$ by $F(\infty)$. Then
\begin{equation}\label{eq:fundamental_domain_embedded}
  \text{$F(\infty)$ is a fundamental domain for the additive action of
    $l(U_K)$ on $\RR^\archplaces$.}
\end{equation}
For fixed $\coordtuple' := (\coord_1, \ldots, \coord_5)\in (K^\times)^5$, we define $\tN_v(\coordtuple'; \cdot) :
(K_v^\times)^2 \times K_v \to (0, \infty)$ by
\begin{equation*}
  \tN_v(\coordtuple'; x_6, x_7, x_8) := \tN_v(\coord_1^{(v)}, \ldots , \coord_5^{(v)}, x_{6}, x_{7}, x_{8}).  
\end{equation*}
Let $S_F(\coordtuple'; \infty)$ be the set of all
\begin{equation*}
  (x_{jv})_{\substack{j \in \{6,7,8\}\\v\in \archplaces}} \in \prod_{v\in\archplaces}((K_v^\times)^2\times K_v)
\end{equation*}
such that
\begin{equation*}
  \frac{1}{3}\cdot(\log\tN_v(\coordtuple'; x_{6v}, x_{7v}, x_{8v}))_{v\in\archplaces} \in F(\infty).
\end{equation*}
Since all terms of the maximum in $\tN_v$ are homogeneous of degree
$3$ in $x_6, x_7, x_8$, the relation
\begin{equation*}
  (\log\tN_v(\coordtuple'; u^{(v)}\cdot(x_{6v},x_{7v},x_{8v})))_{v \in \archplaces} = 3l(u) + (\log \tN_v(\coordtuple'; x_{6v},x_{7v},x_{8v}))_{v \in \archplaces}
\end{equation*}
holds for all $u \in U_K$. Due to this and
\eqref{eq:fundamental_domain_embedded}, the set
\begin{equation*}
  \FDxi(\coordtuple') := \{(\coord_6,\coord_7,\coord_8) \in (K^\times)^2 \times K \mid \sigma(\coord_6,\coord_7,\coord_8) \in S_F(\coordtuple'; \infty)\}
\end{equation*}
is a fundamental domain for the action of $U_K$ on $(K^\times)^2
\times K$ by scalar multiplication.

Let $\FDei$ be a fundamental domain for the multiplicative action of
$\units$ on $K^\times$, chosen in such a way that
\begin{equation}\label{eq:conj_norm}
  N(a)^{d_v/d}\ll
  \absv{a} \ll N(a)^{d_v/d}
\end{equation}
holds for all $a \in \FDei$ and all $v \in \archplaces$. By ignoring the last coordinate, the action described by \eqref{eq:unit_action} induces an action of $U_K\times(\units)^5$ on $(K^\times)^7\times K$. Basic linear algebra with the exponents $m^{(1)}, \ldots, m^{(8)}$ shows that this action is free and has the fundamental domain
\begin{equation*}
  \FDtot' := \bigg\{(\coord_1, \ldots, \coord_8) \in (K^\times)^7\times K \WHERE
  \begin{aligned}
    &\coordtuple' \in \FDei^5,\\
    &(\coord_6,\coord_7,\coord_8) \in \FDxi(\coordtuple')
  \end{aligned}
  \bigg\}.
\end{equation*}
Therefore, we may choose
\begin{equation*}
  \FDtot := \FDtot' \times K
\end{equation*}
as our fundamental domain for the action of $U_K\times(\units)^5$ on $(K^\times)^7\times K^2$.

The main advantage of this fundamental domain is that it allows a
natural incorporation of the height conditions
\eqref{eq:A3+A1_height}. Indeed, let
\begin{equation*}
  F(B) := F + (-\infty, \log B]\cdot \delta.
\end{equation*}
It follows immediately from the definitions that a tuple
$(x_{jv})_{j,v} \in S_F(\coordtuple', \infty)$ satisfies
\begin{equation*}
  \prod_{v\in\archplaces}\tN_v(\coordtuple'; x_{6v},x_{7v},x_{8v}) \leq B
\end{equation*}
if and only if it is an element of the subset
\begin{align*}
  S_F(\coordtuple'; B) :=\left\{(x_{jv})_{j,v}\ \Big|\
    \frac{1}{3}\cdot(\log\tN_v(\coordtuple';x_{6v}, x_{7v},
    x_{8v}))_{v\in \archplaces}\in F(B^{1/(3d)})\right\}
\end{align*}
of $\prod_{v\in\archplaces}((K_v^\times)^2\times K_v)$. Let
$\FDxi(\coordtuple'; u_\classtuple B)$ be the set of all $(\coord_6,
\coord_7, \coord_8) \in \FDxi(\coordtuple')$ that satisfy
\eqref{eq:A3+A1_height}. Then
\begin{equation*}
  \FDxi(\coordtuple'; u_\classtuple B) = \{(\coord_6,\coord_7,\coord_8) \in (K^\times)^2 \times K \mid \sigma(\coord_6,\coord_7,\coord_8) \in S_F(\coordtuple'; u_\classtuple B)\}.
\end{equation*}

The following observation will be crucial for all our upcoming error
estimates. Our construction of $F(\infty)$ implies that
\begin{equation*}
  \tN_v(\coordtuple'; x_{6v},x_{7v},x_{8v})^{1/d_v} \ll \tN_w(\coordtuple'; x_{6w},x_{7w},x_{8w})^{1/d_w} \ll \tN_v(\coordtuple'; x_{6v},x_{7v},x_{8v})^{1/d_v}
\end{equation*}
for all $(x_{jv})_{j,v} \in S_F(\coordtuple'; \infty)$ and all $v, w
\in \archplaces$. In particular,
\begin{equation}\label{eq:Nv_same_size}
  \tN_v(\coordtuple'; x_{6v},x_{7v},x_{8v})\ll B^{d_v/d}
\end{equation}
holds for all $(x_{jv})_{j,v} \in S_F(\coordtuple'; u_\classtuple B)$ and
all $v \in \archplaces$.

If we identify $\CC$ with $\RR^2$ then $\prod_{v\in\archplaces}K_v^3 =
\RR^{3d}$. Hence, we define the volume of a (measurable) subset of
$\prod_{v\in\archplaces}K_v^3$ as its usual Lebesgue measure. As one
would expect, the volume of $S_F(\coordtuple'; u_\classtuple B)$ will
appear at a later point as part of an asymptotic formula. Therefore, we
compute it here.

\begin{lemma}\label{lem:vol_SFB}
  For $B \geq 0$, the set $S_F(\coordtuple'; B)$ is measurable with
  volume
  \begin{equation*}
    \vol(S_F(\coordtuple'; B)) = \frac{1}{3}\cdot 2^{r_1}\cdot \left(\frac{\pi}{4}\right)^{r_2}\cdot\left(\prod_{v \in \archplaces}\omega_v(\tS)\right) \cdot R_K \cdot \frac{B}{|N(\coord_2\coord_3\coord_4\coord_5)|}.
  \end{equation*}
\end{lemma}
\begin{proof}
  First of all, we observe that $S_F(\coordtuple'; B) = B^{1/(3d)}
  S_F(\coordtuple'; 1)$ is homogeneously expanding, so it suffices to
  compute $\vol(S_F(\coordtuple'; 1))$. For $v\in\archplaces$, we
  define a scaling factor $l_v :=
  \absv{\coord_1\coord_2\coord_3\coord_4^2\coord_5^3}$. We transform
  the coordinates $x_{jv}$ in $S_F(\coordtuple';1)$ to
  \begin{align*}
    y_{0v} &= l_v^{-1/(3d_v)}\sigma_v(\coord_2)\cdot x_{8v}\\
    y_{1v} &= l_v^{-1/(3d_v)}\sigma_v(\coord_{1}\coord_{2}\coord_{3}\coord_{4}\coord_{5})\cdot x_{6v}\\
    y_{2v} &=
    l_v^{-1/(3d_v)}\sigma_v(\coord_{3}\coord_{4}^2\coord_{5}^3)\cdot
    x_{7v},
  \end{align*}
  for $v \in \archplaces$. The Jacobi determinant of this
  transformation has absolute value
  $|N(\coord_2\coord_3\coord_4\coord_5)|^{-1}$, and we easily
  verify that $\tN_v(\coordtuple'; x_{6v}, x_{7v}, x_{8v}) = N_v(y_{0v}, y_{1v},
  y_{2v})$. Thus,
  \begin{equation*}
    \vol(S_F(\coordtuple'; 1)) = \frac{1}{|N(\coord_2\coord_3\coord_4\coord_5)|}\int_{\substack{(y_{jv})_{j,v}\in \prod_{v \in \archplaces}(K_v\times(K_v^\times)^2)\\1/3\cdot(\log N_v(y_{0v},y_{1v}, y_{2v}))_{v\in\archplaces} \in F(1)}}\prod_{\substack{v\in\archplaces\\j\in\{0,1,2\}}}\dd y_{jv}.
  \end{equation*}
  Let $f : \prod_{v\in\archplaces}K_v^3 \to \RR_{\geq
    0}^\archplaces$ be given by
  \begin{equation*}
    f((y_{jv})_{j,v}) := (N_v(y_{0v}, y_{1v}, y_{2v}))_{v\in\archplaces}.
  \end{equation*}
  Then $f$ is Lebesgue-measurable and
  \begin{equation*}
    \vol(S_F(\coordtuple';1)) = \frac{1}{|N(\coord_2\coord_3\coord_4\coord_5)|}\cdot f_*(\vol)(\exp(3F(1))),
  \end{equation*}
  where $\exp$ is the coordinate-wise exponential map.  Let us compute
  the pushforward measure $f_*(\vol)$.  For $T \geq 0$, let $S_v(T) :=
  \{(y_0, y_1, y_2) \in K_v^3 \mid N_v(y_0,y_1,y_2)\leq T\}$. Since
  \begin{equation*}
    N_v(t \cdot (y_{0},y_{1}, y_{2})) =
    \absv{t}^3N_v(y_{0},y_{1},y_{2})
  \end{equation*}
  for all $t \in K_v$, we have $S_v(T) = T^{1/(3d_v)}S_v(1)$, and
  $\vol(S_v(T)) = T \vol(S_v(1))$. Hence, for any cell
  $E:=\prod_{v\in\archplaces}(\alpha_v, \beta_v]$ in $\RR_{\geq 0}^\archplaces$,
  we have
  \begin{equation*}
    (\vol\circ f^{-1})(E) = \prod_{v \in \archplaces}\int_{N_v(y_0,y_1,y_2)\in (\alpha_v, \beta_v]}\dd y_0\dd y_1 \dd y_2 = \prod_{v\in\archplaces}(\beta_v-\alpha_v)\vol(S_v(1)).
  \end{equation*}
  We conclude that
  \begin{equation*}
    f_*(\vol) = \prod_{v\in\archplaces}\vol(S_v(1)) \cdot \vol  = \left(\frac{2}{3}\right)^{r_1}\left(\frac{\pi}{12}\right)^{r_2}\prod_{v\in\archplaces}(\omega_v(\tS)) \cdot \vol.
  \end{equation*}
  To finish the proof, we need to compute $\vol(\exp(3F(1)))$. To this
  end, choose $w\in\archplaces$ and transform the coordinates by
  \begin{align*}
    x_v &= e^{3 y_v+ 3 d_v t}\text{, for } v \in \archplaces\smallsetminus\{w\},\\
    x_{w} &= e^{-3(\sum_{v\in\archplaces\smallsetminus\{w\}}y_v) +3d_w t},
  \end{align*}
  with Jacobi determinant $3^{|\archplaces|} d
  \prod_{v\in\archplaces} x_v = 3^{|\archplaces|}de^{3dt}$. We
  obtain
  \begin{equation*}
    \int_{\exp(3F(1))}\prod_{v\in\archplaces}\dd x_v = \int_{F}\prod_{v\in\archplaces\smallsetminus\{w\}}\dd y_v\int_{-\infty}^03^{|\archplaces|}de^{3d \cdot t}\dd t = 3^{|\archplaces|-1} R_K.\qedhere
  \end{equation*}
\end{proof}

\section{M\"obius inversions}\label{sec:lattice_point_problem}
In this section, we fix $\classtuple \in \classrepsyst^6$ and reduce
the task of counting, for fixed $\coord_1, \ldots, \coord_5$, the set
of all $(\coord_6, \coord_7, \coord_8, \coord_9)$ with $(\coord_1,
\ldots, \coord_9) \in M_\classtuple(B)$ to a lattice point
problem. The main job here is to deal with the coprimality conditions
\eqref{eq:A3+A1_coprimality} for $\coord_6, \coord_7, \coord_8,
\coord_9$. We write
\begin{align*}
  \coordtuple' &:=  (\coord_1, \ldots, \coord_5)\\
  \coordfracidealproduct &:= \coordfracideal_{1*}\times \cdots \times \coordfracideal_{5*}\\
  \coordidealtuple'&:=
  (\coordideal_1, \ldots, \coordideal_5).
\end{align*}
To encode the coprimality conditions \eqref{eq:A3+A1_coprimality} for
$\coordtuple'\in\coordfracidealproduct$, we define the function
$\theta_0(\coordidealtuple') :=
\prod_\ppp\theta_{0,\ppp}(J_\ppp(\coordidealtuple'))$, where
$J_\ppp(\coordidealtuple') := \{j \in \{1,\ldots,5\}\ :\ \ppp \mid
\coordideal_j\}$ and
\begin{equation*}
  \theta_{0,\ppp}(J):=
  \begin{cases}
    1 &\text{ if } J = \emptyset, \{1\}, \{2\}, \{3\}, \{4\}, \{5\}, \{3,4\}, \{4,5\},\\
    0 &\text{ otherwise.}
  \end{cases}
\end{equation*}
The product over $\ppp$ runs over all nonzero prime ideals of
$\Ok$. Clearly, $\theta_{0}(\coordidealtuple') = 1$ if and only if
\eqref{eq:A3+A1_coprimality} holds for all $j,k \in \{1,\ldots,5\}$,
and $\theta_{0}(\coordidealtuple') = 0$ otherwise. We rewrite the
coprimality conditions \eqref{eq:A3+A1_coprimality} as follows.

\begin{lemma}\label{lem:coprimality_rewritten}
  Let $(\coord_1, \ldots, \coord_9) \in \coordfracideal_{1}\times
  \cdots \times \coordfracideal_{9}$ satisfy the torsor equation \eqref{eq:A3+A1_torsor}. Then the coprimality conditions \eqref{eq:A3+A1_coprimality} hold if
  and only if the following conditions are satisfied:
  \begin{align}
    \theta_0(\coordidealtuple') &= 1\label{eq:cop_remaining_vars}\\
    \coordideal_6 + \coordideal_4\coordideal_5 &= \integersk\label{eq:cop_e3}\\
    \coordideal_7 + \coordideal_1\coordideal_2\coordideal_3\coordideal_4\cdot \coordideal_6 &= \integersk\label{eq:cop_e7}\\
    \coordideal_8 + \coordideal_3\coordideal_4\coordideal_5\cdot \coordideal_6 &= \integersk\label{eq:cop_e8}\\
    \coordideal_9 + \coordideal_6 &= \integersk\label{eq:cop_e9}\text.
  \end{align}
\end{lemma}

\begin{proof}
  The conditions \eqref{eq:A3+A1_coprimality} are equivalent to
  \eqref{eq:cop_remaining_vars}, \eqref{eq:cop_e3}, \eqref{eq:cop_e7},
  and
  \begin{align}
    \coordideal_8 + \coordideal_1\coordideal_3\coordideal_4\coordideal_5\cdot\coordideal_6 &= \integersk\label{eq:cop_e8_orig}\\
    \coordideal_9 +
    \coordideal_2\coordideal_3\coordideal_4\coordideal_5\cdot\coordideal_6
    &= \integersk.\label{eq:cop_e9_orig}
  \end{align}
  We show that the torsor equation \eqref{eq:A3+A1_torsor} and conditions \eqref{eq:cop_remaining_vars},
  \eqref{eq:cop_e7}, and \eqref{eq:cop_e8} already imply
  \eqref{eq:cop_e8_orig}. Assume that $\ppp \mid \coordideal_8 +
  \coordideal_1$.  Then in particular,
  \begin{equation*}
    \coord_2\coord_8\classrep_0^{-1} = \coordideal_2\coordideal_8 \subseteq \ppp \text{ and } \coord_1\coord_9\classrep_0^{-1} = \coordideal_1\coordideal_9 \subseteq\ppp.
  \end{equation*}
  Using \eqref{eq:A3+A1_torsor},
  \begin{equation*}
    \coordideal_3\coordideal_4^2\coordideal_5^3\coordideal_7 = \coord_3\coord_4^2\coord_5^3\coord_7\classrep_0^{-1} = (\coord_2\coord_8+\coord_1\coord_9)\classrep_0^{-1}\subseteq \coord_2\coord_8\classrep_0^{-1}+\coord_1\coord_9\classrep_0^{-1} \subseteq \ppp.
  \end{equation*}
  We conclude that $\ppp \mid
  \coordideal_1+\coordideal_3\coordideal_4^2\coordideal_5^3\coordideal_7$,
  which contradicts \eqref{eq:cop_remaining_vars} or
  \eqref{eq:cop_e7}. Similarly, one can show that
  \eqref{eq:A3+A1_torsor}, \eqref{eq:cop_remaining_vars},
  \eqref{eq:cop_e7}, \eqref{eq:cop_e8}, and \eqref{eq:cop_e9} imply
  \eqref{eq:cop_e9_orig}.
\end{proof}
We use the following notation for certain $6$-tuples of nonzero
ideals:
\begin{align*}
  \divisoridealtuple &:= (\divisorideal_{67}, \divisorideal_{68}, \divisorideal_{69}, \divisorideal_6, \divisorideal_7, \divisorideal_8),\\
  \mu_K(\divisoridealtuple) &:=
  \mu_K(\divisorideal_{67})\mu_K(\divisorideal_{68})\mu_K(\divisorideal_{69})\mu_K(\divisorideal_6)\mu_K(\divisorideal_7)\mu_K(\divisorideal_8),
\end{align*}
where $\mu_K(\aaa)$ is the M\"obius function for nonzero ideals $\aaa$
of $\Ok$.  In the next lemma, $\divisoridealtuple$ runs over all
$6$-tuples of nonzero ideals satisfying the conditions (depending on
$\coordidealtuple'$):
\begin{equation}\label{eq:dij_conditions}
  \begin{aligned}
    \divisorideal_{67} + \coordideal_1\coordideal_2\coordideal_3\coordideal_4\coordideal_5 &=\integersk, \\
    \divisorideal_{68} + \coordideal_1\coordideal_3\coordideal_4\coordideal_5 &=\integersk, \\
    \divisorideal_{69} + \coordideal_2\coordideal_3\coordideal_4\coordideal_5 &=\integersk, \\
    \divisorideal_{68} + \divisorideal_{69} &=\integersk,
  \end{aligned}
\end{equation}
and
\begin{equation}\label{eq:di_conditions}
  \begin{aligned}
    \divisorideal_6 &\mid \coordideal_4\coordideal_5, \\
    \divisorideal_7 &\mid \coordideal_1\coordideal_2\coordideal_3\coordideal_4,\\
    \divisorideal_8 &\mid \coordideal_3\coordideal_4\coordideal_5.\\
  \end{aligned}
\end{equation}
For any fixed $\classtuple, \coordtuple', \divisoridealtuple$
satisfying $\theta_0(\coordidealtuple')=1$ and the above conditions,
we define the fractional ideals
\begin{align*}
  \latticefracideal_6 &:=\divisorideal_6(\divisorideal_{67}\cap(\divisorideal_{68}\divisorideal_{69}))\coordfracideal_{6}\\
  \latticefracideal_7 &:=\divisorideal_7\divisorideal_{67}\coordfracideal_{7}\\
  \latticefracideal_8
  &:=\coordideal_1\divisorideal_8\divisorideal_{68}\divisorideal_{69}\coordfracideal_{8}.
\end{align*}
Conditions \eqref{eq:dij_conditions} and \eqref{eq:di_conditions}, together with $\theta_0(\coordidealtuple')=1$, imply that
\begin{equation*}
\coordideal_2\divisorideal_8\divisorideal_{68}\classrep_1\classrep_2\classrep_3 + \coordideal_1\divisorideal_{69}\classrep_1\classrep_2\classrep_3=\classrep_1\classrep_2\classrep_3. 
\end{equation*}
Moreover, $\coord_3\coord_4^2\coord_5^3\Ok = \coordideal_3\coordideal_4^2\coordideal_5^3\coordfracideal_3\coordfracideal_4^2\coordfracideal_5^3 =\coordideal_3\coordideal_4^2\coordideal_5^3\classrep_1\classrep_2\classrep_3$, so $\coord_3\coord_4^2\coord_5^3\equiv 0 \bmod \classrep_1\classrep_2\classrep_3$. Unique ideal factorization in $\Ok$ allows us to apply the Chinese remainder theorem with not necessarily coprime moduli to conclude that the congruence
\begin{equation}\label{eq:gamma_8_cong_syst}
  \gamma_8^* \equiv
  \begin{cases}
    0 &\mod \coordideal_2\divisorideal_8\divisorideal_{68}\classrep_1\classrep_2\classrep_3,\\
    -\coord_3\coord_4^2\coord_5^3 &\mod
    \coordideal_1\divisorideal_{69}\classrep_1\classrep_2\classrep_3
  \end{cases}
\end{equation}
has a solution $\gamma_8^*(\classtuple,
\coordtuple', \divisoridealtuple) \in \Ok$, which is unique modulo
\begin{equation*}
  \coordideal_2\divisorideal_8\divisorideal_{68}\classrep_1\classrep_2\classrep_3 \cap \coordideal_1\divisorideal_{69}\classrep_1\classrep_2\classrep_3 = \coordideal_1\coordideal_2\divisorideal_8\divisorideal_{68}\divisorideal_{69}\classrep_1\classrep_2\classrep_3
= \coord_2\coordfracideal_7^{-1}\latticefracideal_8.
\end{equation*}
Let $\gamma_8 := \gamma_8^*/\coord_2$. Then we define $\countinggroup(\classtuple,
\coordtuple', \divisoridealtuple)$ as the additive subgroup of $K^3$
consisting of all $(\coord_6,\coord_7,\coord_8)$ with
\begin{align*}
  \coord_6 &\in \latticefracideal_6\\
  \coord_7 &\in \latticefracideal_7\\
  \coord_8 &\in \gamma_8\cdot\coord_7 + \latticefracideal_8.
\end{align*}
Note that $\countinggroup(\classtuple, \coordtuple',
\divisoridealtuple)$ does not depend on the choice of $\gamma_8^*$.

\begin{lemma}\label{lem:removed_torsor_coprim}
  Let $\classtuple \in \classrepsyst^6$ and $B \geq 0$. With $\countinggroup(\classtuple,
\coordtuple', \divisoridealtuple)$ as defined above and $\FDei, \FDxi(\coordtuple';
    u_\classtuple B)$ as defined in Section \ref{sec:fundamental_domain},
  \begin{align*}
    &|M_\classtuple(B)|=\sum_{\coordtuple' \in
      \FDei^5\cap\coordfracidealproduct}\theta_0(\coordidealtuple')\sum_{\substack{\divisoridealtuple\\\eqref{eq:dij_conditions},\eqref{eq:di_conditions}}}\mu_K(\divisoridealtuple)|\countinggroup(\classtuple,
    \coordtuple', \divisoridealtuple) \cap \FDxi(\coordtuple';
    u_\classtuple B)|.
  \end{align*}
\end{lemma}

\begin{proof}
  For now, let us fix $\coordtuple'$ with
  $\theta_{0}(\coordidealtuple') = 1$ and write $\FDxi :=
  \FDxi(\coordtuple';u_\classtuple B)$. We define
  \begin{equation*}
    \tM = \tM(\classtuple, \coordtuple', B) := |\{(\coord_6, \coord_7, \coord_8, \coord_9) \mid (\coord_1, \ldots, \coord_9) \in M_\classtuple(B)\}|\text.
  \end{equation*}
  M\"obius inversion for the coprimality condition \eqref{eq:cop_e9} shows that
  \begin{equation*}
    \tM=\sum_{\divisorideal_{69}\in \idealsk}\mu_K(\divisorideal_{69})\Bigg|\left\{
      \begin{aligned}
        (\coord_6,\coord_7,\coord_8,\coord_9) \in
        ((\divisorideal_{69}\coordfracideal_{6}\times\coordfracideal_{7}\times\coordfracideal_{8})
        \cap \FDxi)\times
        \divisorideal_{69}\coordfracideal_{9})& \\
        :\eqref{eq:A3+A1_torsor}, \eqref{eq:cop_e3}-\eqref{eq:cop_e8}&
      \end{aligned}
    \right\}\Bigg|.
  \end{equation*}
  There is a one-to-one correspondence between
  $(\coord_6,\coord_7,\coord_8,\coord_9)$ as above satisfying the torsor equation  \eqref{eq:A3+A1_torsor} and triples $(\coord_6,\coord_7,\coord_8)\in
  (\divisorideal_{69}\coordfracideal_{6}\times\coordfracideal_{7}\times\coordfracideal_{8})\cap
  \FDxi$ satisfying the congruence
  \begin{equation}\label{eq:torsor_congruence}
    \coord_3\coord_4^2\coord_5^3\coord_7 + \coord_2\coord_8 \equiv 0 \mod \coord_1\divisorideal_{69}\coordfracideal_9 = \divisorideal_{69}\coordideal_1\classrep_0. 
  \end{equation}
  Moreover, \eqref{eq:torsor_congruence} and
  \eqref{eq:cop_e3}-\eqref{eq:cop_e8} imply that $\divisorideal_{69} +
  \coordideal_2\coordideal_3\coordideal_4\coordideal_5 =
  \integersk$. We apply M\"obius inversion to resolve the coprimality
  condition $\coordideal_6+\coordideal_8 = \integersk$ resulting from
  \eqref{eq:cop_e8}. As a result, $\tM$ is equal to
  \begin{equation*}
    \sum_{\substack{\divisorideal_{68},\divisorideal_{69} \in \idealsk\\\divisorideal_{69}+\coordideal_2\coordideal_3\coordideal_4\coordideal_5=\integersk}}\hspace{-5mm}\mu_K(\divisorideal_{68}, \divisorideal_{69})\Bigg|\left\{
      \begin{aligned}
        (\coord_6,\coord_7,\coord_8) \in
        ((\divisorideal_{68}\cap\divisorideal_{69})\coordfracideal_{6}\times
        \coordfracideal_{7} \times
        \divisorideal_{68}\coordfracideal_{8})\cap\FDxi&\\:
        \eqref{eq:torsor_congruence},
        \eqref{eq:cop_e3},\eqref{eq:cop_e7}, \coordideal_8 +
        \coordideal_3\coordideal_4\coordideal_5 = \integersk&
      \end{aligned}
    \right\}\Bigg|,
  \end{equation*}
  where $\mu_K(\divisorideal_{68}, \divisorideal_{69}) :=
  \mu_K(\divisorideal_{68})\mu_K(\divisorideal_{69})$. Clearly, the
  summand is $0$ whenever
  $\divisorideal_{68}+\coordideal_3\coordideal_4\coordideal_5\neq
  \integersk$. Moreover, due to \eqref{eq:torsor_congruence} and
  \eqref{eq:cop_e7}, we see that
  $\divisorideal_{68}+\divisorideal_{69}\coordideal_1 =
  \integersk$. One further application of M\"obius inversion to
  resolve the coprimality condition between $\coordideal_6$ and
  $\coordideal_7$ resulting from \eqref{eq:cop_e7} shows that
  \begin{align*}
    \tM=&\sum_{\substack{\divisorideal_{67},\divisorideal_{68},\divisorideal_{69} \in \idealsk\\\divisorideal_{68}+\coordideal_1\coordideal_3\coordideal_4\coordideal_5=\integersk\\\divisorideal_{69}+\coordideal_2\coordideal_3\coordideal_4\coordideal_5=\integersk\\\divisorideal_{68}+\divisorideal_{69}=\integersk}}\hspace{-5mm}\mu_K(\divisorideal_{67}, \divisorideal_{68}, \divisorideal_{69})\cdot\\
    &\Bigg|\left\{
      \begin{aligned}
        (\coord_6,\coord_7,\coord_8) \in
        ((\divisorideal_{67}\cap(\divisorideal_{68}\divisorideal_{69}))\coordfracideal_{6}\times
        \divisorideal_{67}\coordfracideal_{7} \times
        \divisorideal_{68}\coordfracideal_{8})\cap\FDxi&\\:
        \eqref{eq:torsor_congruence}, \eqref{eq:cop_e3},\coordideal_7
        +
        \coordideal_1\coordideal_2\coordideal_3\coordideal_4=\integersk,
        \coordideal_8 + \coordideal_3\coordideal_4\coordideal_5 =
        \integersk&
      \end{aligned}
    \right\}\Bigg|,
  \end{align*}
  with $\mu_K(\divisorideal_{67}, \divisorideal_{68},
  \divisorideal_{69})$ defined similarly as above. Clearly, we may add
  the condition
  $\divisorideal_{67}+\coordideal_1\coordideal_2\coordideal_3\coordideal_4\coordideal_5=\Ok$
  under the sum. After three more applications of M\"obius inversion
  to resolve the remaining coprimality conditions,
  \begin{align*}
    \tM=\sum_{\substack{\divisoridealtuple\\\eqref{eq:dij_conditions},\eqref{eq:di_conditions}}}\mu_K(\divisoridealtuple)\Bigg|\left\{
      \begin{aligned}
        (\coord_6,\coord_7,\coord_8) \in (\latticefracideal_6\times
        \latticefracideal_7 \times
        \divisorideal_8\divisorideal_{68}\coordfracideal_{8})\cap\FDxi&\\:
        \eqref{eq:torsor_congruence}
      \end{aligned}
    \right\}\Bigg|.
  \end{align*}
  The conditions $\coord_8 \in
  \divisorideal_8\divisorideal_{68}\coordfracideal_{8}$ and
  \eqref{eq:torsor_congruence} are equivalent to the system of
  congruences
  \begin{equation}\label{eq:a_2_a_8_cong_syst}
    \begin{aligned}
      \coord_2\coord_8 &\equiv 0& &\mod \coordideal_2\divisorideal_8\divisorideal_{68}\classrep_0\\
      \coord_2\coord_8 &\equiv -\coord_3\coord_4^2\coord_5^3\coord_7&
      &\mod \divisorideal_{69}\coordideal_1\classrep_0.
    \end{aligned}
  \end{equation}
  Recall that $\gamma_8^*$ is a solution to the system
  \eqref{eq:gamma_8_cong_syst}. Multiplying by $a_7$, we see that
  $a_2a_8 = \gamma_8^*a_7$ is a solution to
  \eqref{eq:a_2_a_8_cong_syst}. Hence, \eqref{eq:a_2_a_8_cong_syst} is equivalent to
  \begin{equation*}
    \coord_2\coord_8
    \equiv \gamma_8^*\coord_7 \mod
    (\coordideal_2\divisorideal_8\divisorideal_{68}\classrep_0\cap\divisorideal_{69}\coordideal_1\classrep_0) = \coordideal_1\coordideal_2\divisorideal_8\divisorideal_{68}\divisorideal_{69}\classrep_0.
  \end{equation*}
  Dividing by $\coord_2$ proves the lemma.
\end{proof}

\section{Small conjugates}\label{sec:small_conjugates}
From the conditions $\coord_6, \coord_7 \neq 0$ in
$\FDxi(\coordtuple'; u_\classtuple B)$, we see that every
$(\coord_6,\coord_7,\coord_8) \in
\countinggroup(\classtuple,\coordtuple',\divisoridealtuple)\cap
\FDxi(\coordtuple'; u_\classtuple B)$ satisfies $N(\coord_6) \geq
\N\latticefracideal_6$ and $N(\coord_7) \geq
\N\latticefracideal_7$. We would like to replace these by the stronger
conditions
\begin{equation*}
  \abs{\coord_6}_v \geq \N\latticefracideal_6^{d_v/d}\text{ and }\abs{\coord_7}_v \geq \N\latticefracideal_7^{d_v/d} \text{ for all }v \in \archplaces. 
\end{equation*}
If $|\archplaces| = 1$ then there is nothing to do. Let us first
prove some auxiliary results.

\begin{lemma}\label{lem:box_counting}
  Let $\aaa$ be a nonzero fractional ideal of $K$ and $b \in K$. For each
  $v \in \archplaces$, let $y_v \in K_v$ and $c_v > 0$. Define
  \begin{equation*}
    \countingbox := \{a \in K \mid |a^{(v)}+y_v|_v \leq c_v \text{ for all }v\in\archplaces\}. 
  \end{equation*}
  Then
  \begin{equation*}
    |(b+\aaa) \cap \countingbox| \ll \frac{1}{\N\aaa}\left(\prod_{v\in\archplaces}c_v\right)+1.
  \end{equation*}
\end{lemma}

\begin{proof}
  Replacing $y_v$ by $y_v+b^{(v)}$, we may assume that $b =
  0$. Denote by $B_{-y_v}(c_v)$ the closed ball in $K_v$ (with respect
  to $\absv{\cdot}$) with center $-y_v$ and radius $c_v$. Let $M :=
  \prod_{v\in\archplaces}B_{-y_v}(c_v)$. Let $c :=
  (\prod_{v\in\archplaces}c_v)^{1/d}$, and let $\tau :
  \prod_{v\in\archplaces}K_v \to \prod_{v\in\archplaces}K_v$ be the
  $\RR$-linear transformation of determinant $1$ given by $\tau(x_v) =
  cc_v^{-1/d_v} \cdot x_v$.  Clearly,
  \begin{equation*}
    |\aaa \cap \countingbox| = |\sigma(\aaa) \cap M| = |\tau(\sigma(\aaa)) \cap \tau(M)|.
  \end{equation*}
  With the usual identification $\prod_{v\in\archplaces}K_v=\RR^{d}$,
  the boundary of $\tau(M)$ is Lipschitz-parameterizable with Lipschitz
  constant $\ll c$ (cf. \cite{MR2247898}, \cite[V,~\textsection
  2]{MR1282723}) and $\tau(\sigma(\aaa))$ is a lattice in $\RR^d$ with
  determinant $2^{-s}\N\aaa\sqrt{|\Delta_K|}$ and first successive
  minimum $\lambda_1 \geq \N\aaa^{1/d}$ (cf. \cite[Lemma
  5.1]{MR3107569}). A classical counting argument
  (cf. \cite[Lemma 2]{MR2247898}) shows that
  \begin{equation*}
    |\tau(\sigma(\aaa)) \cap \tau(M)| \ll \frac{c^d}{\N\aaa}+1.\qedhere
  \end{equation*}
\end{proof}

\begin{lemma}\label{lem:norm_sum_no_lower_bound}
  Let $\aaa$ be a nonzero fractional ideal of $K$ and $c_v > 0$ for
  all $v \in \archplaces$. For $\alpha \in [0, 1)$,
  \begin{equation*}
    \sum_{\substack{0 \neq \coord \in \aaa\\\absv{\coord} \leq c_v\ \forall v\in\archplaces}}\frac{1}{N(\coord)^\alpha} \ll_\alpha \frac{1}{\N\aaa}\prod_{v\in\archplaces}c_v^{(1-\alpha)}.
  \end{equation*}
  For $\alpha > 1$,
  \begin{equation*}
    \sum_{\substack{0 \neq \coord \in \aaa\\\absv{\coord} \geq c_v\ \forall v\in\archplaces}}\frac{1}{N(\coord)^\alpha} \ll_\alpha \frac{1}{\N\aaa}\prod_{v\in\archplaces}c_v^{(1-\alpha)}.
  \end{equation*}
\end{lemma}

\begin{proof}
  Write $c := \prod_{v\in\archplaces}c_v$. Let us start by
  proving the assertion for $\alpha\in [0,1)$. We have
  \begin{equation*}
    \sum_{\substack{0 \neq \coord \in \aaa\\\absv{\coord} \leq c_v\ \forall v\in\archplaces}}\frac{1}{N(\coord)^\alpha} = \sum_{\substack{\bbb \in \principalidealsk\\\bbb\subseteq\aaa \\\N\bbb \leq c}}\frac{n(\bbb)}{\N\bbb^\alpha},
  \end{equation*}
  with
  \begin{equation*}
    n(\bbb) := |\{\coord \in K^\times \mid \coord\integersk = \bbb,\ \absv{\coord}\leq c_v \text{ for all }v\in\archplaces\}|.
  \end{equation*}
  Let us find an upper bound for $n(\bbb)$.  Let $b \in
  K^\times$ with $b\integersk = \bbb$, write $q := |\archplaces|-1$,
  and let $u_1, \ldots, u_{q}$ be a system of
  fundamental units of $\integersk$. Then
  \begin{align*}
    n(\bbb) &= |\{u \in \units \mid \absv{u b} \leq
    c_v\text{ for all }v\in\archplaces\}|\\ &=
    |\mu_K|\cdot|\{(e_1, \ldots, e_{q})\in \ZZ^{q} \mid
    \absv{u_1^{e_1}\cdots u_{q}^{e_{q}}} \leq
    c_v/\absv{b} \text{ for all }v\in\archplaces\}.
  \end{align*}
  Choose $w \in \archplaces$. Taking logarithms
  and using the fact that $|N(u_j)|=1$ to express the condition
  for $w$, we see that $n(\bbb)$ is $|\mu_K|$ times the
  number of solutions $(e_1, \ldots, e_{q}) \in \ZZ^{q}$ of the system
  \begin{align*}
    &e_1\log\absv{u_1} + \cdots + e_{q}\log\absv{u_{q}} \leq \log(c_v/\absv{b}),\quad\text{ for }v\in\archplaces\smallsetminus\{w\}\text{ and}\\
    &e_1\left(-\sum_{v\in\archplaces\smallsetminus\{w\}}\log\absv{u_1}\right)
    + \cdots +
    e_{q}\left(-\sum_{v\in\archplaces\smallsetminus\{w\}}\log\absv{u_{q}}\right)
    \leq \log(c_{w}/\abs{b}_{w}).
  \end{align*}
  For every $v\in\archplaces\smallsetminus\{w\}$, we add the first
  inequality for all $v_0 \neq v$ to the second one and obtain
  \begin{equation*}
    \log(c_v/\absv{b})-\log(c/\N\bbb) \leq e_1\log\absv{u_1} + \cdots + e_{q}\log\absv{u_{q}} \leq \log(c_v/\absv{b}).
  \end{equation*}
  The fundamental units $u_1, \ldots, u_{q}$ can be chosen such that
  the $(\log\absv{u_i})_{v\in\archplaces\smallsetminus\{w\}}$, $1 \leq i
  \leq q$, are a basis of a lattice in $\RR^{q}$ of determinant and
  first successive minimum $\gg 1$. Hence, we need to estimate the
  number of lattice points in a box of edge-length
  $\log(c/\N\bbb)$. With $\epsilon := (1-\alpha)/2>0$, we have
  \begin{equation*}
    n(\bbb) \ll \log(c/\N\bbb)^{q}+1 \ll_\alpha (c/\N\bbb)^{\epsilon}.
  \end{equation*}
  We conclude that
  \begin{equation*}
    \sum_{\substack{\bbb \in \principalidealsk\\\bbb\subseteq\aaa\\\N\bbb \leq c}}\frac{n(\bbb)}{\N\bbb^\alpha} \ll_\alpha\hspace{-0.1cm} \sum_{\substack{\bbb \in \principalidealsk\\\bbb\subseteq\aaa\\\N\bbb \leq c}}\frac{c^{\epsilon}}{\N\bbb^{\alpha+\epsilon}} = \frac{c^{\epsilon}}{\N\aaa^{\alpha+\epsilon}}\hspace{-0.2cm}\sum_{\substack{\bbb\in [\aaa^{-1}]\cap \idealsk\\\N\bbb\leq c/\N\aaa}}\frac{1}{\N\bbb^{\alpha+\epsilon}} \ll_{\alpha} \frac{c^{\epsilon}}{\N\aaa^{\alpha+\epsilon}}\cdot\left(\frac{c}{\N\aaa}\right)^{1-\alpha-\epsilon}.
  \end{equation*}
  The second assertion follows analogously. This time,
  \begin{equation*}
    \sum_{\substack{0 \neq \coord \in \aaa\\\absv{\coord} \geq c_v\ \forall v\in\archplaces}}\frac{1}{N(\coord)^\alpha} = \sum_{\substack{\bbb \in \principalidealsk\\\bbb\subseteq\aaa \\\N\bbb \geq c}}\frac{n^*(\bbb)}{\N\bbb^\alpha},
  \end{equation*}
  with
  \begin{equation*}
    n^*(\bbb) := |\{\coord \in K^\times \mid \coord\integersk = \bbb,\ \absv{\coord}\geq c_v \text{ for all }v\in\archplaces\}|.
  \end{equation*}
  The same argument as before with reversed inequalities shows that
  $n^*(\bbb)$ can be estimated by the number of lattice points in a
  box of edge-length $\log(\N\bbb/c)$, so
  $n^*(\bbb)\ll_{\alpha}(\N\bbb/c)^{\epsilon}$ for $\epsilon :=
  (a-1)/2>0$. Thus,
  \begin{equation*}
    \sum_{\substack{\bbb \in \principalidealsk\\\bbb\subseteq\aaa\\\N\bbb \geq c}}\frac{n(\bbb)}{\N\bbb^\alpha} \ll_\alpha\hspace{-0.1cm}\sum_{\substack{\bbb \in \principalidealsk\\\bbb\subseteq\aaa\\\N\bbb \geq c}}\frac{c^{-\epsilon}}{\N\bbb^{\alpha-\epsilon}} = \frac{c^{-\epsilon}}{\N\aaa^{\alpha-\epsilon}}\hspace{-0.2cm}\sum_{\substack{\bbb\in [\aaa^{-1}]\cap \idealsk\\\N\bbb\geq c/\N\aaa}}\frac{1}{\N\bbb^{\alpha-\epsilon}} \ll_{\alpha} \frac{c^{-\epsilon}}{\N\aaa^{\alpha-\epsilon}}\cdot\left(\frac{c}{\N\aaa}\right)^{1-\alpha+\epsilon}.\qedhere
  \end{equation*}
\end{proof}

The following technical lemma provides conditions under which certain
error terms are summable. We use it in our error estimates here and
later in Sections \ref{sec:volumes_of_projections} and
\ref{sec:first_summation}. Recall the definitions of
$\divisoridealtuple$ and of the $\latticefracideal_j$ from Section
\ref{sec:lattice_point_problem}. For $\beta\in\RR^{\neq 0}$, let $\sgn(\beta)\in\{\pm 1\}$ be its sign.

\begin{lemma}\label{lem:error_sum}
  Let $\classtuple \in \classrepsyst^6$, let $\epsilon, \alpha_6,
  \alpha_7, \alpha_8 > 0$, and $\alpha \in [0,1]$. Let $e_1, \ldots,
  e_5 \in \ZZ$, not all equal to $0$, and $\beta \in \RR^{\neq
    0}$. Consider norm conditions
  \begin{equation}\label{eq:general_height_conditions}
    \begin{aligned}
      N(\coord_1^{e_1}\cdots \coord_5^{e_5})^{\sgn(\beta)} &\ll
      B^{\sgn(\beta)} \text{ and }\\
      N(\coord_j) &\ll B \text{ for all }j\in \{1, \ldots, 5\}.
    \end{aligned}
  \end{equation}
  If
  \begin{equation}\label{eq:dij_exp_greater_1}
    \alpha\alpha_6+\alpha_7 > 1 \quad\text{ and }\quad (1-\alpha)\alpha_6 + \alpha_8 > 1
  \end{equation}
  then the sum
  \begin{equation*}
    \sum_{\substack{\coordtuple' \in \FDei^5\cap\coordfracidealproduct\\\eqref{eq:general_height_conditions}}}\hspace{-0.2cm}\sum_{\substack{\divisoridealtuple\\\eqref{eq:dij_conditions},\eqref{eq:di_conditions}}}\hspace{-0.3cm}\frac{|\mu_K(\divisoridealtuple)|}{\N\latticefracideal_6^{\alpha_6}\N\latticefracideal_7^{\alpha_7}\N\latticefracideal_8^{\alpha_8}}\cdot \frac{B}{|N(\coord_1)|^{1-\alpha_8}|N(\coord_2\coord_3\coord_4\coord_5)|}\cdot \left(\frac{B}{|N(\coord_1^{e_1}\cdots \coord_5^{e_5})|}\right)^{-\beta}
  \end{equation*}
  is $\ll B(\log B)^{4+\epsilon}$, for $B \geq 3$. The implicit
  constant depends on $K$, $\epsilon$, $\alpha_6$, $\alpha_7$, $\alpha_8$,
  $\alpha$, $e_1, \ldots, e_5$, $\beta$ and on the implicit constants in
  \eqref{eq:general_height_conditions}.
\end{lemma}

\begin{proof}
  From the definitions of the $\latticefracideal_j$, we see that
  \begin{equation*}
    \N\latticefracideal_6^{\alpha_6}\N\latticefracideal_7^{\alpha_7}\N\latticefracideal_8^{\alpha_8} \gg \N\divisorideal_6^{\alpha_6}\N\divisorideal_7^{\alpha_7}\N\divisorideal_8^{\alpha_8}\N\divisorideal_{67}^{\alpha\alpha_6+\alpha_7}\N(\divisorideal_{68}\divisorideal_{69})^{(1-\alpha)\alpha_6+\alpha_8}N(a_1)^{\alpha_8}.
  \end{equation*}
  By \eqref{eq:dij_exp_greater_1},  
  \begin{equation*}
    \sum_{\divisorideal_{67}\in\idealsk}\frac{1}{\N\divisorideal_{67}^{\alpha\alpha_6+\alpha_7}}\ll  1,
  \end{equation*}
  and similar estimates hold for the sums over
  $\divisorideal_{68}, \divisorideal_{69}$. To sum over the
  $\divisorideal_j$, we use the assumption that $\alpha_j > 0$. Let
  $\epsilon' := \min\{\epsilon/35, 1\}$. For $j=6$, for example, we
  have
  \begin{equation*}
    \sum_{\divisorideal_6 \mid \coordideal_4\coordideal_5}\frac{|\mu_K(\divisorideal_6)|}{\N\divisorideal_6^{\alpha_6}}=\prod_{\ppp \mid \coordideal_4\coordideal_5}(1+\N\ppp^{-\alpha_6}) \ll\prod_{\ppp\mid\coordideal_4\coordideal_5}(1+\epsilon') = (1+\epsilon')^{\omega_K(\coordideal_4\coordideal_5)},
  \end{equation*}
  where $\omega_K(\coordideal)$ is the number of distinct prime ideals
  dividing the ideal $\coordideal$. Similar bounds hold for
  $j=7,8$. Since
  \begin{equation*}
     (1+\epsilon')^{\omega_K(\coordideal_4\coordideal_5) + \omega_K(\coordideal_1\coordideal_2\coordideal_3\coordideal_4) + \omega_K(\coordideal_3\coordideal_4\coordideal_5)} \leq (1+\epsilon/5)^{\omega_K(\coordideal_1\coordideal_2\coordideal_3\coordideal_4\coordideal_5)}, 
  \end{equation*}
  the sum in the lemma is
  \begin{equation*}
    \ll  \sum_{\substack{\coordtuple' \in \FDei^5\cap\coordfracidealproduct\\\eqref{eq:general_height_conditions}}}\frac{(1+\epsilon/5)^{\omega_K(\coordideal_1\coordideal_2\coordideal_3\coordideal_4\coordideal_5)}B}{|N(\coord_1\coord_2\coord_3\coord_4\coord_5)|}\left(\frac{B}{|N(\coord_1^{e_1}\cdots \coord_5^{e_5})|}\right)^{-\beta}.
  \end{equation*}
  If $\coord_j$ runs through $\FDei\cap\coordfracideal_{j*}$ then
  $\coordideal_j = \coord_j\coordfracideal_j^{-1}$ runs through all
  nonzero ideals in the class of $\coordfracideal_j^{-1}$. Moreover,
  $\N\coordideal_j \ll N(\coord_j) \ll \N\coordideal_j$, so the above
  sum is
  \begin{equation}\label{eq:final_sum}
    \ll \sum_{\substack{\coordidealtuple' \in \idealsk^5\\\eqref{eq:general_height_condition_ideals}, \eqref{eq:aj_bound_ideals}}}\frac{(1+\epsilon/5)^{\omega_K(\coordideal_1\coordideal_2\coordideal_3\coordideal_4\coordideal_5)}B}{N(\coordideal_1\coordideal_2\coordideal_3\coordideal_4\coordideal_5)}\left(\frac{B}{N(\coordideal_1^{e_1}\cdots \coordideal_5^{e_5})}\right)^{-\beta},
  \end{equation}
  where $\coordidealtuple'$ runs through all $5$-tuples of ideals
  $(\coordideal_1, \ldots, \coordideal_5) \in \idealsk^5$ with
  \begin{align}
      N(\coordideal_1^{e_1}\cdots \coordideal_5^{e_5})^{\sgn(\beta)} &\ll
    B^{\sgn(\beta)} \text{ and }\label{eq:general_height_condition_ideals}\\
    N(\coordideal_j) &\ll B \text{ for all }j\in \{1, \ldots, 5\}\label{eq:aj_bound_ideals}.
  \end{align}
  All of the following summations are elementary applications of partial
  summation. For details, see \cite[Lemma 2.9, Lemma
  2.4]{arXiv:1302.6151}.

  Let us first consider the case where $\beta
  > 0$. If $e_j\leq 0$ for all $j$ then the sum in
  \eqref{eq:final_sum} is $\ll B^{1-\beta}(\log
  B)^{4+4\epsilon/5}$. Hence, we may assume without loss of generality
  that $e_1 > 0$. Using \eqref{eq:general_height_condition_ideals} to
  sum over $\coordideal_1$, we see that the sum in
  \eqref{eq:final_sum} is
  \begin{equation}\label{eq:final_sum_result}
    \ll \sum_{\substack{\coordideal_2, \ldots, \coordideal_5\\\eqref{eq:aj_bound_ideals}}}\frac{(1+\epsilon/5)^{\omega_K(\coordideal_2\coordideal_3\coordideal_4\coordideal_5)}B(\log B)^{\epsilon/5}}{N(\coordideal_2\coordideal_3\coordideal_4\coordideal_5)} \ll B(\log B)^{4+\epsilon}.
  \end{equation}
  Now we assume that $\beta < 0$. If $e_j\leq 0$ for all $j$ then the
  sum in \eqref{eq:final_sum} is $\ll 1$. If
  $e_1 > 0$ then we may use \eqref{eq:general_height_condition_ideals}
  again to sum over $\coordideal_1$ and obtain
  \eqref{eq:final_sum_result}.
\end{proof}

For $w \in \archplaces$, let
\begin{align*}
  \cutoff_{6}^{(w)}(\classtuple, \coordtuple', \divisoridealtuple; B) &:= \{(\coord_6, \coord_7, \coord_8) \in \countinggroup(\classtuple, \coordtuple', \divisoridealtuple)\cap \FDxi(\coordtuple'; u_\classtuple B) \mid \abs{\coord_6}_w< \N\latticefracideal_6^{d_w/d}\},\\
  \cutoff_{7}^{(w)}(\classtuple, \coordtuple', \divisoridealtuple; B)
  &:= \{(\coord_6,\coord_7,\coord_8) \in \countinggroup(\classtuple,
  \coordtuple', \divisoridealtuple)\cap \FDxi(\coordtuple';
  u_\classtuple B) \mid \abs{\coord_7}_w<
  \N\latticefracideal_7^{d_w/d}\}.
\end{align*}

We show that the contribution of all $\cutoff_{6}^{(w)}$ and
$\cutoff_{7}^{(w)}$ to $|\countinggroup(\classtuple, \coordtuple',
\divisoridealtuple)\cap \FDxi(\coordtuple'; u_\classtuple B)|$ is
insignificant. To this end, we note some conditions satisfied by all
$(\coord_6,\coord_7,\coord_8)\in\countinggroup(\classtuple,
\coordtuple', \divisoridealtuple)\cap \FDxi(\coordtuple';
u_\classtuple B)$, which follow from \eqref{eq:Nv_same_size} and the
definition of $\tN_v$. For all $v \in \archplaces$, we have
\begin{align}
  \absv{\coord_2\coord_7\coord_8^2 + \coord_3\coord_4^2\coord_5^3\coord_7^2\coord_8} &\ll \absv{\coord_1}B^{d_v/d},\label{eq:cusps_a8_bound}\\
  \absv{\coord_3^2\coord_4^3\coord_5^4\coord_6\coord_7^2} &\ll
  B^{d_v/d},\label{eq:cusps_a7_bound}\\
  \absv{\coord_1^2\coord_2^2\coord_3^2\coord_4\coord_6^3} &\ll
  B^{d_v/d}.\label{eq:cusps_a6_bound}
\end{align}
Moreover, any $\coordtuple'$, $\divisoridealtuple$ with
$\countinggroup(\classtuple, \coordtuple', \divisoridealtuple)\cap
\FDxi(\coordtuple'; u_\classtuple B) \neq \emptyset$ satisfies
\begin{align}
  N(\coord_j) &\ll B \text{ for all } j \in \{1,\ldots, 5\},\label{eq:cusps_aj_bound}\\
  N(\coord_1^2\coord_2^2\coord_3^2\coord_4) &\ll B.\label{eq:cusps_a1_bound}
\end{align}
In our calculations, we will encounter the new height condition 
\begin{equation}
  \label{eq:additional_height_condition}
  N(\coord_3^2\coord_4^4\coord_5^6)\leq N(\coord_1\coord_2)v_\classtuple B,
\end{equation}
with $v_\classtuple :=
\N(\coordfracideal_3^2\coordfracideal_4^4\coordfracideal_5^6\coordfracideal_1^{-1}\coordfracideal_{2}^{-1})
$, so $1\ll v_\classtuple \ll 1$.
 
\begin{lemma}\label{lem:a6_conj_lower_bound}
  Let $\classtuple \in \classrepsyst^6$ and $\epsilon > 0$. Then, for $B \geq 3$,
  \begin{equation*}
    \sum_{\coordtuple' \in \FDei^5\cap\coordfracidealproduct} \theta_0(\coordidealtuple')\sum_{\substack{\divisoridealtuple\\\eqref{eq:dij_conditions},\eqref{eq:di_conditions}}}|\mu_K(\divisoridealtuple)|\cdot |\cutoff_6^{(w)}(\classtuple, \coordtuple', \divisoridealtuple; B)| \ll_\epsilon B(\log B)^{4+\epsilon}.
  \end{equation*}
\end{lemma}

\begin{proof}
  Let us first fix $\coordtuple', \divisoridealtuple, \coord_6,
  \coord_7$ and find an upper bound for the number of $\coord_8$ with
  $(\coord_6,\coord_7,\coord_8) \in
  \cutoff_6^{(w)}(\classtuple,\coordtuple',\divisoridealtuple;B)$. Condition
  \eqref{eq:cusps_a8_bound} implies that, for all $v \in \archplaces$,
  one of
  \begin{equation*}
    \absv{\coord_8} \ll \frac{B^{d_v/(2d)}\absv{\coord_1}^{1/2}}{\absv{\coord_2\coord_7}^{1/2}}\quad\text{ or }\quad \absv{\coord_8+ \frac{\coord_3\coord_4^2\coord_5^3\coord_7}{\coord_2}} \ll \frac{B^{d_v/(2d)}\absv{\coord_1}^{1/2}}{\absv{\coord_2\coord_7}^{1/2}}
  \end{equation*}
  holds. By Lemma \ref{lem:box_counting}, the number of such
  $\coord_8$ in $\gamma_8\coord_7 + \latticefracideal_8$ is
  \begin{equation}\label{eq:sum_a8}
    \ll \frac{1}{\N\latticefracideal_8}\left(\frac{BN(\coord_1)}{N(\coord_2\coord_7)}\right)^{1/2}.
  \end{equation}
  If there is an $\coord_8$ with
  $(\coord_6,\coord_7,\coord_8)\in\cutoff_6^{(w)}(\classtuple,\coordtuple',\divisoridealtuple;
  B)$ then this expression is indeed $\gg 1$: Since
  $\divisorideal_{68}\divisorideal_{69}\mid \coordideal_6$ and
  $\divisorideal_8 \mid \coordideal_3\coordideal_4\coordideal_5$, we
  have $\N(\divisorideal_8\divisorideal_{68}\divisorideal_{69}) \ll
  N(\coord_3\coord_4\coord_5\coord_6)$. Thus,
  \begin{equation*}
    \N\latticefracideal_8^2N(\coord_1^{-1}\coord_2\coord_7)\ll N(\coord_1\coord_2\coord_7)\N(\divisorideal_8\divisorideal_{68}\divisorideal_{69})^2\ll N(\coord_1\coord_2\coord_3^2\coord_4^2\coord_5^2\coord_6^2\coord_7) \ll B,
  \end{equation*}
  by \eqref{eq:A3+A1_height}.

  Next, we still fix $\coordtuple', \divisoridealtuple, \coord_6$ and
  sum the expression in \eqref{eq:sum_a8} over all
  $\coord_7\in\latticefracideal_7^{\neq 0}$ with
  \eqref{eq:cusps_a7_bound} for all $v \in \archplaces$. By Lemma
  \ref{lem:norm_sum_no_lower_bound}, the result is
  \begin{equation*}
    \ll \frac{1}{\N\latticefracideal_8}\left(\frac{BN(\coord_1)}{N(\coord_2)}\right)^{1/2} \cdot \frac{1}{\N\latticefracideal_7}\left(\frac{B}{N(\coord_3^2\coord_4^3\coord_5^4\coord_6)}\right)^{1/4}.
  \end{equation*}
  We use Lemma \ref{lem:norm_sum_no_lower_bound} again to sum this
  over all $\coord_6 \in \latticefracideal_6^{\neq 0}$ with
  $\abs{\coord_6}_w \leq \N\latticefracideal_6^{d_w/d}$ and
  \eqref{eq:cusps_a6_bound} for all $v \in
  \archplaces\smallsetminus\{w\}$. Keeping \eqref{eq:conj_norm} in
  mind, we see that $|\cutoff_{6}^{(w)}(\classtuple, \coordtuple', \divisoridealtuple; B)|$ is
  \begin{align*}
    &\ll \frac{1}{\N\latticefracideal_8}\left(\frac{BN(\coord_1)}{N(\coord_2)}\right)^{\frac{1}{2}} \cdot\frac{1}{\N\latticefracideal_7}\left(\frac{B}{N(\coord_3^2\coord_4^3\coord_5^4)}\right)^{\frac{1}{4}} \cdot \frac{1}{\N\latticefracideal_6^{1-\frac{3d_w}{4d}}}\left(\frac{B}{N(\coord_1^2\coord_2^2\coord_3^2\coord_4)}\right)^{\frac{1}{4} - \frac{d_w}{4d}}\\
    &\ll
    \frac{B}{\N\latticefracideal_6^{1-\frac{3d_w}{4d}}\N\latticefracideal_7\N\latticefracideal_8N(\coord_2\coord_3\coord_4\coord_5)}\left(\frac{B}{N(\coord_1^2\coord_2^2\coord_3^2\coord_4)}\right)^{-\frac{d_w}{4d}}.
  \end{align*}
  Lemma \ref{lem:error_sum} with \eqref{eq:cusps_aj_bound} and \eqref{eq:cusps_a1_bound} now shows the claimed estimate.
\end{proof}

\begin{lemma}\label{lem:a7_conj_lower_bound}
  Let $\classtuple \in \classrepsyst^6$ and $\epsilon > 0$. Then, for $B\geq 3$,
  \begin{equation*}
    \sum_{\coordtuple' \in \FDei^5\cap\coordfracidealproduct}\theta_0(\coordidealtuple')\sum_{\substack{\divisoridealtuple\\\eqref{eq:dij_conditions},\eqref{eq:di_conditions}}}|\mu_K(\divisoridealtuple)|\cdot |\cutoff_7^{(w)}(\classtuple, \coordtuple', \divisoridealtuple; B)| \ll_\epsilon B(\log B)^{4+\epsilon}.
  \end{equation*}
\end{lemma}

\begin{proof}
  If $|\archplaces| = 1$ then the left-hand side is $0$. Hence, assume
  that $|\archplaces| \geq 2$.  As in the previous lemma, we start by
  fixing $\coordtuple', \divisoridealtuple, \coord_6, \coord_7$ and
  see that the number of $\coord_8$ with $(\coord_6,\coord_7,\coord_8)
  \in \cutoff_7^{(w)}$ is bounded by \eqref{eq:sum_a8}. We apply Lemma
  \ref{lem:norm_sum_no_lower_bound} to sum this over all
  $\coord_7\in\latticefracideal_7^{\neq 0}$ with $\abs{\coord_7}_w
  \leq \N\latticefracideal_7^{d_w/d}$ and \eqref{eq:cusps_a7_bound}
  for all $v \in\archplaces\smallsetminus\{w\}$. This gives the bound
  \begin{equation}\label{eq:a7_a8_bound}
    \ll \frac{1}{\N\latticefracideal_8}\left(\frac{BN(\coord_1)}{N(\coord_2)}\right)^{\frac{1}{2}}\cdot \frac{1}{\N\latticefracideal_7^{1-\frac{d_w}{2d}}}\left(\frac{B}{N(\coord_3^2\coord_4^3\coord_5^4)}\right)^{\frac{1}{4}-\frac{d_w}{4d}}\cdot\prod_{v\neq w}\frac{1}{\absv{\coord_6}^{\frac{1}{4}}}.
  \end{equation}
  Our further procedure depends on $\coordtuple'$. We first consider
  all $\coordtuple'$ that satisfy the additional condition
  \eqref{eq:additional_height_condition}.

  In this case, we note that $\prod_{v\neq w}\absv{\coord_6}^{-1} =
  N(\coord_6)^{-1}\abs{\coord_6}_w$ and estimate $\abs{\coord_6}_w$
  by \eqref{eq:cusps_a6_bound} and \eqref{eq:conj_norm}. Hence, the
  expression in \eqref{eq:a7_a8_bound} is
  \begin{equation*}
    \ll \frac{1}{\N\latticefracideal_8}\left(\frac{B N(\coord_1)}{N(\coord_2)}\right)^{\frac{1}{2}}\cdot\frac{1}{\N\latticefracideal_7^{1-\frac{d_w}{2d}}}\left(\frac{B}{N(\coord_3^2\coord_4^3\coord_5^4)}\right)^{\frac{1}{4}-\frac{d_w}{4d}}\cdot\left(\frac{B}{N(\coord_1^2\coord_2^2\coord_3^2\coord_4)}\right)^{\frac{d_w}{12 d}}\cdot \frac{1}{N(\coord_6)^{\frac{1}{4}}}.
  \end{equation*}
  Using Lemma \ref{lem:norm_sum_no_lower_bound} again to sum this over
  all $\coord_6 \in \latticefracideal_6^{\neq 0}$ with
  \eqref{eq:cusps_a6_bound}, we get the upper bound
  \begin{equation*}
    \ll \frac{B}{\N\latticefracideal_6\N\latticefracideal_7^{1-\frac{d_w}{2d}}\N\latticefracideal_8N(\coord_2\coord_3\coord_4\coord_5)}\cdot \left(\frac{N(\coord_1\coord_2)B}{N(\coord_3^2\coord_4^4\coord_5^6)}\right)^{-\frac{d_w}{6d}}.
  \end{equation*}
  We sum this over all $\coordtuple'$ with \eqref{eq:cusps_aj_bound}
  and \eqref{eq:additional_height_condition} and all
  $\divisoridealtuple$ using Lemma \ref{lem:error_sum}.

  Now, let us consider all $\coordtuple'$ with the additional
  condition
  \begin{equation}
    \label{eq:additional_height_condition_2}
    N(\coord_3^2\coord_4^4\coord_5^6)\gg N(\coord_1\coord_2)B.
  \end{equation}
  We already know that the number of $(\coord_7,\coord_8)$ for fixed
  $\coordtuple', \divisoridealtuple, \coord_6$ is bounded by
  \eqref{eq:a7_a8_bound}. For the existence of an
  $\coord_7 \in \latticefracideal_7^{\neq 0}$ with $\abs{\coord_7}_w
  \leq \N\latticefracideal_7^{d_w/d}$ and \eqref{eq:cusps_a7_bound}
  for all $v \neq w$, it is required that
  \begin{equation*}
    1 \ll \left(\frac{B}{N(\coord_3^2\coord_4^3\coord_5^4)}\right)^{\frac{1}{4}-\frac{d_w}{4d}}\cdot\prod_{v\neq w}\frac{1}{\absv{\coord_6}^{\frac{1}{4}}} \ll \left(\left(\frac{B}{N(\coord_3^2\coord_4^3\coord_5^4)}\right)^{\frac{1}{4}-\frac{d_w}{4d}}\cdot\prod_{v\neq w}\frac{1}{\absv{\coord_6}^{\frac{1}{4}}}\right)^{\frac{5d}{4(d-d_w)}}.
  \end{equation*}
  Hence, we may further estimate the expression in
  \eqref{eq:a7_a8_bound} by
  \begin{align*}
    &\frac{1}{\N\latticefracideal_8}\left(\frac{BN(\coord_1)}{N(\coord_2)}\right)^{\frac{1}{2}}\hspace{-0.2cm}\frac{1}{\N\latticefracideal_7^{1-\frac{d_w}{2d}}}\left(\frac{B}{N(\coord_3^2\coord_4^3\coord_5^4)}\right)^{\frac{5}{16}}\hspace{-0.1cm}\left(\frac{B}{N(\coord_1^2\coord_2^2\coord_3^2\coord_4)}\right)^{\frac{5 d_w}{48(d-d_w)}}\hspace{-0.2cm}\frac{1}{N(\coord_6)^{\frac{5d}{16(d-d_w)}}}.
  \end{align*}
  The exponent of $N(\coord_6)$ is in $(0,1)$ since $|\archplaces|
  \geq 2$ and $d_w \leq 2$. Hence, we can sum this over all $\coord_6
  \in \latticefracideal_6^{\neq 0}$ with \eqref{eq:cusps_a6_bound}
  using Lemma \ref{lem:norm_sum_no_lower_bound}. We obtain the bound
  \begin{equation*}
    \ll \frac{B}{\N\latticefracideal_6\N\latticefracideal_7^{1-\frac{d_w}{2d}}\N\latticefracideal_8N(\coord_2\coord_3\coord_4\coord_5)}\cdot\left(\frac{N(\coord_1\coord_2)B}{N(\coord_3^2\coord_4^4\coord_5^6)}\right)^{\frac{1}{24}}.
  \end{equation*}
  Again, we use Lemma \ref{lem:error_sum} to sum this over
  $\coordtuple'$ satisfying \eqref{eq:additional_height_condition_2}
  and all $\divisoridealtuple$.
\end{proof}

To recapitulate the results of Lemma \ref{lem:a6_conj_lower_bound} and Lemma
\ref{lem:a7_conj_lower_bound}, we introduce the sets
\begin{equation*}
  S_F^*(\classtuple, \coordtuple', \divisoridealtuple; B) := \{(x_{jv})_{j,v}\in S_F(\coordtuple'; B) \mid \forall v:\absv{x_{6v}}\geq \N\latticefracideal_6^{d_v/d}, \absv{x_{7v}}\geq \N\latticefracideal_7^{d_v/d}\}
\end{equation*}
and
\begin{equation*}
  \FDxi^*(\classtuple, \coordtuple', \divisoridealtuple; u_\classtuple B) := \{(\coord_6,\coord_7,\coord_8)\in (K^\times)^2\times K \mid \sigma(\coord_6,\coord_7,\coord_8)\in S_F^*(\classtuple, \coordtuple',\divisoridealtuple; u_\classtuple B)\}.
\end{equation*}
We have just proved that, for $\classtuple \in \classrepsyst^6$ and
$\epsilon > 0$,
\begin{equation}\label{eq:MCB_with_lower_bounds}
  \begin{aligned}
    |M_\classtuple(B)| &= \sum_{\coordtuple' \in
      \FDei^5\cap\coordfracidealproduct}\theta_0(\coordidealtuple')\sum_{\substack{\divisoridealtuple\\\eqref{eq:dij_conditions},\eqref{eq:di_conditions}}}\mu_K(\divisoridealtuple)\cdot|\countinggroup(\classtuple,
    \coordtuple', \divisoridealtuple) \cap \FDxi^*(\classtuple,
    \coordtuple',\divisoridealtuple; u_\classtuple B)|\\ &+
    O_\epsilon(B(\log B)^{4+\epsilon}).
  \end{aligned}
\end{equation}

The lower bounds for the $\absv{\coord_6}$, $\absv{\coord_7}$ allow us
to introduce \eqref{eq:additional_height_condition} as an additional
height condition:

\begin{lemma}\label{lem:additional_height_condition}
  For $\classtuple \in \classrepsyst^6$ and $\epsilon > 0$, we have
  \begin{equation}\label{eq:MCB_with_lower_bounds}
    \begin{aligned}
      |M_\classtuple(B)| &= \sum_{\substack{\coordtuple' \in
          \FDei^5\cap\coordfracidealproduct\\\eqref{eq:additional_height_condition}}}\theta_0(\coordidealtuple')\sum_{\substack{\divisoridealtuple\\\eqref{eq:dij_conditions},\eqref{eq:di_conditions}}}\mu_K(\divisoridealtuple)\cdot|\countinggroup(\classtuple,
      \coordtuple', \divisoridealtuple) \cap \FDxi^*(\classtuple, \coordtuple',\divisoridealtuple; u_\classtuple B)|\\
      &+ O_\epsilon(B(\log B)^{4+\epsilon}).
    \end{aligned}
  \end{equation}
\end{lemma}

\begin{proof}
  It is enough to prove that
  \begin{equation}\label{eq:sum_with_additional_height_condition_2}
    \sum_{\substack{\coordtuple' \in
        \FDei^5\cap\coordfracidealproduct\\\eqref{eq:additional_height_condition_2}}}\theta_0(\coordidealtuple')\sum_{\substack{\divisoridealtuple\\\eqref{eq:dij_conditions},\eqref{eq:di_conditions}}}|\mu_K(\divisoridealtuple)|\cdot|\countinggroup(\classtuple,
    \coordtuple', \divisoridealtuple) \cap \FDxi^*(\classtuple, \coordtuple',\divisoridealtuple; u_\classtuple B)|\ll B(\log B)^{4+\epsilon}.
  \end{equation}
  Again, we fix $\coordtuple', \divisoridealtuple, \coord_6, \coord_7$
  and bound the number of $\coord_8$ with
  $(\coord_6,\coord_7,\coord_8) \in \countinggroup(\classtuple,
  \coordtuple', \divisoridealtuple) \cap \FDxi^*(\classtuple,
  \coordtuple',\divisoridealtuple; u_\classtuple B)$ by
  \eqref{eq:sum_a8}.

  We sum this over all $\coord_6\in\latticefracideal_6^{\neq 0}$ with
  \eqref{eq:cusps_a7_bound} and obtain an upper bound
  \begin{align*}
    \ll
    \frac{1}{\N\latticefracideal_8}\left(\frac{BN(\coord_1)}{N(\coord_2)}\right)^{1/2}\cdot\frac{1}{\N\latticefracideal_6}\frac{B}{N(\coord_3^2\coord_4^3\coord_5^4)}\cdot\frac{1}{N(\coord_7)^{5/2}}.
  \end{align*}
  By Lemma \ref{lem:norm_sum_no_lower_bound}, the sum of this
  expression over all $\coord_7 \in \latticefracideal_7 \neq 0$ with
  $\absv{\coord_7}\geq\N\latticefracideal_7^{d_v/d}$ for all $v$ is
  \begin{equation*}
    \ll \frac{B}{\N\latticefracideal_6\N\latticefracideal_7^{5/2}\N\latticefracideal_8N(\coord_2\coord_3\coord_4\coord_5)}\left(\frac{N(\coord_1\coord_2)B}{N(\coord_3^2\coord_4^4\coord_5^6)}\right)^{1/2}.
  \end{equation*}
  We apply Lemma \ref{lem:error_sum} with \eqref{eq:additional_height_condition_2}.
\end{proof}

\section{Symmetries}\label{sec:linear_transformations}

In this section, we consider $\classtuple, \coordtuple',
\divisoridealtuple$ as fixed. From here on, it will be convenient to
write $\countinggroup := \countinggroup(\classtuple, \coordtuple',
\divisoridealtuple)$, $\FDxi^* := \FDxi^*(\classtuple,
\coordtuple',\divisoridealtuple; u_\classtuple B)$, and $S_F^* :=
S_F^*(\classtuple,\coordtuple',\divisoridealtuple; u_\classtuple B)$.

In Lemma \ref{lem:additional_height_condition}, we established that in
order to find an asymptotic formula for $|M_\classtuple(B)|$, we need
to count $|\countinggroup \cap \FDxi^*|$. The embedding $\sigma : K^3
\to \prod_{v\in\archplaces}K_v^3$ transforms this to
\begin{equation*}
  |\sigma(\countinggroup) \cap S_F^*|.
\end{equation*}
We use some symmetries of $S_F^*$ to facilitate our counting
problem.  For any $M \subseteq \archplaces$, let $S_F^{M} =
S_F^{M}(\classtuple,\coordtuple',\divisoridealtuple; u_\classtuple B)$
be the set of all $(x_{jv})_{j,v}\in S_F^*$ with
\begin{align*}
  &\absv{\sigma_v(\coord_2)x_{8v}} \geq \absv{\sigma_v(\coord_2)x_{8v} + \sigma_v(\coord_3\coord_4^2\coord_5^3)x_{7v}} \text{ for all } v \in M\\
  &\absv{\sigma_v(\coord_2)x_{8v}} \leq \absv{\sigma_v(\coord_2)x_{8v}
    + \sigma_v(\coord_3\coord_4^2\coord_5^3)x_{7v}} \text{ for all } v
  \notin M.
\end{align*}
Of these sets, $S_F^{\emptyset}$ is the most convenient to count
lattice points in it. Let $\phi_M : \prod_{v\in\archplaces}K_v^3\to
\prod_{v\in\archplaces}K_v^3$ be the $\RR$-linear involution given by
$\phi_M((x_{jv})_{j,v}) := (x_{jv}')_{j,v}$, with
\begin{align*}
  x_{6v}' &:= x_{6v} \text{ for } v\in\archplaces\\
  x_{7v}' &:=
  \begin{cases}
    x_{7v} &\text{ for } v\in\archplaces\smallsetminus M\\
    -x_{7v} &\text{ for } v \in M
  \end{cases}\\
  x_{8v}' &:=
  \begin{cases}
    x_{8v} &\text{ for } v \in \archplaces\smallsetminus M\\
    x_{8v}+\sigma_v(\coord_3\coord_4^2\coord_5^3/\coord_2) \cdot
    x_{7v} &\text{ for }v\in M.
  \end{cases}
\end{align*}
Then $|\det \phi_M| = 1$, and one readily verifies that
\begin{equation*}
  \tN_v(\coordtuple'; x_{6v}, x_{7v}, x_{8v}) = \tN_v(\coordtuple'; x_{6v}', x_{7v}', x_{8v}')
\end{equation*}
for all $v \in \archplaces$ and all $(x_{6v},x_{7v},x_{8v}) \in
K_v^3$. Therefore, $\phi_M$ induces a bijection between $S_F^{M}$ and
$S_F^{\emptyset}$.

Let $\tau : \prod_{v\in\archplaces}K_v^3 \to \prod_{v\in\archplaces}K_v^3$
be the $\RR$-endomorphism of determinant $\det \tau =
\N(\latticefracideal_6\latticefracideal_7\latticefracideal_8)^{-1}$
given by
\begin{equation*}
  \tau((x_{jv})_{j,v}) := (\N\latticefracideal_j^{-1/d}\cdot x_{jv})_{j,v}.
\end{equation*}
Define $\Lambda_M = \Lambda_M(\classtuple,\coordtuple',\divisoridealtuple) :=
\tau(\phi_M(\sigma(\countinggroup)))$.

\begin{lemma}\label{lem:symmetries}
  For any $\classtuple,\coordtuple',\divisoridealtuple$ as in Lemma
  \ref{lem:additional_height_condition}, we have
  \begin{align*}
    |\countinggroup\cap \FDxi^*| &= \sum_{M \subseteq
      \archplaces}|\Lambda_M \cap \tau(S_F^{\emptyset})| +
    O\left(\max_{\substack{M,N\subseteq
          \archplaces\\N\neq\emptyset}}|\Lambda_M \cap
      \tau(S_F^{\emptyset}\cap S_F^{N})| \right).
  \end{align*}
\end{lemma}

\begin{proof}
  Since $S_F^*$ is the union of the sets $S_F^M$, $M \subseteq
  \archplaces$, we have
  \begin{equation*}
    |\countinggroup \cap \FDxi^*| = |\sigma(\countinggroup)\cap S_F^*| = \sum_{M \subseteq \archplaces}|\sigma(\countinggroup)\cap S_F^M| + O\left(\max_{M\neq N\subseteq\archplaces}|\sigma(\countinggroup)\cap S_F^M \cap S_F^N|\right).
  \end{equation*}
  To prove the lemma, we apply $\tau \circ \phi_M$ to each summand and
  to the argument of the maximum in the error term.
\end{proof}

Let us collect some information about
$\Lambda_M(\classtuple,\coordtuple',\divisoridealtuple)$.

\begin{lemma}\label{lem:Lamda_J}
  The subset $\Lambda_M(\classtuple,\coordtuple',\divisoridealtuple)
  \subseteq \prod_{v\in\archplaces}K_v = \RR^{3d}$ is a lattice of
  rank $3d$ and determinant $(2^{-r_2}\sqrt{|\Delta_K|})^3$. Let
  $\lambda_1$ be its first successive minimum in the sense of
  Minkowski, with respect to the unit ball in $\RR^{3d}$. Then
  $\lambda_1 \ge 1$.
\end{lemma}

\begin{proof}
  It is well known that
  $\sigma(\latticefracideal_6\times\latticefracideal_7\times\latticefracideal_8)$
  is a lattice in $\prod_{v\in\archplaces}K_v^3$ of rank $3d$ and
  determinant
  $(2^{-r_2}\sqrt{|\Delta_K|})^3\N(\latticefracideal_6\latticefracideal_7\latticefracideal_8)$.

  It is clear from the definition of $\countinggroup(\classtuple,
  \coordtuple', \divisoridealtuple)$ that
  $\sigma(\countinggroup(\classtuple,\coordtuple',\divisoridealtuple))$
  arises from this lattice via the $\RR$-endomorphism $\phi:
  \prod_{v\in\archplaces}K_v^3 \to \prod_{v\in\archplaces}K_v^3$ of
  determinant $1$ defined by
  \begin{equation*}
    \phi((x_{jv})_{j,v})_{iw}=
    \begin{cases}
      x_{iw} &\text{ if }i \in \{6,7\},\\
      \gamma_8^{(w)}x_{7w}+x_{8w} &\text{ if }i=8.
    \end{cases}
  \end{equation*}
  Hence,
  $\sigma(\countinggroup(\classtuple,\coordtuple',\divisoridealtuple))$
  is a lattice of the same rank and determinant. Since
  $\tau\circ\phi_M$ is a linear transformation with
  $|\det(\tau\circ\phi)|=
  \N(\latticefracideal_6\latticefracideal_7\latticefracideal_8)^{-1}\neq
  0$, the set $\Lambda_M(\classtuple,\coordtuple',\divisoridealtuple)$
  is a lattice, and its rank and determinant are as claimed.

  We still need to consider $\lambda_1$. To this end, let
  $0\neq(\coord_6,\coord_7,\coord_8) \in
  \countinggroup(\classtuple,\coordtuple',\divisoridealtuple)$. We
  show that $(\tau\circ\phi_M\circ\sigma)(\coord_6,\coord_7,\coord_8)$
  has length $\ge 1$. Assume first that $\coord_6 \neq 0$. Since
  $\tau(\phi_M(\sigma(\coord_6,\coord_7,\coord_8)))_{6v} =
  \N\latticefracideal_6^{-1/d}\cdot\coord_6^{(v)}$ for all
  $v\in\archplaces$, we have
  \begin{equation*}
    |\tau(\phi_M(\sigma(\coord_6,\coord_7,\coord_8)))| \geq |\N\latticefracideal_6^{-1/d}\sigma(\coord_6)|\geq 1.
  \end{equation*}
  In the second inequality we used the fact that the first successive
  minimum of $\sigma(\latticefracideal_6)$ is at least
  $\N\latticefracideal_6^{1/d}$ (cf. \cite[Lemma 5]{MR2247898}). A
  similar argument shows the statement if $\coord_7 \neq 0$, and if
  $\coord_6=\coord_7=0$ and $\coord_8\neq 0$.
\end{proof}

\section{Definability in an o-minimal structure}\label{sec:definability}
In Lemma \ref{lem:symmetries}, we reduced our counting problem to
controlling the quantities
\begin{equation*}
  |\Lambda_M\cap \tau(S_F^\emptyset\cap S_F^N)|,
\end{equation*}
for $M, N \subseteq \archplaces$. We already know the determinant and
a lower bound for the first successive minimum of the lattice
$\Lambda_M = \Lambda_M(\classtuple,\coordtuple',\divisoridealtuple)$.

To count the lattice points in $\tau(S_F^\emptyset\cap S_F^N)$, we use
a technique going back to Davenport \cite{MR0043821}, which was
recently adapted to the framework of $o$-minimal structures by
Barroero and Widmer \cite{arXiv:1210.5943}. We will apply
\cite[Theorem 1.3]{arXiv:1210.5943}, so our sets
$\tau(S_F^\emptyset\cap S_F^N)$ should be fibers of definable families
$Z^{(N)}$ with bounded fibers in an o-minimal structure. For a quick
introduction to o-minimal structures, we refer to the survey
\cite{sem:Wilkie2007}.

By \eqref{eq:conj_norm}, there is a constant $c_1 \gg 1$ such that
$\absv{\coord}\geq c_1$ for all $v\in\archplaces$ and $\coord \in
\FDei\cap\coordfracideal_{j*}$, with $j \in \{1, \ldots, 5\}$.

Let $Z^{(N)}$ be the set of all
\begin{equation*}
  (\beta, \beta_6, \beta_7, \beta_8, (x_{jv})_{\substack{1\leq j \leq 8\\v\in\archplaces}})\in \RR^4 \times \prod_{v\in\archplaces}K_v^8
\end{equation*}
that satisfy the conditions
\begin{equation*}
  \begin{aligned}
    \beta, \beta_3, \beta_7, \beta_8 &> 0,\\
    \absv{x_{jv}}&\geq c_1 \text{ for all }j\in\{1,\ldots, 5\}, v\in\archplaces,\\
    \absv{x_{6v}}, \absv{x_{7v}}&\geq 1 \text{ for all }v\in\archplaces,\\
    \absv{x_{2v}\beta_8x_{8v}}&\leq \absv{x_{2v}\beta_8x_{8v} +
      x_{3v}x_{4v}^2x_{5v}^3\beta_7x_{7v}} \text{ for all }v\in\archplaces\smallsetminus N,\\
    \absv{x_{2v}\beta_8x_{8v}}&= \absv{x_{2v}\beta_8x_{8v} +
      x_{3v}x_{4v}^2x_{5v}^3\beta_7x_{7v}} \text{ for all }v\in N,\\
  \end{aligned}
\end{equation*}
\begin{equation*}
\left(\tN_v(x_{1v},\ldots, x_{5v}, \beta_6x_{6v},\beta_7x_{7v},\beta_8x_{8v})^{1/3}\right)_{v\in\archplaces} \in \exp(F(\beta^{1/(3d)})),
\end{equation*}
where $\exp : \RR^\archplaces \to \RR_{>0}^{\archplaces}$ is the
coordinate-wise exponential function. For
\begin{equation*}
  T = (\beta, \beta_3, \beta_7, \beta_8, (x_{jv})_{\substack{1\leq j\leq 5\\v\in\archplaces}}) \in \RR^4 \times \prod_{v\in\archplaces}K_v^5,
\end{equation*}
we define the fiber
\begin{equation*}
  Z_T^{(N)} := \left\{(x_{jv})_{\substack{j\in\{6,7,8\}\\v\in\archplaces}}\in \prod_{v\in\archplaces}K_v^3\ \Big| \ (\beta, \beta_3, \beta_7, \beta_8, (x_{jv})_{\substack{1\leq j\leq 8\\v\in\archplaces}}) \in Z^{(N)}\right\}.
\end{equation*}
We see immediately from the definitions that $\tau(S_F^\emptyset\cap
S_F^N)$ is just the fiber $Z_T^{(N)}$, where
\begin{equation*}
  T := (u_\classtuple B, \N\latticefracideal_6^{1/d},\N\latticefracideal_7^{1/d}, \N\latticefracideal_8^{1/d}, (\coord_{j}^{(v)})_{\substack{j\in\{1,\ldots,5\}\\v\in\archplaces}}).
\end{equation*}
Hence, the following lemma allows us to apply \cite[Theorem
1.3]{arXiv:1210.5943}. Recall that we identify
$\prod_{v\in\archplaces}K_v^8$ with $\RR^{8d}$ by identifying $K_v$
with $\RR$ or $\RR^2$.

\begin{lemma}\label{eq:definable_in_R_exp}
  For any $N\subseteq \archplaces$, the subset $Z^{(N)} \subseteq
  \RR^{4+8d}$ is definable in the o-minimal structure $\RR_{\exp} =
  \langle\RR;<, +,\cdot,-,\exp\rangle$. Moreover, the fibers
  $Z_T^{(N)}$ are bounded.
\end{lemma}

\begin{proof}
  O-minimality of the structure $\RR_{\exp}$ is a well-known
  consequence of Wilkie's theorem \cite{MR1398816}. After recalling
  the definitions of $F$ and $F(B) \subseteq \RR^\archplaces =
  \RR^{q+1}$ from Section \ref{sec:fundamental_domain}, it is clear
  that $Z^{(N)}$ is definable in $\RR_{\exp}$. Since
  $\exp(F(\beta^{1/(3d)}))$ is bounded for any fixed $\beta>0$,
  boundedness of the fibers follows at once.
\end{proof}

\section{Volumes of projections}\label{sec:volumes_of_projections}
For any coordinate subspace $W$ of
$\prod_{v\in\archplaces}K_v^3=\RR^{3d}$, obtained by equating some
coordinates in $\RR^{3d}$ to $0$, we write $V_W =
V_W(\classtuple,\coordtuple',\divisoridealtuple;u_\classtuple B)$ for
the $(\dim W)$-dimensional volume (i.e.,~Lebesgue measure) of the orthogonal
projection of $\tau(S_F^\emptyset)$ to $W$. By convention, the
zero-dimensional volume of a point is $1$. The following lemma
summarizes our progress of the last sections.

\begin{lemma}\label{lem:barroero_widmer_estimate}
  For any $\classtuple,\coordtuple',\divisoridealtuple$ as in Lemma
  \ref{lem:additional_height_condition}, we have
  \begin{equation*}
    |\countinggroup\cap \FDxi^*| = \frac{2^{3r_2}\vol S_F^*}{|\Delta_K|^{3/2}\N(\latticefracideal_6\latticefracideal_7\latticefracideal_8)} + O\left(\sum_{W}V_W\right).
  \end{equation*}
  The implied constant in the error term depends only on $K$, and $W$
  runs over all proper coordinate subspaces of $\RR^{3d}$.
\end{lemma}

\begin{proof}
  We start from Lemma \ref{lem:symmetries}. By the results of the
  previous section, the sets $\tau(S_F^\emptyset \cap S_F^N)$ are
  fibers of families $Z^{(N)}$ definable in the o-minimal structure
  $\RR_{\exp}$. Hence, by \cite[Theorem 1.3]{arXiv:1210.5943} and
  Lemma \ref{lem:Lamda_J},
  \begin{equation*}
    |\Lambda_M \cap \tau(S_F^\emptyset \cap S_F^N)| = \frac{\vol(\tau(S_F^\emptyset\cap S_F^{N}))}{\det \Lambda_M} + O\left(\sum_{j=0}^{3d-1}V_j^{(N)}\right),
  \end{equation*}
  where $V_j^{(N)}$ is the sum of the $j$-dimensional volumes of the
  orthogonal projections of $\tau(S_F^\emptyset\cap S_F^{N})$ to all
  $j$-dimensional coordinate spaces of $\RR^{3d}$.

  If $N\neq \emptyset$ then $\vol(\tau(S_F^\emptyset\cap S_F^N))=0$.
  Moreover, $\tau(S_F^\emptyset\cap S_F^N)\subseteq
  \tau(S_F^\emptyset)$, so the same inclusion holds for the
  projections.

  For $N = \emptyset$, we have $\vol(\tau(S_F^\emptyset)) =
  \vol(S_F^\emptyset)/\N(\latticefracideal_6\latticefracideal_7\latticefracideal_8)
  =
  \vol(S_F^M)/\N(\latticefracideal_6\latticefracideal_7\latticefracideal_8)$
  for all $M \subseteq \archplaces$. Since $S_F^*$ is the union of all
  $S_F^M$ and the intersection of any two of them has volume zero, the
  lemma follows immediately.
\end{proof}

Our next goal is to find good estimates for the $V_W$. Recall that all
$(x_{jv})_{j,v} \in S_F^\emptyset$ satisfy
\begin{equation*}
  \begin{aligned}
    \absv{x_{6v}} &\ge \N\latticefracideal_6^{d_v/d},\\
    \absv{x_{7v}} &\ge \N\latticefracideal_7^{d_v/d},\\
    \absv{\sigma_v(\coord_2)x_{8v}} &\leq \absv{\sigma_v(\coord_2)x_{8v} +
      \sigma_v(\coord_3\coord_4^2\coord_5^3)x_{7v}},\\
    \tN_v(\coordtuple';x_{6v},x_{7v},x_{8v}) &\ll B^{d_v/d},
  \end{aligned}
\end{equation*}
for all $v \in \archplaces$. Let
\begin{equation*}
  c_6 := \left(\frac{B}{N(\coord_1^2\coord_2^2\coord_3^2\coord_4)}\right)^{1/3},\quad c_7 := \left(\frac{B}{N(\coord_3^2\coord_4^3\coord_5^4)}\right)^{1/2},\quad c_8 := \left(\frac{N(\coord_1)B}{N(\coord_2)}\right)^{1/2}.
\end{equation*}
Using \eqref{eq:conj_norm}, we see that every $(x_{jv})_{j,v} \in
S_F^\emptyset$ satisfies in particular, for $v \in \archplaces$,
\begin{align}
  &\N\latticefracideal_6^{d_v/d}\ll \absv{x_{6}}\ll c_6^{d_v/d},\label{eq:e3_bound_split}\\
  &\N\latticefracideal_7^{d_v/d} \ll \absv{x_{7}}\ll c_7^{d_v/d}\cdot \frac{1}{\absv{x_6}^{1/2}},\label{eq:e7_bound_split}\\
  &\absv{x_8} \ll c_8^{d_v/d}\cdot\frac{1}{\absv{x_7}^{1/2}}\label{eq:e8_bound_split},\\
  &\N\latticefracideal_8^{d_v/d} \ll c_8^{d_v/d}\cdot\frac{1}{\absv{x_7}^{1/2}}.\label{eq:c8_lower_bound}
\end{align}
Here, \eqref{eq:e3_bound_split} -- \eqref{eq:e8_bound_split} follow
directly from the properties listed above, and
\eqref{eq:c8_lower_bound} follows similarly as in the paragraph after
(\ref{eq:sum_a8}).

For fixed $\classtuple, \coordtuple', \divisoridealtuple$, and $v \in
\archplaces$, let $S^{(v)}_F =
S_F^{(v)}(\classtuple,\coordtuple',\divisoridealtuple;u_\classtuple
B)$ be the set of all $(x_6,x_7,x_8)\in K_v^3$ that satisfy
\eqref{eq:e3_bound_split}--\eqref{eq:c8_lower_bound}.

Let $\tau_v : K_v^3 \to K_v^3$, $(x_6,x_7,x_8)\mapsto
(\N\latticefracideal_6^{-1/d}x_6, \N\latticefracideal_7^{-1/d}x_7,
\N\latticefracideal_8^{-1/d}x_8)$. Then
\begin{equation}\label{eq:contained_in_product}
  \tau(S_F^\emptyset) \subseteq \prod_{v\in\archplaces} \tau_v(S^{(v)}_F).
\end{equation}
Hence, each projection of $\tau(S_F^\emptyset)$ to a coordinate subspace
is contained in a product of projections of the $\tau_v(S^{(v)}_F)$ to
coordinate subspaces in $K_v^3=\RR^{3d_v}$. Let us investigate these
projections. In our estimates, we will use the quantities
\begin{align*}
  s_0 &:= \frac{c_8c_7^{1/2}c_6^{3/4}}{\N\latticefracideal_6\N\latticefracideal_7\N\latticefracideal_8}=\frac{B}{\N\latticefracideal_6\N\latticefracideal_7\N\latticefracideal_8N(\coord_2\coord_4\coord_5\coord_6)},\\
  s_6 &:= \frac{c_8c_7^{1/2}c_6^{1/4}}{\N\latticefracideal_6^{1/2}\N\latticefracideal_7\N\latticefracideal_8}=\frac{B}{\N\latticefracideal_6^{\frac{1}{2}}\N\latticefracideal_7\N\latticefracideal_8N(\coord_2\coord_3\coord_4\coord_5)}\left(\frac{B}{N(\coord_1^2\coord_2^2\coord_3^2\coord_4)}\right)^{-\frac{1}{6}},\\
  s_7 &:= \frac{c_8c_7^{1/4}c_6^{7/8}}{\N\latticefracideal_6\N\latticefracideal_7^{1/2}\N\latticefracideal_8}=\frac{B}{\N\latticefracideal_6\N\latticefracideal_7^{\frac{1}{2}}\N\latticefracideal_8N(\coord_2\coord_3\coord_4\coord_5)}\left(\frac{N(\coord_1\coord_2)B}{N(\coord_3^{2}\coord_4^{4}\coord_5^{6})}\right)^{-\frac{1}{12}},\\
  s_8 &:=\frac{c_8^{1/2}c_7^{3/4}c_6^{5/8}}{\N\latticefracideal_6\N\latticefracideal_7\N\latticefracideal_8^{1/2}}=
  \frac{B}{\N\latticefracideal_6\N\latticefracideal_7\N\latticefracideal_8^{\frac{1}{2}}N(\coord_1)^{\frac{1}{2}}N(\coord_2\coord_3\coord_4\coord_5)}\left(\frac{N(\coord_3\coord_4^2\coord_5^3)B}{N(\coord_1^2\coord_2^2)}\right)^{-\frac{1}{6}}.
\end{align*}

\subsection{Real places}
Here, we investigate $\tau_v(S_F^{(v)})$ when $v$ is a real place, so
$K_v^3 = \RR^3$.

\begin{lemma}\label{lem:vol_proj_estimate_real}
  Let $v \in \archplaces$ be a real place. For any $P = (p_6,p_7,p_8)
  \in \{0,1\}^3$, let $V_P$ be the $(3-(p_6+p_7+p_8))$-dimensional
  volume of the orthogonal projection of $\tau_v(S_F^{(v)})$ to the
  coordinate subspace of $\RR^3$ given by
  \begin{equation}\label{eq:def_proj_real}
    \begin{aligned}
      x_j = 0  &\text{ for all $j\in \{6,7,8\}$ with }p_j = 1.\\
    \end{aligned}
  \end{equation}
  Then
  \begin{equation*}
    V_P \ll
    \begin{cases}
      s_0^{1/d} &\text{ if } p_3=p_7=p_8=0,\\
      s_6^{1/d} &\text{ if } p_6 = 1,\ p_7=p_8=0,\\
      s_7^{1/d} &\text{ if } p_7 = 1,\ p_8=0,\\
      s_8^{1/d} &\text{ if } p_8 =1.\\
    \end{cases}
  \end{equation*}
\end{lemma}

\begin{remark}
  The bounds $s_j$ are adapted to the complex case (Lemma
  \ref{lem:vol_proj_estimate_complex}), so the following proof is more
  complicated than it could be with different bounds.
\end{remark}

\begin{proof}
  We may assume that $S_F^{(v)} \neq \emptyset$. Let $W_P$ be the
  projection of $S_F^{(v)}$ to the subspace given by
  \eqref{eq:def_proj_real}, which we identify with
  $\RR^{3-p_1-p_2-p_3}$. Since $\tau_v$ is just a rescaling of the
  coordinates,
  \begin{equation}\label{eq:Vp_Wp_real}
    V_P = \frac{\vol(W_P)}{(\N\latticefracideal_6^{1-p_6}\N\latticefracideal_7^{1-p_7}\N\latticefracideal_8^{1-p_8})^{1/d}}.
  \end{equation}
  For any $\xx'=(x_j)_{\substack{j\in\{6,7,8\}\\p_j=0}} \in W_P$, we
  consider the point $\mathbf{y}(\xx')=(y_6,y_7,y_8)$ with
  \begin{equation*}
    y_j :=
    \begin{cases}
      x_j &\text{ if } p_j = 0,\\
      \N\latticefracideal_j^{1/d} &\text{ if } p_j = 1.
    \end{cases}
  \end{equation*}
  Then it is not hard to see from
  \eqref{eq:e3_bound_split}--\eqref{eq:c8_lower_bound} that
  $\mathbf{y}(\xx')$ is an element of the set $D\subseteq \RR^3$
  defined by the following conditions:
  \begin{align}
    \N\bbb_6^{1/d} \ll \absv{y_6} &\ll c_6^{1/d},\label{eq:y3_bound_real}\\
    \N\bbb_7^{1/d} \ll \absv{y_7} &\ll \frac{c_7^{1/d}}{\absv{y_6}^{1/2}},\label{eq:y7_bound_real}\\
    \absv{y_8} &\ll
    \frac{c_8^{1/d}}{\absv{y_7}^{1/2}}\label{eq:y8_bound_real}.
  \end{align}
  In the following integrals, $\dd y_j$ indicates the usual Lebesgue
  measure on $\RR$ if $p_j=0$, and the Dirac measure at the point
  $\N\latticefracideal_j^{1/d}$ if $p_j=1$. Then
  \begin{equation*}
    \vol(W_P) \leq  \int_{\mathbf{y}(\xx') \in D}\prod_{p_j= 0}\dd x_j = \int_{\eqref{eq:y3_bound_real}-\eqref{eq:y8_bound_real}}\dd  y_6\dd y_7 \dd y_8.
  \end{equation*}
  If $P = (0,0,0)$, this implies $\vol(W_P) \ll c_8^{1/d} c_7^{1/(2d)}
  c_6^{3/(4d)}$, which, together with \eqref{eq:Vp_Wp_real}, proves
  the lemma in this case. Next, let $p_7=p_8=0$ and $p_6 = 1$. Then
  \begin{align*}
    \vol(W_P) &\leq
    \int_{(\N\bbb_6^{1/d},y_7,y_8)\in D}\dd y_7
    \dd y_8 \ll \frac{c_8^{1/d}c_7^{1/(2d)}}{\N\latticefracideal_6^{1/(4d)}}
    \ll \frac{c_8^{1/d}c_7^{1/(2d)}c_6^{1/(4d)}}{\N\latticefracideal_6^{1/(2d)}}.
  \end{align*}
  Again, together with \eqref{eq:Vp_Wp_real}, this provides the
  desired bound. Now let us investigate the cases with $p_8 = 0$ and
  $p_7 = 1$. Here, with $D_1$ denoting the set of all $(y_6,y_7) \in \RR^2$ that satisfy \eqref{eq:y3_bound_real} and \eqref{eq:y7_bound_real},
  \begin{align*}
    \vol(W_P) &\ll \frac{c_8^{1/d}}{\N\latticefracideal_7^{1/(2d)}}
    \int_{(y_6,\N\bbb_7^{1/d})\in D_1}\dd y_6 \ll
    \frac{c_8^{1/d}c_7^{1/(4d)}}{\N\latticefracideal_7^{1/(2d)}}
    \int_{\eqref{eq:y3_bound_real}} \frac{1}{\absv{y_6}^{1/8}}\dd y_6\\
    &\ll \frac{c_8^{1/d}c_7^{1/(4d)}}{\N\latticefracideal_7^{1/(2d)}}
    \begin{cases}
      c_6^{7/(8d)} &\text{ if }p_6=0,\\
      \N\latticefracideal_6^{-1/(8d)} &\text{ if }p_6=1
    \end{cases}\ll
    \frac{c_8^{1/d}c_7^{1/(4d)}c_6^{7/(8d)}}{\N\latticefracideal_7^{1/(2d)}\N\latticefracideal_6^{p_6/d}}.
  \end{align*}

  Finally, let us consider all $P$ with $p_8 = 1$. We have
  \begin{align*}
    \vol(W_P) &\ll
    \int_{(y_6,y_7,\N\bbb_8^{1/d})\in D}\dd y_6 \dd y_7 \ll
    \frac{c_8^{1/(2d)}}{\N\latticefracideal_8^{1/(2d)}}\int_{D_1}\frac{1}{\absv{y_7}^{1/4}}\dd
    y_6\dd y_7.
  \end{align*}
  For fixed $y_6$,
  \begin{align*}
    &\int_{\eqref{eq:y7_bound_real}}\frac{1}{\absv{y_7}^{1/4}}\dd y_7
    \ll
    \begin{cases}
      (c_7^{1/d}/\absv{y_6}^{1/2})^{3/4} &\text{ if }p_7=0,\\
      \N\latticefracideal_7^{-1/(4d)} &\text{ if }p_7=1\\
    \end{cases} \ll
    \frac{c_7^{3/(4d)}}{\N\latticefracideal_7^{p_7/d}\absv{y_6}^{3/8}},
  \end{align*}
  so
  \begin{align*}
    \vol(W_P) &\ll \frac{c_8^{1/(2d)}c_7^{3/(4d)}}{\N\latticefracideal_8^{1/(2d)}\N\latticefracideal_7^{p_7/d}}\int_{\eqref{eq:y3_bound_real}}\frac{1}{\absv{y_6}^{3/8}}\dd y_6\\
    &\ll
    \frac{c_8^{1/(2d)}c_7^{3/(4d)}}{\N\latticefracideal_8^{1/(2d)}\N\latticefracideal_7^{p_7/d}}
    \begin{cases}
      c_6^{5/(8d)}&\text{ if }p_6 = 0,\\
      \N\latticefracideal_6^{-3/(8d)}&\text{ if }p_6 = 1\\
    \end{cases}
    \ll
    \frac{c_8^{1/(2d)}c_7^{3/(4d)}c_6^{5/(8d)}}{\N\latticefracideal_8^{1/(2d)}\N\latticefracideal_7^{p_7/d}\N\latticefracideal_6^{p_6/d}}.
  \end{align*}
\end{proof}

\subsection{Complex places}
Now, we consider $\tau_v(S_F^{(v)})$ for complex places $v$. Then $d_v
= 2$ and $K_v^3 = \CC^3$, which we identify with $\RR^6$. The
following lemma and its proof are similar to the real case, but more
complicated. Recall that $\absv{\ \cdot\ } = \abs{\ \cdot\ }^2$ on
$\CC$.

\begin{lemma}\label{lem:vol_proj_estimate_complex}
  Let $v \in \archplaces$ be a complex place. For any $P =
  (p_6,p_7,p_8) \in \{0,1,2\}^3$, let $V_P$ be the
  $(6-(p_6+p_7+p_8))$-dimensional volume of the orthogonal projection
  of $\tau_v(S_F^{(v)})$ to one of the coordinate subspaces of $\CC^3
  = \RR^6$ given as follows: for every $j
  \in \{6,7,8\}$ with $p_j = 1$, we take one of the equations
  \begin{equation*}
     \Re x_j=0\quad \text{ or }\quad \Im x_j = 0, 
  \end{equation*}
  and for every $j \in \{6,7,8\}$ with $p_j = 2$, we take the equation
  \begin{equation*}
    x_j=0
  \end{equation*}
  to define the coordinate subspace. Then
  \begin{equation*}
    V_P \ll
    \begin{cases}
      s_0^{2/d} &\text{ if } p_6=p_7=p_8=0,\\
      s_6^{2/d} &\text{ if } p_6 \in \{1, 2\},\ p_7=p_8=0,\\
      s_7^{2/d} &\text{ if } p_7 \in \{1, 2\},\ p_8=0,\\
      s_8^{2/d} &\text{ if } p_8 \in \{1, 2\}.\\
    \end{cases}
  \end{equation*}
\end{lemma}

\begin{proof}
  Again, we may assume that $S_F^{(v)} \neq \emptyset$.  Clearly,
  $\tau_v(S_F^{(v)})$ is invariant with respect to swapping real and
  imaginary parts, so it suffices to consider projections to $\Im x_j
  = 0$ and $x_j = 0$. Then every $P \in \{0,1,2\}^3$ describes a
  unique coordinate subspace. Let $W_P$ be the projection of
  $S_F^{(v)}$ to this subspace, which we identify with
  $\RR^{6-p_6-p_7-p_8}$. Then
  \begin{equation}\label{eq:Vp_Wp}
    V_P = \frac{\vol(W_P)}{(\N\latticefracideal_6^{2-p_6}\N\latticefracideal_7^{2-p_7}\N\latticefracideal_8^{2-p_8})^{1/d}}.
  \end{equation}
  Let $\xx'=(x'_j)_{\substack{j\in\{6,7,8\}\\p_j\in\{0,1\}}}\in W_P$,
  i.e., there is an element $(x_6,x_7,x_8)\in S_F^{(v)}$ with
  \begin{equation*}
    x'_j =
    \begin{cases}
       x_j \in \CC=\RR^2 &\text{ if } p_j = 0,\\
       \Re x_j \in \RR &\text{ if } p_j = 1.
    \end{cases}
  \end{equation*}
  Similarly as in the real case, we consider the point
  $\mathbf{y}(\xx')=(y_6,y_7,y_8)$ with
  \begin{equation*}
    y_j :=
    \begin{cases}
      x'_j &\text{ if } p_j \in \{0,1\},\\
      \N\latticefracideal_j^{1/d} &\text{ if } p_j = 2.
    \end{cases}
  \end{equation*}
  Let $K_j := \CC$ if $p_j=0$ and $K_j := \RR$ if $p_j \in
  \{1,2\}$. Then $\mathbf{y}(\xx')$ is an element of the subset $D
  \subseteq K_6\times K_7 \times K_8$ defined by
  \begin{align}
    \max\{\N\latticefracideal_6^{2/d},\absv{y_6}\} &\ll c_6^{2/d},\label{eq:y3_bound}\\
    \max\{\N\latticefracideal_7^{2/d},\absv{y_7}\} &\ll \frac{c_7^{2/d}}{\max\{\N\latticefracideal_6^{1/d}, \absv{y_6}^{1/2}\}},\label{eq:y7_bound}\\
    \absv{y_8} &\ll
    \frac{c_8^{2/d}}{\max\{\N\latticefracideal_7^{1/d},
      \absv{y_7}^{1/2}\}}.\label{eq:y8_bound}
  \end{align}
  In the following integrals, $\dd y_j$ indicates the Lebesgue measure
  on $\CC=\RR^2$ if $p_j=0$, the Lebesgue measure on $\RR$ if $p_j =
  1$, and the Dirac measure on $\RR$ at the point $\N\bbb_j^{1/d}$ if
  $p_j = 2$. Then
  \begin{equation*}
    \vol(W_P) \leq \int_{\mathbf{y}(\xx')\in D}\prod_{p_j\in\{0,1\}}\dd x_j' = \int_{\eqref{eq:y3_bound}-\eqref{eq:y8_bound}}\dd y_6 \dd y_7 \dd y_8.
  \end{equation*}
  As in Lemma \ref{lem:vol_proj_estimate_real}, we consider first the
  trivial case $P = (0,0,0)$. Using polar coordinates, we see that
  $\vol(W_P) \ll (c_8 c_7^{1/2} c_6^{3/4})^{2/d}$, which, together
  with \eqref{eq:Vp_Wp}, proves the lemma in this case. Next, let
  $p_7=p_8=0$ and $p_6 \in \{1,2\}$. Here we obtain
  \begin{align*}
    \vol(W_P) &\ll (c_8c_7^{1/2})^{2/d}\int_{\eqref{eq:y3_bound}}\frac{1}{\absv{y_6}^{1/4}}\dd y_6\\
    &\ll (c_8 c_7^{1/2})^{2/d}
    \begin{cases}
      c_6^{1/(2d)} &\text{ if }p_6 = 1,\\
      \N\latticefracideal_6^{-1/(2d)} &\text{ if }p_6 = 2
    \end{cases}
    \ll
    \frac{(c_8c_7^{1/2}c_6^{1/4})^{2/d}}{\N\latticefracideal_6^{(p_6-1)/d}}.
  \end{align*}
  The last estimate holds by \eqref{eq:y3_bound}. Together
  with \eqref{eq:Vp_Wp}, this provides the desired bound.

  Now let us investigate all $P$ with $p_8 = 0$ and $p_7 \in \{1,
  2\}$. Then
  \begin{align*}
    \vol(W_P) &\ll c_8^{2/d}
    \int_{\eqref{eq:y3_bound},\eqref{eq:y7_bound}}\frac{1}{\max\{\N\latticefracideal_7^{1/d},\absv{y_7}^{1/2}\}}\dd
    y_6\dd y_7.
  \end{align*}
  For any $y_6$ with
  $c_7^{2/d}/\max\{\N\latticefracideal_6^{1/d},\absv{y_6}^{1/2}\}\gg
  1$, we have
  \begin{align*}
    &\int_{\eqref{eq:y7_bound}}\frac{\dd y_7}{\max\{\N\latticefracideal_7^{1/d},\absv{y_7}^{1/2}\}} \ll
    \begin{cases}
      \log(c_7^{2/d}/\max\{\N\latticefracideal_6^{1/d},\absv{y_6}^{1/2}\}+2) &\text{ if }p_7 = 1\\
      \N\latticefracideal_7^{-1/d} &\text{ if }p_7 = 2\\
    \end{cases}\\
    &\ll
    \frac{1}{\N\latticefracideal_7^{(p_7-1)/d}}\left(\frac{c_7^{2/d}}{\max\{\N\latticefracideal_6^{1/d},\absv{y_6}^{1/2}\}}\right)^{1/4}
    \ll
    \frac{c_7^{1/(2d)}}{\N\latticefracideal_7^{(p_7-1)/d}\N\latticefracideal_6^{p_6/(8d)}}\frac{1}{\absv{y_6}^{(2-p_6)/16}}.
  \end{align*}
  Thus,
  \begin{align*}
    \vol(W_P) &\ll \frac{(c_8
      c_7^{1/4})^{2/d}}{\N\latticefracideal_7^{(p_7-1)/d}\N\latticefracideal_6^{p_6/(8d)}}\int_{\eqref{eq:y3_bound}}\frac{1}{\absv{y_6}^{(2-p_6)/16}}\dd
    y_6\\ &\ll \frac{(c_8
      c_7^{1/4})^{2/d}}{\N\latticefracideal_7^{(p_7-1)/d}\N\latticefracideal_6^{p_6/(8d)}}
    \begin{cases}
      c_6^{7/(4d)}& \text{ if } p_6 = 0,\\
      c_6^{7/(8d)}& \text{ if } p_6 = 1,\\
      1 & \text{ if } p_6 = 2
    \end{cases}
    \ll
    \frac{(c_8c_7^{1/4}c_6^{7/8})^{2/d}}{\N\latticefracideal_7^{(p_7-1)/d}\N\latticefracideal_6^{p_6/d}}.
  \end{align*}
  
  Finally, let us consider all $P$ with $p_8 \in \{1,2\}$. We have
  \begin{align*}
    &\int_{\eqref{eq:y8_bound}}\dd y_8 \ll
    \begin{cases}
      (c_8^{2/d}/\max\{\N\latticefracideal_7^{1/d},\absv{y_7}^{1/2}\})^{1/2}&\text{ if } p_8 = 1,\\
      1 &\text{ if } p_8 = 2\\
    \end{cases}\\
    &\ll
    \frac{c_8^{1/d}}{\N\latticefracideal_8^{(p_8-1)/d}\max\{\N\latticefracideal_7^{1/d},\absv{y_7}^{1/2}\}^{1/2}}\ll
    \frac{c_8^{1/d}}{\N\latticefracideal_8^{(p_8-1)/d}\
      N\latticefracideal_7^{p_7/(4d)}}\cdot
    \frac{1}{\absv{y_7}^{(2-p_7)/8}},
  \end{align*}
  due to \eqref{eq:y8_bound}. Then
  \begin{align*}
    &\int_{\eqref{eq:y7_bound}}\frac{1}{\absv{y_7}^{(2-p_7)/8}}\dd y_7
    \ll
    \begin{cases}
      (c_7^{2/d}/\max\{\N\latticefracideal_6^{1/d},\absv{y_6}^{1/2}\})^{3/4} &\text{ if }p_7=0,\\
      (c_7^{2/d}/\max\{\N\latticefracideal_6^{1/d},\absv{y_6}^{1/2}\})^{3/8} &\text{ if }p_7=1,\\
      1 &\text{ if }p_7=2
    \end{cases}\\
    &\ll
    \frac{c_7^{3/(2d)}}{\N\latticefracideal_7^{3p_7/(4d)}\max\{\N\latticefracideal_6^{1/d},
      \absv{y_6}^{1/2}\}^{3/4}}\ll
    \frac{c_7^{3/(2d)}}{\N\latticefracideal_7^{3p_7/(4d)}\N\latticefracideal_6^{3p_6/(8d)}}\cdot\frac{1}{\absv{y_6}^{3(2-p_6)/16}},
  \end{align*}
  by \eqref{eq:y7_bound}. Hence,
  \begin{align*}
    &\vol(W_P) \ll \frac{(c_8^{1/2}c_7^{3/4})^{2/d}}{\N\latticefracideal_8^{(p_8-1)/d}\N\latticefracideal_7^{p_7/d}\N\latticefracideal_6^{3p_6/(8d)}}\int_{\eqref{eq:y3_bound}}\frac{1}{\absv{y_6}^{3(2-p_6)/16}}\dd y_6\\
    &\ll
    \frac{(c_8^{1/2}c_7^{3/4})^{2/d}}{\N\latticefracideal_8^{\frac{p_8-1}{d}}\N\latticefracideal_7^{\frac{p_7}{d}}\N\latticefracideal_6^{\frac{3p_6}{8d}}}
    \begin{cases}
      c_6^{5/(4d)}&\text{ if }p_6 = 0,\\
      c_6^{5/(8d)}&\text{ if }p_6 = 1,\\
      1&\text{ if }p_6 = 2
    \end{cases}
    \ll
    \frac{(c_8^{1/2}c_7^{3/4}c_6^{5/8})^{2/d}}{\N\latticefracideal_8^{\frac{p_8-1}{d}}\N\latticefracideal_7^{\frac{p_7}{d}}\N\latticefracideal_6^{\frac{p_6}{d}}}.\qedhere
  \end{align*}
\end{proof}
 
Next, we use the bounds from Lemma \ref{lem:vol_proj_estimate_real}
and Lemma \ref{lem:vol_proj_estimate_complex} to show that the sum
over all $\classtuple, \coordtuple', \divisoridealtuple$ of the error
term in Lemma \ref{lem:barroero_widmer_estimate} is sufficiently
small. We have already seen in Lemma
\ref{lem:additional_height_condition} that it suffices to sum over all
$\coordtuple'$ with \eqref{eq:additional_height_condition}. Moreover,
it is clearly enough to sum the error term over all $\classtuple, \coordtuple',
\divisoridealtuple$ with
\begin{equation}\label{eq:hc_satisfied}
  S_F^*(\classtuple,\coordtuple',\divisoridealtuple;u_\classtuple B) \neq \emptyset,
\end{equation}
since otherwise $|\countinggroup\cap \FDxi^*|$ and the main term are
both $0$.

\begin{lemma}\label{lem:sum_vol_projections}
  Let $\classrep \in \classrepsyst^6$. Let $W$ be a proper coordinate subspace of
  $\prod_{v\in\archplaces}K_v^3 = \RR^{3d}$, and let
  $V_W(\classtuple,\coordtuple',\divisoridealtuple; u_\classtuple B)$
  be the $\dim(W)$-dimensional volume of the orthogonal projection of
  $\tau(S_F^\emptyset(\classtuple,\coordtuple',\divisoridealtuple;u_\classtuple
  B))$ to $W$. For $\epsilon > 0$, we have
  \begin{equation}\label{eq:sum_V_estimate}
    \sum_{\substack{\coordtuple' \in
        \FDei^5\cap\coordfracidealproduct\\\eqref{eq:additional_height_condition}}}\theta_0(\coordidealtuple')\sum_{\substack{\divisoridealtuple\\\eqref{eq:dij_conditions},\eqref{eq:di_conditions}\\\eqref{eq:hc_satisfied}}}|\mu_K(\divisoridealtuple)|\cdot V_W(\classtuple,\coordtuple',\divisoridealtuple; u_\classtuple B)\ll_\epsilon B(\log B)^{5-1/d+\epsilon}.
  \end{equation}
\end{lemma}

\begin{proof}
  For $i \in \{0,6,7,8\}$, let
  \begin{equation*}
    \Sigma_i :=    \sum_{\substack{\coordtuple' \in
        \FDei^5\cap\coordfracidealproduct\\\eqref{eq:additional_height_condition}}}\sum_{\substack{\divisoridealtuple\\\eqref{eq:dij_conditions},\eqref{eq:di_conditions}\\\eqref{eq:hc_satisfied}}}|\mu_K(\divisoridealtuple)| \cdot s_i.
  \end{equation*}
  Let us start by estimating the $\Sigma_i$ from above. Condition \eqref{eq:hc_satisfied} has the consequences 
  \begin{align}
    N(\coord_j)&\ll B \text{ for all }i,j,\label{eq:vars_bounded_by_B}\\
    N(\coord_1^2\coord_2^2\coord_3^2\coord_4)&\leq \N(\coordfracideal_1^2\coordfracideal_2^2\coordfracideal_3^2\coordfracideal_4)B,\label{eq:height_2}\\
    N(\coord_1^2\coord_2^2)&\ll N(\coord_4\coord_5^3\coord_6^2)B.\nonumber
  \end{align}
  Using these and \eqref{eq:additional_height_condition} with Lemma
  \ref{lem:error_sum}, we see that $\Sigma_i \ll_\epsilon B(\log
  B)^{4+\epsilon}$ holds for $i \in \{6,7,8\}$. A simple computation
  using just \eqref{eq:vars_bounded_by_B} shows that $\Sigma_0
  \ll_\epsilon B(\log B)^{5+\epsilon}$.

  Now let us prove \eqref{eq:sum_V_estimate}. By
  \eqref{eq:contained_in_product}, the projection of
  $\tau(S_F^\emptyset(\classtuple,\coordtuple',\divisoridealtuple;u_\classtuple
  B))$ to $W$ is contained in a product of projections of
  $\tau_v(S_F^{(v)}(\classtuple,\coordtuple',\divisoridealtuple;u_\classtuple
  B))$ to subspaces of $K_v^3$. The volume of each such projection is
  bounded by an $s_{i(v)}^{d_v/d}$, so
  \begin{equation*}
    V_W(\classtuple, \coordtuple', \divisoridealtuple; u_\classtuple B) \leq \prod_{v\in\archplaces}s_{i(v)}^{d_v/d}.
  \end{equation*}
  Since $W$ is a proper subspace of $\prod_{v\in\archplaces}K_v^3$,
  there is at least one $v \in \archplaces$ with $i(v) \neq 0$. Using
  H\"older's inequality, we see that the sum in
  \eqref{eq:sum_V_estimate} is
  \begin{align*}
    &\leq \sum_{\substack{\coordtuple' \in
        \FDei^5\cap\coordfracidealproduct\\\eqref{eq:additional_height_condition}}}\sum_{\substack{\divisoridealtuple\\\eqref{eq:dij_conditions},\eqref{eq:di_conditions}\\\eqref{eq:hc_satisfied}}}|\mu_K(\divisoridealtuple)|\prod_{v\in\archplaces}s_{i(v)}^{d_v/d} \leq \prod_{v\in\archplaces}\Sigma_{i(v)}^{d_v/d}\\
    &\ll_\epsilon (B(\log B)^{5+\epsilon})^{(d-1)/d}(B(\log
    B)^{4+\epsilon})^{1/d} = B(\log B)^{5-1/d+\epsilon}.\qedhere
  \end{align*}
\end{proof}

\section{Completion of the first summation}\label{sec:first_summation}

We have already seen in Lemma \ref{lem:additional_height_condition}
that we may restrict ourselves to $\classtuple, \coordtuple'$ with
\eqref{eq:additional_height_condition}. Moreover,
$S_F^*(\classtuple,\coordtuple',\divisoridealtuple;u_\classtuple B) =
\emptyset$ unless \eqref{eq:height_2} holds.

Let us first show that, under these conditions,
$S_F^*(\classtuple,\coordtuple',\divisoridealtuple;u_\classtuple B)$
is not much smaller than $S_F(\coordtuple';u_\classtuple B)$, whose
volume we have already computed in Lemma \ref{lem:vol_SFB}.

\begin{lemma}\label{lem:B_vs_SFB}
  Let $\classtuple \in \classrepsyst^6$ and $\epsilon > 0$. Then
  \begin{equation*}
    \sum_{\substack{\coordtuple' \in
        \FDei^5\cap\coordfracidealproduct\\\eqref{eq:additional_height_condition},\eqref{eq:height_2}}}\sum_{\substack{\divisoridealtuple\\\eqref{eq:dij_conditions},\eqref{eq:di_conditions}}}|\mu_K(\divisoridealtuple)|\cdot \frac{\vol(S_F(\coordtuple'; u_\classtuple B)\smallsetminus S_F^*(\classtuple,\coordtuple',\divisoridealtuple;u_\classtuple B))}{\N(\latticefracideal_6\latticefracideal_7\latticefracideal_8)}\ll_\epsilon B(\log B)^{4+\epsilon}.
  \end{equation*}
\end{lemma}

\begin{proof}
  For $w \in \archplaces$ and fixed $\coordtuple',
  \divisoridealtuple, u_\classtuple B$, let $V_w^{(6)}$
  (resp.~$V_w^{(7)}$) be the volume of the subset of
  $S_F(\coordtuple'; u_\classtuple B)$ where $\abs{x_{6w}}_w <
  \N\latticefracideal_6^{d_w/d}$ (resp.~$\abs{x_{7w}}_w <
  \N\latticefracideal_7^{d_w/d}$). Let
  \begin{equation*}
    R_w := \prod_{\substack{v \in
        \archplaces\\v\neq w}}\int_{\tN_v(\coordtuple'; x_{6v}, x_{7v},
      x_{8v})\ll B^{d_v/d}}\dd x_{6v} \dd x_{7v} \dd
    x_{8v}.
  \end{equation*}
  Then \eqref{eq:Nv_same_size} implies that, for $j \in \{6,7\}$,
  \begin{align*}
    V_w^{(j)} &\ll R_w\cdot \int_{\substack{\tN_w(\coordtuple';
        x_{6w}, x_{7w}, x_{8w})\ll
        B^{d_w/d}\\\abs{\coord_{jw}}_w<\N\latticefracideal_j^{d_w/d}}}\dd
    x_{6w}\dd x_{7w}\dd x_{8w}.
  \end{align*}
  Let us bound $R_w$. For each $v\neq w$, we use the last term in the
  maximum in $\tN_v(\coordtuple'; x_{6v},x_{7v},x_{8v})$ to bound the
  integral over $x_{7v}, x_{8v}$ (cf.~\cite[Lemma 5.1, (5)]{MR2520770}
  for real $v$ and \cite[Lemma 3.4, (4)]{arXiv:1302.6151} for complex
  $v$) and the second term in the maximum to bound the integral over
  $x_{6v}$. With \eqref{eq:conj_norm}, this leads to
  \begin{equation*}
    R_w \ll \left(\frac{B}{N(\coord_2\coord_3\coord_4\coord_5)}\right)^{1-d_w/d}.
  \end{equation*}
  Recall that $\tN_w(\coordtuple'; x_{6w}, x_{7w}, x_{8w})\ll
  B^{d_w/d}$ implies
  \begin{align*}
    \abs{x_{7w}}_w\ll
    \left(\frac{B}{N(\coord_3^2\coord_4^3\coord_5^4)}\right)^{d_w/(2d)}\cdot\frac{1}{\abs{x_{6w}}_w^{1/2}}.
  \end{align*}
  Using the first term in the minimum in \cite[Lemma 5.1,
  (4)]{MR2520770} resp. \cite[Lemma 3.4, (2)]{arXiv:1302.6151} to
  bound the integral over $x_{8w}$ and the above inequality to bound
  the integral over $x_{7w}$, we obtain
  \begin{equation*}
    \int\limits_{\substack{\tN_w(\coordtuple'; x_{6w}, x_{7w},
        x_{8w})\ll
        B^{d_w/d}\\\abs{\coord_{6w}}_w<\N\latticefracideal_6^{d_w/d}}}\hspace{-1cm}\dd
    x_{6w}\dd x_{7w}\dd x_{8w} \ll \N\latticefracideal_6^{\frac{3d_w}{4d}}\left(\frac{B}{N(\coord_3^2\coord_4^3\coord_5^4)}\right)^{\frac{d_w}{4d}}\left(\frac{N(\coord_1)B}{N(\coord_2)}\right)^{\frac{d_w}{2d}}.
  \end{equation*}
  Therefore,
  \begin{equation*}
    \frac{V_w^{(6)}}{\N(\latticefracideal_6\latticefracideal_7\latticefracideal_8)} \ll \frac{1}{\N\latticefracideal_6^{1-3d_w/(4d)}\N(\latticefracideal_7\latticefracideal_8)}\cdot\frac{B}{N(\coord_2\coord_3\coord_4\coord_5)}\cdot\left(\frac{B}{N(\coord_1^2\coord_2^2\coord_3^2\coord_4)}\right)^{-d_w/(4d)}.
  \end{equation*}
  We sum this over $\coordtuple',\divisoridealtuple$ with Lemma \ref{lem:error_sum}. Similarly,
  \begin{align*}
    \int\limits_{\substack{\tN_w(\coordtuple';x_{6w}, x_{7w},
        x_{8w})\ll
        B^{d_w/d}\\\abs{x_{7w}}_w<\N\latticefracideal_7^{d_w/d}}}\hspace{-1cm}\dd
    x_{6w} \dd x_{7w} \dd x_{8w}\ll
    \left(\frac{B}{N(\coord_1^2\coord_2^2\coord_3^2\coord_4)}\right)^{\frac{d_w}{3d}}
    \N\latticefracideal_7^{\frac{d_w}{2d}} \left(\frac{B
        N(\coord_1)}{N(\coord_2)}\right)^{\frac{d_w}{2d}},
  \end{align*}
  and thus
  \begin{align*}
    \frac{V_w^{(7)}}{\N(\latticefracideal_6\latticefracideal_7\latticefracideal_8)}
    \ll
    \frac{1}{\N\latticefracideal_7^{1-d_w/(2d)}\N(\latticefracideal_6\latticefracideal_8)}\cdot\frac{B}{N(\coord_2\coord_3\coord_4\coord_5)}\cdot\left(\frac{N(\coord_1\coord_2)B}{N(\coord_3^2\coord_4^4\coord_5^6)}\right)^{-d_w/(6d)}.
  \end{align*}
  Again, we sum this up using Lemma \ref{lem:error_sum}.
\end{proof}

In the rest of this article, a product over $\ppp$ runs over non-zero
prime ideals of $\Ok$.  Let $\theta_1(\coordidealtuple')$ be the
arithmetic function defined by $\theta_1(\coordidealtuple') :=
\prod_{\ppp}\theta_{1,\ppp}(J_\ppp(\coordidealtuple'))$, where
$J_\ppp(\coordidealtuple') := \{j\in\{1,\ldots, 5\} : \ppp \mid
\coordideal_j\}$ and
\begin{equation*}
  \theta_{1,\ppp}(J) :=
  \begin{cases}
    \left(1-\frac{1}{\N\ppp}\right)^2\left(1+\frac{2}{\N\ppp}\right) &\text{ if }J = \emptyset,\\
    \left(1-\frac{1}{\N\ppp}\right)^2\left(1+\frac{1}{\N\ppp}\right) &\text{ if }J = \{1\}, \{2\},\\
    \left(1-\frac{1}{\N\ppp}\right)^2 &\text{ if }J = \{3\}, \{5\},\\
    \left(1-\frac{1}{\N\ppp}\right)^3 &\text{ if }J = \{4\}, \{3,4\}, \{4,5\},\\
    0 &\text{ otherwise.}
  \end{cases}
\end{equation*}

\begin{lemma}\label{lem:arith_function}
  For any $\coordidealtuple' \in \idealsk^5$, we have
  \begin{equation*}
    \theta_0(\coordidealtuple')\sum_{\substack{\divisoridealtuple\\\eqref{eq:dij_conditions},\eqref{eq:di_conditions}}}\frac{\mu_K(\divisoridealtuple)}{\N(\divisorideal_{6}\divisorideal_7\divisorideal_8\divisorideal_{67}\divisorideal_{68}\divisorideal_{69}(\divisorideal_{67}\cap\divisorideal_{68}\divisorideal_{69}))} = \theta_1(\coordidealtuple').
  \end{equation*}
\end{lemma}

\begin{proof}
  For any ideal $\coordideal\in\idealsk$, let
  $\phi_K^*(\coordideal):=\prod_{\ppp\mid\coordideal}(1+1/\N\ppp)$.
  For fixed $\coordidealtuple'$, we have
  \begin{equation*}
    \sum_{\substack{\divisorideal_{6},\divisorideal_7,\divisorideal_8\\\eqref{eq:di_conditions}}}\frac{\mu_K(\divisorideal_{6})\mu_K(\divisorideal_7)\mu_K(\divisorideal_8)}{\N(\divisorideal_{6}\divisorideal_7\divisorideal_8)} = \phi_K^*(\coordideal_4\coordideal_5)\phi_K^*(\coordideal_1\coordideal_2\coordideal_3\coordideal_4)\phi_K^*(\coordideal_3\coordideal_4\coordideal_5),
  \end{equation*}
  and
  \begin{align*}
    &\sum_{\substack{\divisorideal_{67},\divisorideal_{68},\divisorideal_{69}\\\eqref{eq:dij_conditions}}}
    \frac{\mu_K(\divisorideal_{67})\mu_K(\divisorideal_{68})\mu_K(\divisorideal_{69})}{\N(\divisorideal_{67}\divisorideal_{68}\divisorideal_{69}(\divisorideal_{67}\cap\divisorideal_{68}\divisorideal_{69}))}\\
    &=
    \prod_{\substack{\ppp\mid\coordideal_1\\\ppp\nmid\coordideal_2\coordideal_3\coordideal_4\coordideal_5}}\left(1-\frac{1}{\N\ppp^2}\right)\prod_{\substack{\ppp
        \mid \coordideal_2\\\ppp\nmid
        \coordideal_1\coordideal_3\coordideal_4\coordideal_5}}\left(1-\frac{1}{\N\ppp^2}\right)\prod_{\ppp
      \nmid
      \coordideal_1\coordideal_2\coordideal_3\coordideal_4\coordideal_5}\left(1-\frac{3}{\N\ppp^2}+\frac{2}{\N\ppp^3}\right).
  \end{align*}
  A simple comparison of Euler factors proves the lemma.
\end{proof}

\begin{lemma}\label{lem:first_summation}
  Let $\epsilon > 0$. Then, for $B \geq 3$,
  \begin{align*}
    N_{U,H}(B) &= \frac{2^{r_1}(2\pi)^{r_2} h_K
      R_K}{3\cdot|\mu_K|\cdot|\Delta_K|^{3/2}}\left(\prod_{v\in\archplaces}\omega_v(\tS)\right)\sum_{\substack{\coordidealtuple'\in\idealsk^5\\\eqref{eq:additional_height_conditions_ideals}}}\frac{\theta_1(\coordidealtuple')B}{\N(\coordideal_1\cdots\coordideal_5)}
    \\ &+ O_\epsilon(B(\log B)^{5-1/d+\epsilon}),
  \end{align*}
  where the sum runs over all $5$-tuples of ideals $\coordidealtuple'
  =
  (\coordideal_1,\ldots,\coordideal_5)
  \in \idealsk^5$ satisfying
  \begin{equation}\label{eq:additional_height_conditions_ideals}
    \N(\coordideal_3^2\coordideal_4^4\coordideal_5^6) \leq \N(\coordideal_1\coordideal_2)B \quad\text{ and }\quad \N(\coordideal_1^2\coordideal_2^2\coordideal_3^2\coordideal_4)\leq B.
  \end{equation}
\end{lemma}

\begin{proof}
  By Lemma \ref{lem:A3+A1_passage_to_torsor}, Lemma
  \ref{lem:additional_height_condition}, Lemma
  \ref{lem:barroero_widmer_estimate}, Lemma
  \ref{lem:sum_vol_projections}, Lemma \ref{lem:B_vs_SFB}, and Lemma
  \ref{lem:vol_SFB}, the quantity $N_{U,H}(B)$ is
  \begin{equation}\label{eq:NUH_after_first_summation}
    \begin{aligned}
      &\frac{2^{r_1}(2\pi)^{r_2}
        R_K}{3\cdot|\mu_K|\cdot|\Delta_K|^{3/2}}\left(\prod_{v\in\archplaces}\omega_v(\tS)\right)\sum_{\classtuple\in\classrepsyst^6}\sum_{\substack{\coordtuple'
          \in
          \FDei^5\cap\coordfracidealproduct\\\eqref{eq:additional_height_condition},\eqref{eq:height_2}}}\hspace{-0.3cm}\theta_0(\coordidealtuple')\hspace{-0.2cm}\sum_{\substack{\divisoridealtuple\\\eqref{eq:dij_conditions},\eqref{eq:di_conditions}}}\hspace{-0.2cm}\frac{\mu_K(\divisoridealtuple)
        u_\classtuple B}{\N(\latticefracideal_6\latticefracideal_7\latticefracideal_8)N(\coord_2\coord_3\coord_4\coord_5)}\\
      &+O_\epsilon(B(\log B)^{5-1/d+\epsilon}).
    \end{aligned}
  \end{equation}
  From the definitions of the $\latticefracideal_j,
  \coordfracideal_j$, we see that
  \begin{equation*}
    \frac{\mu_K(\divisoridealtuple)u_\classtuple B}{\N(\latticefracideal_6\latticefracideal_7\latticefracideal_8)N(\coord_2\coord_3\coord_4\coord_5)} =\frac{B}{\N(\coordideal_1\cdots\coordideal_5)}\cdot \frac{\mu_K(\divisoridealtuple)}{\N(\divisorideal_{6}\divisorideal_7\divisorideal_8\divisorideal_{67}\divisorideal_{68}\divisorideal_{69}(\divisorideal_{67}\cap\divisorideal_{68}\divisorideal_{69}))}
  \end{equation*}
  We evaluate the sum over $\divisoridealtuple$ by Lemma
  \ref{lem:arith_function}. Moreover, we observe that the
  $\coordfracideal_j$, $j \in \{1,\ldots,5\}$, are independent from
  $\classrep_0$. Hence, the main term in
  \eqref{eq:NUH_after_first_summation} is
  \begin{equation*}
    \frac{2^{r_1}(2\pi)^{r_2} h_K R_K}{3\cdot|\mu_K|\cdot|\Delta_K|^{3/2}}\left(\prod_{v\in\archplaces}\omega_v(\tS)\right)\sum_{\substack{(\classrep_1, \ldots, \classrep_5)\\ \in\classrepsyst^5}} \sum_{\substack{\coordtuple'\in \FDei^5\cap\coordfracidealproduct\\\eqref{eq:additional_height_condition},\eqref{eq:height_2}}}\frac{\theta_1(\coordidealtuple')B}{\N(\coordideal_1\cdots\coordideal_5)}
    .
  \end{equation*}
  When $(\classrep_1,\ldots, \classrep_5)$ runs through all of
  $\classrepsyst^5$ then
  $([\coordfracideal_1],\ldots,[\coordfracideal_5])$
  runs through all $5$-tuples of ideal classes. When $\coord_j$ runs
  through $\FDei \cap \coordfracideal_{j*}$ then
  $\coordideal_j=\coord_j\coordfracideal_{j}^{-1}$ runs through all
  integral ideals in the class $[\coordfracideal_j]$. Furthermore, the
  $\coordideal_j$ satisfy
  \eqref{eq:additional_height_conditions_ideals} if and only if the
  $\coord_j$ satisfy \eqref{eq:additional_height_condition} and
  \eqref{eq:height_2}.
\end{proof}

\section{The remaining summations}\label{sec:completion}
All that remains to be done is the evaluation of the sum over
$\coordidealtuple'$ in Lemma \ref{lem:first_summation}. We proceed as
in the proof of \cite[Proposition 7.3]{arXiv:1302.6151}, except that
we start at $r=5$ instead of $r+1$. Using \cite[Proposition
7.2]{arXiv:1302.6151} inductively, we see that
\begin{equation*}
  \sum_{\substack{\coordidealtuple'\in\idealsk^5\\\eqref{eq:additional_height_conditions_ideals}}}\frac{\theta_1(\coordidealtuple')B}{\N(\coordideal_1\cdots\coordideal_5)} = \left(\frac{2^{r_1}(2\pi)^{r_2}R_Kh_K}{|\mu_K|\sqrt{|\Delta_K|}}\right)^5 \theta_0 V_0(B)+O(B(\log B)^4\log\log B),
\end{equation*}
where
\begin{equation*}
  \theta_0 := \mathcal{A}(\theta_1(\coordidealtuple'), \coordideal_5,\ldots,\coordideal_1)
\end{equation*}
is the iterated ``mean value'' of $\theta_1(\coordidealtuple)$ as defined
in \cite[Section 2]{arXiv:1302.6151}, and
\begin{equation*}
  V_0(B) := \int_{\substack{t_1, \ldots,t_5\geq 1\\t_1^2t_2^2t_3^2t_4\leq B\\t_3^2t_4^4t_5^6\leq t_1t_2 B}}\frac{B}{t_1\cdots t_5}\dd t_1\cdots\dd t_5.
\end{equation*}
By \cite[Lemma 2.8]{arXiv:1302.6151}, we compute
\begin{equation*}
  \theta_0 = \prod_{\ppp}\left(1-\frac{1}{\N\ppp}\right)^6\left(1+\frac{6}{\N\ppp}+\frac{1}{\N\ppp^2}\right).
\end{equation*}
Let $\alpha(\tS)$ be the constant defined, for example, in \cite[D\'efinition 4.8]{MR2019019}. We evaluate $V_0(B)$ in terms of $\alpha(\tS)$. The negative curves $[E_1], \ldots, [E_7]$ generate the effective cone of $\tS$, and $[-K_\tS]=[2E_1+2E_2+2E_3+E_4+3E_6]$, $[E_7] = [E_1+E_2-E_4-2E_5+E_6]$. As in the proof of \cite[Lemma
8.1]{arXiv:1302.6151}, we see that
\begin{equation*}
   \alpha(\tS) = \frac{1}{3}\int_{\substack{x_1,\ldots,x_5 \geq 0\\2x_1+2x_2+2x_3+x_4\leq 1\\-x_1-x_2+2x_3+4x_4+6x_5\leq 1}}\dd x_1\cdots\dd x_5.
\end{equation*}
We substitute $t_i = B^{x_i}$ to obtain
\begin{equation*}
  3\alpha(\tS)\cdot B(\log B)^5 = V_0(B).
\end{equation*}
Let us compute the numerical value of $\alpha(\tS)$. By \cite[Theorem~1.3]{MR2377367}, we have
\begin{equation*}
  \alpha(\tS) = \frac{\alpha(S_0)}{|W(R)|} =  \frac{1}{8640},
\end{equation*}
where $S_0$ is a split smooth del Pezzo surface of degree $4$ (with $\alpha(S_0) = 1/180$ by \cite[Theorem~4]{MR2318651}) and $|W(R)|=2\cdot(3+1)!$ is the order of the Weyl group of the root system $\Athree+\Aone$ associated to
the singularities of $S$.

Together with Lemma \ref{lem:first_summation}, this shows the
asymptotic formula in Theorem \ref{maintheorem}, with the constant $c_{S,H}$ described in Subsection \ref{subsec:constant}.

\section{The leading constant}\label{sec:constant}

Let $\gamma : \tS \to S$ be the minimal desingularization defined in
Section \ref{sec:passage}. Neither $S$ nor $\tS$ are Fano, so the
original conjectures of Manin \cite{MR89m:11060, MR1032922} and Peyre
\cite{MR1340296} do not apply. However, \cite[Remarque 2.3.2]{MR1340296} already suggested generalizations of
Manin's conjecture that cover $\tS$, and such a generalization was formulated, for example, by Batyrev and Tschinkel \cite{MR1679843}.
For our purpose, Peyre's variant \cite{MR2019019} is the most
convenient. Indeed, $\tS$ satisfies \cite[Hypoth\`eses
3.3]{MR2019019}, so we can compare our result to the formula
\cite[Formule empirique 5.1]{MR2019019}.

Let $\Ss$ be the smooth locus of $S$, that is, the complement of the
rational points $(0:0:0:0:1)$ and $(0:1:0:0:0)$, and let $U$ be as in
Theorem \ref{maintheorem}. As we have already observed in the proof of
Lemma \ref{lem:A3+A1_passage_to_torsor},
\begin{equation*} N_{U, H}(B) = |\{x \in \gamma^{-1}(U)(K) \mid
  H(\gamma(x)) \leq B\}|.
\end{equation*}

We construct an adelic metric $(\normv{\cdot})_{v\in\places}$ on the
anticanonical line bundle $\omega^{-1}_\tS$ in the sense of
\cite{MR2019019} such that the Arakelov height $(\omega^{-1}_\tS,
(\normv{\cdot})_{v\in\places} )$ induces the height function $H\circ
\gamma$ on $\tS(K)$. We start by relating the canonical sheaves of $S$
and $\tS$ to each other. The surface $S$ is a complete intersection
defined in $\PP_K^4$ by the forms $h_1 := x_0x_3-x_2x_4$, $h_2 :=
x_0x_1+x_1x_3+x_2^2$. Hence, we may define the canonical sheaf of $S$
as
\begin{equation*}
  \omega_{S} := \det(\conorm_{S|\PP_K^4})^\dual \otimes_{\OO_S}\Omega_{\PP_K^4}^4|_S,
\end{equation*}
where $\conorm_{S|\PP_K^4}$ is the conormal sheaf of the immersion
$S\to \PP_K^4$ and $\Omega_{\PP_K^4}^4$ is the sheaf of differentials
of degree $4$ of $\PP_K^4$ over $K$ (see \cite[Definition
6.4.7]{MR1917232}). Then $\omega_{S}$ is invertible and is the
dualizing sheaf of $S$. Moreover, $\omega_S$ is isomorphic to
$\OO_S(-1)$. Let us specify such an isomorphism.

We write $S_{(i)}$ for the affine open subset where $x_i \neq 0$, with
coordinates $x_j^{(i)} := x_j/x_i$ for $j \neq i$. Then $S_{(i)}$ is
defined by the polynomials $h_j^{(i)}(x_0^{(i)}, \ldots,
\widehat{x_i^{(i)}}, \ldots, x_n^{(i)}) := h_j/x_i^2\in K[x_0^{(i)},
\ldots, \widehat{x_i^{(i)}}, \ldots, x_n^{(i)}]$, $j\in\{1,2\}$. Observe that $U = S_{(2)}$.

Since $S$ is a complete intersection, we have an isomorphism
$\OO_S(-2)\oplus \OO_S(-2)\to \conorm_{S|\PP_K^4}$ defined on
$S_{(i)}$ by $(a_1,a_2)\mapsto a_1 h_1 + a_2 h_2$, and hence an
isomorphism $\OO_S(4) \to \det(\conorm_{S|\PP_K^4})^\dual$ defined on
$S_{(i)}$ by $x_i^{4}\mapsto (h_1^{(i)} \wedge
h_2^{(i)})^\dual$. Moreover, we have an isomorphism $\OO_S(-5)\to
\Omega_{\PP_K^4}^4|_S$ given on $S_{(i)}$ by
\begin{equation*}
  x_i^{-5} \mapsto (-1)^{i}\dd x_0^{(i)} \wedge \cdots \wedge \widehat{\dd x_i^{(i)}}\wedge \cdots \wedge \dd x_4^{(i)}.
\end{equation*}
This gives an isomorphism $\OO_S(-1) \to \omega_S$ defined on $S_{(i)}$ by
\begin{equation*}
  \frac{1}{x_i} \mapsto s_i := (-1)^{i}(h_1^{(i)} \wedge h_2^{(i)})^\dual \otimes \dd x_0^{(i)} \wedge \cdots \wedge\widehat{\dd x_i^{(i)}}\wedge \cdots \wedge \dd x_4^{(i)}.
\end{equation*}
On the smooth locus $\Ss$, we can canonically identify $\omega_\Ss$ with
$\Omega_\Ss^2$ by taking the determinant of the conormal sequence
\begin{equation*}
  0 \to \conorm_{\Ss|\PP_K^4} \to \Omega_{\PP_K^4}^1|_{\Ss} \to \Omega_{\Ss}^1 \to 0.
\end{equation*}
For any $k < l \in \{0, \ldots, 4\}\smallsetminus\{i\}$, this leads to the identification 
\begin{equation}\label{eq:si_explicit}
  s_i = (-1)^{i+t}\Delta_{i,k,l}^{-1} \dd x_k^{(i)}\wedge \dd x_l^{(i)}
\end{equation}
in $\omega_{S,\xi} = \omega_{\Ss,\xi} = \Omega^2_{K(S)}$, where $\xi$
is the generic point of $S$, $\Delta_{i,k,l}$ is the determinant of
the Jacobian matrix $(\partial h_j^{(i)}/\partial x^{(i)}_n)_{j,n}$
with the $k$-th and $l$-th columns removed, and $t := k+l$ if $k<i<l$,
and $t:=k+l-1$ otherwise.

Since $\gamma$ is an isomorphism on $\gamma^{-1}(\Ss)$, the pullback
of differential forms gives an isomorphism
$\gamma^{*}\omega_S|_{\gamma^{-1}(\Ss)} \cong
\omega_{\tS}|_{\gamma^{-1}(\Ss)}$. This induces an isomorphism
$\omega_{\tS} \cong \gamma^*\omega_{S}\otimes\OO_\tS(P)$, where $P$
is a divisor supported on the complement $\tS \smallsetminus
\gamma^{-1}(\Ss)$. Since $\tS$ is split and both singularities of $S$
are rational double points, $[P] = 0$ (see \cite[Proposition
8.1.10]{MR2964027}). We conclude that
\begin{equation*}
 \omega_\tS\cong \gamma^*\omega_S,
\end{equation*}
with an isomorphism whose restriction to $\gamma^{-1}(\Ss)$ is given
by the pullback of differential forms. In the following, we use this
isomorphism to identify $\omega_{\tS}$ with $\gamma^*(\omega_S)$ and
its dual to identify $\omega_{\tS}^{-1}$ with
$\gamma^{*}\omega_S^{-1}$.

For every $i \in \{0, \ldots, 4\}$, let $\tau_i$ be the global section
of $\omega_S^{-1}\cong \OO_S(1)$ dual to $s_i$. Then $\tau_0, \ldots,
\tau_4$ define the embedding $S \hookrightarrow \PP_K^4$. The morphism
$\tS \to S \hookrightarrow \PP_K^4$ is given by the sections
$\gamma^*\tau_0, \ldots, \gamma^*\tau_4 \in
H^0(\tS,\omega_\tS^{-1})$. Consider the Arakelov height
$(\omega_\tS^{-1},(\normv{\cdot})_{v\in\places})$ defined by these
global sections: for all $v\in \places$, $x \in \tS(K_v)$, and $\tau
\in \omega_\tS^{-1}(x)$, we use the $v$-adic metric
\begin{equation*}
  \normv{\tau} := \min_{\substack{0\leq i\leq 4\\\gamma^*\tau_i(x)\neq 0}}\left\{\absv{\frac{\tau}{\gamma^*\tau_i(x)}}\right\}.
\end{equation*}
The corresponding height function on $\tS(K)$ (see \cite[D\'efinition
2.3]{MR2019019}) is $H\circ\gamma$, as desired.

According to \cite[Formule empirique 5.1]{MR2019019}, the leading constant in
Theorem \ref{maintheorem} should hence have the form
\begin{equation*}
  c_{S,H} = \alpha(\tS)\beta(\tS)\tau_H(\tS),
\end{equation*}
with $\alpha(\tS),\beta(\tS), \tau_H(\tS)$ as in \cite[D\'efinition
4.8]{MR2019019}.

We have already seen at the end of the the last section how the factor $\alpha(\tS)$ appears in our leading constant.  Moreover,
\begin{equation*}
  \beta(\tS) := |H^1(\Gal(\Qbar/K), \Pic(\tS_\Qbar))|  = 1,
\end{equation*}
since $\tS$ is split.

By \cite[D\'efinition 4.6, D\'efinition 4.8]{MR2019019}, and using the
properties of the model $\tSZ$ proved in Proposition
\ref{model_properties}, the Tamagawa number $\tau_H(\tS)$ is given as
\begin{equation*}
  \tau_H(\tS) = \lim_{s \to 1} (s-1)^6 L(s,\Pic(\tS_\Qbar))\frac{1}{|\Delta_K|}\prod_{v\in\places}\lambda_v^{-1}\bomega_{H,v}(\tS(K_v)),
\end{equation*}
where
\begin{equation*}
 \lambda_v :=
 \begin{cases}
   L_v(1, \Pic(\tS_\Qbar)) &\text{ if } v \in \nonarchplaces,\\
   1 &\text{ if } v \in \archplaces.
 \end{cases}
\end{equation*}
Note that the closure of $\tS(K)$ in the set $\tS(\AAA_K)$ of adelic
points coincides with $\tS(\AAA_K)$, since the rational variety $\tS$
satisfies weak approximation.

By Proposition \ref{model_properties}, \ref{splitness}, the Frobenius morphism associated to
any non-archimedean place $v$ corresponding to a prime ideal $\ppp$
acts trivially on the vector space $\Pic(\tSZ_{\Fbar_\ppp})\otimes
\QQ$ of dimension $6$. Therefore, $L_v(s, \Pic(\tS_\Qbar)) =
(1-\N\ppp^{-s})^{-6}$ and
\begin{equation*}
  L(s, \Pic(\tS_\Qbar)) = \prod_{v \in \nonarchplaces} L_v(s,
  \Pic(\tS_\Qbar)) = \zeta_K(s)^6.
\end{equation*} 
By the analytic class number formula,
\begin{equation*}
  \lim_{s \to 1} (s-1)^6 L(s,\Pic(\tS_\Qbar))\frac{1}{|\Delta_K|} = \left(\frac{2^{r_1}(2\pi)^{r_2}R_K h_K}{|\mu_K|}\right)^6\cdot\frac{1}{|\Delta_K|^4}.
\end{equation*}

Let us compute the $v$-adic measures $\bomega_{H, v}(\tS(K_v))$
defined in \cite[Notations 4.3]{MR2019019}. By Proposition \ref{model_properties}, \ref{smooth_geom_integral}, the base change $\tS_{\OO_v}$ of our model satisfies the hypotheses of \cite[Corollary 2.15]{MR1679841} for all $v\in\nonarchplaces$. If $v$ corresponds to a prime ideal $\ppp$ of $\Ok$, we may thus conclude that
\begin{equation*}
  \bomega_{H,v}(\tS(K_v)) = \frac{|\tS(k(\ppp))|}{\N\ppp^2}.
\end{equation*}
Moreover, Proposition \ref{model_properties}, \ref{splitness} shows that $|\tS(k(\ppp))| = \N\ppp^2 + 6\N\ppp + 1$. Thus, we see that $\lambda_v^{-1}\bomega_{H,v}(\tS(K_v)) = \omega_v(\tS)$, with $\omega_v(\tS)$ as in Subsection \ref{subsec:constant}. It remains to compute $\bomega_{H,v}(\tS(K_v))$ for $v\in\archplaces$.

Since $U=S_{(2)}$ is
smooth, its set of rational points $U(K_v)$ has the structure of a
$K_v$-analytic manifold, and $\gamma : \tS \to S$ induces an analytic
isomorphism $\gamma^{-1}(U)(K_v)\to U(K_v)$, which we again call
$\gamma$.

The preimage $\gamma^{-1}(S\smallsetminus U)$ is the union of the
negative curves on $\tS$. As a union of finitely many
submanifolds of strictly smaller dimension,
$\gamma^{-1}(S\smallsetminus U)(K_v)$ has measure $0$, and
$\bomega_{H,v}(\tS(K_v))=\bomega_{H,v}(\gamma^{-1}(U)(K_v))$

The local coordinates $x_0^{(2)}-1$, $x_3^{(2)}$ at the rational point
$p=(1,-1,0,0)$ of $U$ define an analytic isomorphism $\psi : U(K_v)
\to W := \{(z_0,z_3)\in K_v^2 \mid z_0+z_3+1\neq 0\}$. Since the
restriction of $\gamma$ to $\gamma^{-1}(U)$ is an isomorphism, the
functions $y_0:=(x_0^{(2)}-1)\circ\gamma$,
$y_3:=(x_3^{(2)})\circ\gamma$ form a system of local coordinates at
$\gamma^{-1}(p)$, defining an analytic isomorphism $\phi := \psi \circ
\gamma : \gamma^{-1}(U)(K_v) \to W$.  By definition,
\begin{align*}
  &\bomega_{H,v}(\gamma^{-1}(U)(K_v)) =
  \int_{W}\normv{\left(\frac{\partial}{\partial y_0} \wedge
      \frac{\partial}{\partial y_3}\right)(\phi^{-1}(z_0,z_3))}\dd z_0\dd z_3\\
  &= \int_{W}\normv{\gamma^*\left(\frac{\partial}{\partial x_0^{(2)}}
      \wedge
      \frac{\partial}{\partial x_3^{(2)}}\right)(\phi^{-1}(z_0,z_3))}\dd z_0\dd z_3\\
  &= \int_{W}\min_{\tau_i(\psi^{-1}(z_0,z_3))\neq
    0}\left\{\absv{\frac{\tau_2(\psi^{-1}(z_0,z_3))}{((x^{(2)}_0+x^{(2)}_3)\tau_i)(\psi^{-1}(z_0,z_3))}}\right\}\dd
  z_0\dd z_3,
\end{align*}
since $\partial/(\partial x_0^{(2)})\wedge \partial/(\partial
x_3^{(2)}) = -\Delta_{2,0,3}^{-1}\tau_2 =
-(x^{(2)}_0+x^{(2)}_3)^{-1}\tau_2$ due to
\eqref{eq:si_explicit}. Together with the relation $\tau_i =
x_i^{(2)}\cdot \tau_2$, this shows that $\bomega_{H,v}(\tS(K_v))$ has
the explicit form
\begin{equation*}
 \int_{K_v^2}\frac{\dd z_0 \dd
    z_3}{\max\{1,\absv{z_0+z_3},\absv{z_0(z_0+z_3)}, \absv{z_3(z_0+z_3)}, \absv{z_0z_3(z_0+z_3)}\}}.
\end{equation*}

We assume now that $v$ is a complex place. We transform to
variables $t_0 = -z_0$, $t_1 = z_0+z_1$ and use the identity
\begin{equation*}
  \frac{1}{s}=\int_{t\geq s}\frac{1}{t^2}\dd t,
\end{equation*}
for $s \in \RR\smallsetminus\{0\}$ to obtain
\begin{equation*}
  \bomega_{H,v}(\tS(K_v)) = 4 \int_{\substack{(t_0, t_3)\in
      \CC^2\\t_2\geq\max\{1, |t_1|^2, |t_0t_1|^2, |t_1(t_0+t_1)|^2,
      |t_0t_1(t_0+t_1)|^2\}}}\frac{\dd t_0 \dd t_1 \dd t_2}{t_2^2}.
\end{equation*}
Recall that Peyre normalizes the Haar measure on $K_v$ to be twice the usual
Lebesgue measure on $\CC \cong \RR^2$, which leads to the factor $4$ in front
of the integral.

We apply the transformation $t_2 = 1/y_2^3$, $t_0 = y_0/\sqrt{y_2}$,
$t_1 = y_1/\sqrt{y_2}$ of Jacobian determinant $-3/y_2^6$ and replace
$y_2$ by a complex variable via polar coordinates to see that
$\bomega_{H,v}(\tS(K_v)) = \omega_v(\tS)$. By a similar argument, the
same equality holds for real places $v$.

This shows that the constant $c_{S,H}$ in Theorem \ref{maintheorem} is as
expected.

\begin{ack}
  The first-named author was supported by a Humboldt Research
  Fellowship for Postdoctoral Researchers of the Alexander von
  Humboldt Foundation. The second-named author was supported by grant
  DE 1646/3-1 of the Deutsche Forschungsgemeinschaft. This work was
  partly completed while the first-named author enjoyed the
  hospitality of the Erwin Schr\"odinger Institute in Vienna.

  We thank the anonymous referee for pointing out references to us and for helpful comments that simplified arguments in Sections \ref{sec:passage} and \ref{sec:constant}.
\end{ack}

\bibliographystyle{alpha}

\bibliography{manin_dp4_a3a1}

\end{document}